\documentclass[platex,11pt]{article}
\usepackage{amssymb}
\usepackage{fancybox}
\usepackage{mathrsfs}
\usepackage{amsrefs}
\usepackage{mathtools}
\usepackage{yhmath}

\usepackage{amscd}
\usepackage{amsmath}
\usepackage{amsthm}

\usepackage{comment}
\usepackage{xypic}
\usepackage[margin=3cm]{geometry}
\usepackage[dvipdfmx]{graphicx}
\usepackage{eucal}
\usepackage{bm}

\usepackage{color}
\usepackage{accents}

 \usepackage{scalerel,stackengine}

\theoremstyle{plain}
\newtheorem{thm}{Theorem}[section]
\newtheorem{prop}[thm]{Proposition}
\newtheorem{lem}[thm]{Lemma}
\newtheorem{cor}[thm]{Corollary}

\theoremstyle{definition}
\newtheorem{defi}[thm]{Definition}

\newtheorem{rem}[thm]{Remark}
\newtheorem{ex}[thm]{Example}

\newtheorem{notation}[thm]{Notation}
\newcommand{\ZZ}{\mathbb{Z}}

\newcommand{\NN}{\mathbb{N}}
\newcommand{\QQ}{\mathbb{Q}}
\newcommand{\RR}{\mathbb{R}}
\newcommand{\CC}{\mathbb{C}}
\newcommand{\PP}{\mathbb{P}}
\newcommand{\KK}{\mathbb{K}}
\newcommand{\CS}{\mathbb{C}^{*}}

\newcommand{\GR}{\mathrm{gr}}

\newcommand{\wt}[1]{\widetilde{#1}}

\newcommand{\wtkai}[1]{{#1}\hspace{-0.4cm}\raisebox{0.1cm}{$\widetilde{\quad}$}}

\newcommand{\ov}[1]{\overline{#1}}

\newcommand{\simto}{\overset{\sim}{\to}}

\newcommand{\DBkai}{\mathrm{{D}^{b}}}

\newcommand{\calM}{\mathcal{M}}

\newcommand{\calE}{\mathscr{E}}

\newcommand{\calN}{\mathcal{N}}

\newcommand{\eu}{\mathcal{E}}
\newcommand{\scrM}{\mathscr{M}}
\newcommand{\scrN}{\mathscr{N}}

\newcommand{\scrT}{\mathscr{T}}

\newcommand{\scrA}{\mathscr{A}}

\newcommand{\Ker}{\mathrm{Ker}}

\newcommand{\blal}{\bm{\alpha}}
\newcommand{\blbt}{\bm{\beta}}

\newcommand{\ble}{\bm{e}}

\newcommand{\RES}[1]{{}^\tau\hspace{-1mm}{#1}}
\newcommand{\TO}[1]{{}^T\hspace{-1mm}{#1}}

\newcommand{\ti}{z_{i}}
\newcommand{\Vto}{V_{z_{1}}}

\newcommand{\Mi}{M_{1}}
\newcommand{\pa}{\partial}

\newcommand{\bld}[1]{\boldsymbol{#1}}

\newcommand{\RESB}{{}^\tau}

\stackMath
\newcommand\reallywidehat[1]{%
\savestack{\tmpbox}{\stretchto{%
  \scaleto{%
    \scalerel*[\widthof{\ensuremath{#1}}]{\kern-.6pt\bigwedge\kern-.6pt}%
    {\rule[-\textheight/2]{1ex}{\textheight}}
  }{\textheight}%
}{0.5ex}}%
\stackon[1pt]{#1}{\tmpbox}%
}
\newcommand{\whmu}[1]{(#1)^{\wedge}}
\newcommand{\whmunk}[1]{{#1}^{\wedge}}

\newcommand{\lam}{\lambda}

\newcommand{\integ}{\mathrm{int}}

\newcommand{\irrF}{F^{\mathrm{irr}}}

\newcommand{\Cnt}{\CC^n_{z}}
\newcommand{\Pnt}{\PP^n_{z}}
\newcommand{\Cnx}{\CC^n_{\zeta}}
\newcommand{\Pnx}{\PP^n_{\zeta}}

\newcommand{\UV}{U^{\vee}}

\newcommand{\xizp}{{\zeta_{0}'}}

\newcommand{\zetn}{\zeta_{n}}

\newcommand{\expo}{\mathscr{E}}

\makeatletter
\newcommand{\subjclass}[2][2010]{%
  \let\@oldtitle\@title%
  \gdef\@title{\@oldtitle\footnotetext{#1 \emph{Mathematics subject classification.} #2}}%
}
\newcommand{\keywords}[1]{%
  \let\@@oldtitle\@title%
  \gdef\@title{\@@oldtitle\footnotetext{\emph{Key words and phrases.} #1.}}%
}
\makeatother

\newcommand{\Addresses}{{
  \bigskip
  \footnotesize

  Takahiro Saito, \textsc{Research Institute for Mathematical Sciences, Kyoto University, Kyoto 606-8502, Japan}\par\nopagebreak
  \textit{E-mail address} : \texttt{takahiro@kurims.kyoto-u.ac.jp}

}}

\title{The Hodge filtration of a monodromic mixed Hodge module and the irregular Hodge filtration}
 \date{}
\author{Takahiro Saito}

\keywords{D-module, Perverse sheaf, Mixed Hodge module, Mixed twistor D-module, Irregular Hodge filtration\\
\quad \quad This article will appear at Annales de l'Institut Fourier}

\begin{document}
\maketitle

\begin{abstract}
For an algebraic vector bundle $E$ over a smooth algebraic variety $X$,
a monodromic $D$-module on $E$ is decomposed into a direct sum of some $O$-modules on $X$.
We show that the Hodge filtration of a monodromic mixed Hodge module is decomposed with respect to the decomposition of the underlying $D$-module.
By using this result, we endow the Fourier-Laplace transform $\whmunk{M}$ of the underlying $D$-module $M$ of a monodromic mixed Hodge module with a mixed Hodge module structure.
Moreover, we describe the irregular Hodge filtration on $\whmunk{M}$ concretely and 
show that it coincides with the Hodge filtration at all integer indices.
\end{abstract}
\tableofcontents

\section{{I}ntroduction}
The present paper is a continuation of our previous paper~\cite{SaiMon}.

In this paper, we only deal with algebraic objects.
Let $E$ be an algebraic vector bundle over a smooth algebraic variety $X$ of finite type over $\CC$ and $\pi\colon E\to X$ be the projection.
Since $\pi$ is affine morphism, we will identify an algebraic $D$-module $M$ on $E$ with an algebraic $\pi_{*}D$-module $\pi_{*}M$ (see Lemma~\ref{haruni}).
We denote by $\eu_{E}$ the Euler vector field on $E$.
For a trivialization $\pi^{-1}(U)\simeq U\times \CC^n$ ($U\subset X$) and coordinates $(z_{1},\dots ,z_{n})$ of $\CC^n$,
$\eu_{E}$ is $\sum_{i=1}^{n}z_{i}\pa_{z_{i}}$.
Then, if an algebraic $D$-module $M$ on $E$ is monodromic (see Definition~\ref{mondef}),
we have a decomposition
\begin{align}\label{sirusi}
M=\bigoplus_{\beta\in \RR}M^{\beta},
\end{align}
where $M^{\beta}=\bigcup_{l\geq 0}\Ker((\eu_{E}-\beta)^l\colon \pi_{*}M\to \pi_{*}M)$ (see Proposition~\ref{decomprop}).
A monodromic mixed Hodge module is a mixed Hodge module on $E$ whose underlying $D$-module is monodromic.
Then, our first main result is the following.
\begin{thm}[Theorem~\ref{Fdecomp}]\label{owaranai}
Let $\calM$ be a monodromic mixed Hodge module on $E$ and $M$ the underlying $D$-module.
Then, the Hodge filtration $\{F_{p}M\}_{p\in \ZZ}$ of $M$ is decomposed with respect to the decomposition~(\ref{sirusi}):
\begin{align}\label{yasasiku}
F_{p}M=\bigoplus_{\beta\in \RR}F_{p}M^{\beta},
\end{align}
where $F_{p}M^{\beta}:=F_{p}M\cap M^{\beta}$.
\end{thm}

\begin{rem}
Theorem~\ref{owaranai} was shown in a different way in a recent preprint \cite{ChenDirks} by Chen-Dirks.
\end{rem}

When the rank of $E$ is one, this result was already shown in \cite{SaiMon}.
We will prove it by using the fact that the pull-back of a monodromic $D$-module on $\CC^n$ by the blowing up morphism $\wt{\CC^n}\to \CC^n$ is a monodromic $D$-module on a line bundle (Lemma~\ref{rankoneFdecom}) and some results for monodromic mixed Hodge modules on line bundles in \cite{SaiMon}.

We consider the Fourier-Laplace transform $\whmunk{M}$ of a $D$-module $M$ on $E$, which is a $D$-module on the dual vector bundle $\pi^{\vee}\colon E^{\vee}\to X$.
When $E$ is trivial: $E\simeq X\times \CC^n$ (then, $E^{\vee}$ is also trivial: $E\simeq X\times {\check{\CC}}^n$) and $X$ is affine, we can identify $D$-modules on $E$ (resp. $E^{\vee}$) with the $\Gamma(E;D)$-modules (resp. $\Gamma(E^{\vee};D)$) of their global sections.
In this case,
$\whmunk{M}$ is $M$ as a set and its $\Gamma(E^{\vee};D)$-module structure is defined so that for $P\in \Gamma(X;D)$ and $1\leq i\leq n$ we have
\begin{align*}
    P\cdot \whmunk{m}=&\whmu{Pm},\\
    \zeta_{i}\cdot\whmunk{m}=&\whmu{\pa_{z_{i}}m},\quad \mbox{and}\\
    \pa_{\zeta_{i}}\cdot\whmunk{m}=&-\whmu{z_{i}m},
\end{align*}
where 
$(z_{1},\dots, z_{n})$ (resp. $(\zeta_{1},\dots,\zeta_{n})$) is the coordinate system of $\CC^n$ (resp. the dual ${\check{{\CC}}}^n$) and $\whmunk{m}$ is a global section in $\whmunk{M}$ corresponding to a global section $m\in M$ (see Lemma~\ref{march}).
Furthermore, in this case, for a $\Gamma(X;O)$-submodule $F\subset M$, we define a $\Gamma(X;O)$-submodule $\whmunk{F}\subset \whmunk{M}$ as
\[\whmunk{F}:=\{\whmunk{m}\in \whmunk{M}\ |\ m\in F\}.\]
Even in a general case (not necessary $E$ is trivial), we can define an $O_{X}$-submodule $\whmunk{F}$ of $\pi^{\vee}_{*}\whmunk{M}$ for an $O_{X}$-module $F$ of $\pi_{*}M$ (Definition~\ref{tetuko}).
Recall that the underlying $D$-module of a mixed Hodge module is regular holonomic.
In general, even if $M$ is regular, $\whmunk{M}$ may not be regular.
Therefore, even if $M$ is the underlying $D$-module of a mixed Hodge module,
$\whmunk{M}$ may not be so.
Nevertheless, it is known that when a $D$-module $M$ is monodromic and regular, so is $\whmunk{M}$ (see Lemma~\ref{tobuneko}).
Therefore, for the underlying $D$-module $M$ of a monodromic mixed Hodge module, $\whmunk{M}$ may be equipped with a mixed Hodge module structure.
Reichelt~\cite{ReiLau} gave definitions of mixed Hodge module structures on the homogeneous $A$-hypergeometric $D$-modules, which are expressed as the Fourier-Laplace transform of certain monodromic $D$-modules.
Moreover, Reichelt-Walther~\cite{RWWGKZ} defined a mixed Hodge module whose underlying $D$-module is the Fourier-Laplace transform of the underlying $D$-module of a monodromic mixed Hodge module.
By a different method from theirs, using Theorem~\ref{owaranai} (for a line bundle), we show the following.
\begin{prop}[Definition~\ref{ganbaruzo}]\label{blue}
For a monodromic mixed Hodge module $\calM$ on $E$,
we can naturally define a mixed Hodge module $\whmunk{\calM}$ on $E^{\vee}$ whose underlying $D$-module is $\whmunk{M}$.
\end{prop}
When $E$ is a line bundle, this result was already proved in \cite{SaiMon}.
To show it in the general case, we will describe $\whmunk{M}$ in terms of the Fourier-Laplace transform on a line bundle (Lemma~\ref{scaret}) and use the results in \cite{SaiMon}.
We can describe the Hodge filtration of $\whmunk{M}$ concretely below (Corollary~\ref{kitaku}).

Next, we consider the irregular Hodge filtrations.
For a holomorphic function $f$ on $E$,
the exponential $D$-module $\calE^{f}$ (
i.e. the structure sheaf $O_{E}$ with the connection $g\mapsto dg+df\cdot g$) is not regular in general.
Since the underlying $D$-module of a mixed Hodge module is regular,
we can not apply the theory of mixed Hodge module to endow it with a natural Hodge filtration.
Nevertheless, for the underlying $D$-module $M$ of a mixed Hodge module $\calM$ on $E$,
Esnault-Sabbah-Yu~\cite{ESY} and Sabbah-Yu~\cite{SabYuIrrFil} defined a natural filtration $\irrF_{\alpha+\bullet}(M\otimes \calE^{f})$ of the exponentially twisted $D$-module $M\otimes \calE^{f}$ for $\alpha\in [0,1)$, called the irregular Hodge filtration.
Note that combining $\irrF_{\alpha+\bullet}(M\otimes \calE^{f})$ for all $\alpha\in [0,1)$, we can consider the filtration $\{\irrF_{\gamma}(M\otimes \calE^{f})\}_{\gamma\in \RR}$ indexed by $\RR$, not only $\ZZ$.
These constructions were generalized as the ``irregular Hodge theory'' in \cite{IrrHodge} and 
the category of irregular Hodge modules was established, which contains mixed Hodge modules and ``exponentially twisted mixed Hodge modules'',
as a full subcategory of the category of integrable mixed twistor $D$-modules introduced by Mochizuki~\cite{MTM}.
By using the pushforward and pullback functors between the category of irregular Hodge modules (see Proposition~\ref{freeway}),
we obtain the ``Fourier-Laplace transform of a mixed Hodge module'' in the category of irregular Hodge modules (\ref{hyakusyo}),
and we thus obtain natural filtrations (also called the irregular Hodge filtrations) on the Fourier-Laplace transforms and their stalks: the twisted de Rham cohomologies, so that they are generalizations of the ones defined in Deligne~\cite{DelIrr}, Yu~\cite{YuIrr} and Sabbah~\cite{SabFoutwo}.
However,
these filtrations are not the Hodge filtration of a mixed Hodge module in general.
In general, it is difficult to compute the irregular Hodge filtrations concretely.
However, in our situation, due to Theorem~\ref{owaranai} we can describe the irregular Hodge filtration (for $\alpha\in [0,1)$) $\{\irrF_{\alpha+p}\whmunk{M}\}_{p\in \ZZ}$ of $\whmunk{M}$ in terms of the original Hodge filtration $F_{\bullet}M$ as follows.
\begin{thm}[Thoerem~\ref{prosyuto}]\label{irrgutai}
For $\alpha\in [0,1)$ and $p\in \ZZ$,
we have
\[\irrF_{\alpha+p}\whmunk{M}=\bigoplus_{\beta\in \RR}\whmu{F_{p+\lfloor\alpha-\beta\rfloor}M^{\beta}},\]
where $\lfloor\alpha-\beta\rfloor$ is the largest integer less than or equal to $\alpha-\beta$.
\end{thm}
We will prove it by describing the rescaled module (Subsection~\ref{subirrH}) and its Kashiwara-Malgrange filtration explicitly (Lemma~\ref{golgo}).

Eventually, on $\whmunk{M}$, we have two Hodge filtrations: the first one is the Hodge filtration $\{F_{p}\whmunk{M}\}_{p\in \ZZ}$ defined in Proposition~\ref{blue} and the second one is the irregular Hodge filtration $\{\irrF_{\alpha+p}\whmunk{M}\}_{p\in \ZZ}$ (for $\alpha\in [0,1)$).
In fact, these filtrations are equal at all integer indices.

\begin{thm}[Theorem~\ref{takarajima}]\label{konokimochiha}
For $p\in \ZZ$, we have
\[\irrF_{p}\whmunk{M}=F_{p}\whmunk{M}.\]
\end{thm}

In particular, we can say that ``the irregular Hodge filtration indexed by $\ZZ$ of the Fourier-Laplace transform of the monodromic mixed Hodge module is in fact the Hodge filtration''.
Remark that the irregular Hodge filtration $\{\irrF_{\gamma}\whmunk{M}\}_{\gamma\in \RR}$ jumps at also non-integer indices in general and thus the filtration indexed by $\RR$ has more information.
For example, it plays an important role when considering the tensor products (see Proposition~3.18 of \cite{IrrHodge})

Combining this result with Theorem~\ref{irrgutai},
the Hodge filtration (defined by Proposition~\ref{blue})
can be described as follows.
\begin{cor}[Corollary~\ref{nenmatu}]\label{kitaku}
For $p\in \ZZ$, we have
\[F_{p}\whmunk{M}=\bigoplus_{\beta\in \RR}\whmu{F_{p+\lfloor-\beta\rfloor}M^{\beta}}.\]
\end{cor}

Finally, we consider the irregular Hodge filtration ``at infinity''.
Let $\wt{E}$ (resp. $\wt{E^{\vee}}$) be the projective compactification of $E$ (resp. $E^{\vee}$).
We denote by $j\colon E\hookrightarrow \wt{E}$ and $j^{\vee}\colon E^{\vee}\hookrightarrow \wt{E^{\vee}}$ the inclusions and $D_{\infty}$ (resp. $D_{\infty}^{\vee}$) the divisor $\wt{E}\setminus E$ (resp. $\wt{E^{\vee}}\setminus E^{\vee}$).
For a $D$-module $M$ on $E$ let $N$ be the pushforward $j_{*}M$ of $M$ by $j$, which is a $D$-module on $\wt{E}$ (recall that our $D$-modules are algebraic).
Then, we can consider the Fourier-Laplace transform $\whmunk{N}$ of $N$, which is a $D$-module on $\wt{E^{\vee}}$.
We will see that this is equal to $j^{\vee}_{*}\whmunk{M}$ (Lemma~\ref{train}).
For a mixed Hodge module $\calM$ on $E$ and its extension $\calN$ to $\wt{E}$ whose underlying $D$-module is $N$,
we can also consider the irregular Hodge filtration on $\whmunk{N}(=j_{*}^{\vee}\whmunk{M})$.
Then, we can compute explicitly the irregular Hodge filtration on $\whmunk{N}$ (Theorem~\ref{dengaku}) and get the following result.

\begin{cor}[Corollary~\ref{sincyan}]\label{dekireba}
The irregular Hodge filtration $\{F_{\alpha+p}\whmunk{N}\}_{p\in \ZZ}$ for $\alpha\in [0,1)$ has the strict specializability property along $D_{\infty}^{\vee}$.
\end{cor}

For the definition of the strict specializability,
see Definition~\ref{tron}.
Remark that this result is shown in a more general setting in a recent preprint by Mochizuki~\cite{MochiResc}.

By the definition of $\calN$, we have $\calN=\calN[*D_{\infty}]$ (see Lemma~\ref{hiphop}).
For the definition of ``$[*D_{\infty}]$'', see Proposition~\ref{localizfac}.
In particular, the Hodge filtration of $\calN$ is described in terms of $F_{\bullet}V^{0}_{D_{\infty}}N$,
where $V^{\bullet}_{D_{\infty}}N$ is the Kashiwara-Malgrange filtration of $N$ along $D_{\infty}$. 
For the $D$-module with a filtration $(\whmunk{N},F_{\bullet}\whmunk{N})$,
we denote by 
$(\whmunk{N},F_{\bullet}\whmunk{N})[*D_{\infty}^{\vee}]$ the $D$-module $\whmunk{N}[*D_{\infty}^{\vee}]:=\whmunk{N}(*D_{\infty}^{\vee})(=\whmunk{N})$ with the filtration 
$F_{\bullet}(\whmunk{N}[*D_{\infty}^{\vee}])$
defined by the same formula as usual ``$[*{D_{\infty}^{\vee}}]$'' for the localization of a mixed Hodge module.
Then, we have the following.
\begin{cor}[Corollary~\ref{nonnonbaa}]\label{aozora}
We have
\[(\whmunk{N},\irrF_{\bullet}\whmunk{N})=
(\whmunk{N},\irrF_{\bullet}\whmunk{N})[*D_{\infty}^{\vee}].\]
\end{cor}

By Corollaries~\ref{dekireba} and \ref{aozora}, we can say that ``the irregular Hodge filtration of the Fourier-Laplace transform of a monodromic mixed Hodge module has the same properties as the usual Hodge filtrations''.
To be more precise, we can conclude as follows.
Let $\wt{\whmunk{\calM}}$ be the extension of the mixed Hodge module $\whmunk{\calM}$ on $E^{\vee}$ to $\wt{E^{\vee}}$ such that $\wt{\whmunk{\calM}}=\wt{\whmunk{\calM}}[*D_{\infty}^{\vee}]$,
whose underlying $D$-module is denoted by $\wt{\whmunk{M}}$
(in fact, there exists such an extension by the definition of ``algebraic'' mixed Hodge module in \cite{MHM}). 
Note that we have $\wt{\whmunk{M}}=\whmunk{N}$.
Then, by Theorem~\ref{konokimochiha} and Corollary~\ref{aozora} we have the following.
\begin{cor}[Corollary~\ref{rattlechain}]
The irregular Hodge filtration $\{\irrF_{p}\whmunk{N}\}_{p\in \ZZ}$ indexed by integers is the Hodge filtration of the mixed Hodge module $\wt{\whmunk{\calM}}$.
\end{cor}

\subsection*{Acknowledgments}
I would like to express my sincere gratitude to Professor Takuro Mochizuki for helpful discussions, valuable suggestions, and explaining some of his results on mixed twistor D-modules used in Section~4.
I would like to deeply appreciate Professor Claude Sabbah's help including answering many questions, reading the first version of the manuscript, and suggesting improvements to it. 
This work and the previous one \cite{SaiMon} have been inspired by his suggestion in 2018.
I am also grateful to Tatsuki Kuwagaki and Yota Shamoto for fruitful conversations.
I thank Yuichi Ike for answering my questions about constructible sheaves.
My thanks also go to Genki Sato for his help with English writing. 
I thank Alberto Casta\~no Dom\'inguez for helpful comments on the Fourier-Laplace transforms of $R$-modules.
I also thank the referee for useful comments.
This work is supported by JSPS KAKENHI Grant Number 20J00922.

\section{{M}onodromic mixed Hodge modules}
\subsection{{M}onodromic mixed Hodge modules on vector bundles}
In this subsection, we recall the notion of monodromic $D$-module on a vector bundle and some basic facts.
We refer to \cite{BryFou} and \cite{SaiMon}.
Let $X$ be a smooth algebraic variety of finite type over $\CC$,
$O_{X}$ the structure sheaf on $X$ and $D_{X}$ the sheaf of differential operators on $X$.
Basically, we consider only algebraic left $D$-modules in this paper.
Moreover, all the $D$-modules and $O$-modules are quasi-coherent.

Let $\pi\colon E\to X$ be an algebraic vector bundle and $M$ a (quasi-coherent) $D$-module (resp. $O$-module) on $E$.
Then, $\pi_{*}M$ is a $\pi_{*}D$-module (resp. $\pi_{*}O$-module).
Conversely, for a $\pi_{*}D$-module $N$ (resp. $\pi_{*}O$-modules) on $X$,
we define a $D_{E}$-module (resp. $O_{E}$-modules) $\pi^{*}N$ as
\begin{align*}
\pi^*N:=D_{E}\otimes_{\pi^{-1}\pi_{*}D_{X}}\pi^{-1}N.\\
(\mbox{resp.\ } \pi^*N:=O_{E}\otimes_{\pi^{-1}\pi_{*}O_{X}}\pi^{-1}N)
\end{align*}
Because $\pi$ is an affine morphism (i.e. the inverse image of an affine open subset of $X$ is affine),
we have the following equivalence.

\begin{lem}[see Proposition~7.10 of Brylinski~\cite{BryFou}]\label{haruni}
The functors $\pi_{*}$ and $\pi^*$ define an equivalence of categories between the category of quasi-coherent $D_{E}$-modules (resp. $O_{E}$-modules) and that of quasi-coherent $\pi_{*}D_{E}$-modules (resp. $\pi_{*}O_{E}$-modules).
The same assertion holds for the categories of coherent $D$ or $O$-modules.
\end{lem}

In the following, we identify $D_{E}$-modules (resp. $O_{E}$-modules) with $\pi_{*}D_{E}$-modules (resp. $\pi_{*}D_{E}$-modules).

\begin{rem}
When $X$ is the one point set, $E$ is just a vector space $\CC^n$.
We sometimes consider a $D$-module on an affine subset of $\CC^n$, such as $(\CC^*)^n$, where $\CC^*:=\CC\setminus \{0\}$.
In this case, we have the equivalence similar to Lemma~\ref{haruni}.
For example, we identify $D_{(\CC^*)^n}$-modules with 
$\CC[z_{1}^{\pm 1},\dots,z_{n}^{\pm 1}]\langle\partial_{z_{1}},\dots \partial_{z_{n}}\rangle$-modules,
where $\CC[z_{1}^{\pm 1},\dots,z_{n}^{\pm 1}]\langle\partial_{z_{1}},\dots \partial_{z_{n}}\rangle$ is a ring of differential operators on $(\CC^*)^n$.
\end{rem}

Let $V$ be a vector space of rank $n$ and $e_{1},\dots, e_{n}$ a basis of $V$.
The coordinates of $V$ associated to the basis $e_{1},\dots, e_{n}$ is the isomorphism $\CC^n\simeq V$ which sends $(z_{1},\dots,z_{n})\in \CC^n$ to $z_{1}e_{1}+\dots+z_{n}e_{n}\in V$.
It is easy to see that the vector field $\sum_{i=1}^{n}z_{i}\partial_{z_{i}}$ on $V$ does not depend on the choice of the basis $e_{1},\dots, e_{n}$.
It is called the Euler vector field on $V$.
More generally, 
for any local trivialization $\pi^{-1}(U) \simeq U\times \CC^n$ ($U\subset X$) of the vector bundle $\pi\colon E\to X$ and bundle coordinates $(x_{1},\dots,x_{m},z_{1},\dots,z_{n})\in U\times \CC^n$,
the vector field $\sum_{i=1}^{n}z_{i}\partial_{z_{i}}$ on $\pi^{-1}(U)$ does not depend on the choice of the local trivialization and we thus obtain a vector field on $E$.
It is called the Euler vector field on $E$ and we denote it by $\eu_{E}$.
We can regard $\eu_{E}$ as a section of $\pi_{*}D_{E}$.

\begin{defi}\label{mondef}
Let $M$ be a $D$-module on $E$.
We say that $M$ is monodromic if for any (local algebraic) section $m\in \pi_{*}M$ there is a polynomial $b(u)\in \CC[u]$ such that $b(\eu_{E})m=0$.
Moreover, if all the roots of the minimal polynomial of such $b(u)$ is in $\QQ$ (resp. $\RR$) for any $m$,
we say that $M$ is $\QQ$-monodromic (resp. $\RR$-monodromic).
\end{defi}

Since we only consider $\QQ$-monodromic $D$-module in this paper,
we will say ``monodromic'' as ``$\QQ$-monodromic''.

\begin{rem}\label{magical}
For a subset $V$ of $\CC^n$ (e.g. $(\CS)^n$) and a $D$-module $M$ on $X\times V$,
we also define ``$M$ is monodromic'' in the same way as Definition~\ref{mondef}.
\end{rem}

\begin{rem}\label{cyuusyo}
Remark that $E$ is equipped with a natural $\CS$-action.
Assume that $M$ is regular holonomic.
Let $K$ be the perverse sheaf corresponding to $M$ via the Riemann-Hilbert correspondence.
Then, $M$ is monodromic if and only if $K$ is cohomologically locally constant on each $\CS$-orbit in $E$, i.e. each cohomology $H^{j}(K)$ is locally constant on each $\CS$-orbit (Proposition~7.12 of Brylisnki~\cite{BryFou}).
\end{rem}

Note that we can also endow $\pi_{*}M$ with a $O_{X}$-modules structure by the adjunction $O_{X}\to \pi_{*}\pi^{-1}O_{X}\to \pi_{*}O_{E}$.
For a $D_{E}$-module $M$ and $\beta\in \QQ$,
we define a $O_{X}$-submodule $M^{\beta}$ of $\pi_{*}M$ by
\begin{align}\label{gion}
M^{\beta}:=\bigcup_{l\geq 0}\Ker((\eu_{E}-\beta)^{l}\colon \pi_{*}M\to \pi_{*}M).
\end{align}

\begin{prop}\label{decomprop}
A $D_{E}$-module $M$ is monodromic if and only if 
$M$ (recall that we identify it with $\pi_{*}M$) is a direct sum of the family of $O_{X}$-modules
$\{M^{\beta}\}_{\beta\in \RR}$ as
\begin{align}\label{macdonald}
M=\bigoplus_{\beta\in \RR}M^{\beta}.
\end{align}
\end{prop}

\begin{proof}
The proof is similar to the proof of Proposition~1.7 of \cite{SaiMon}.
\end{proof}

\begin{rem}\label{basie}
If $E$ is trivial and we fix a trivialization $E\simeq X\times \CC^n$,
we can endow $M^{\beta}$ with a natural $D_{X}$-module structure.
Because a lift of a section of $D_{X}$ to $E$ is not unique,
we can not define a natural $D_{X}$-module structure on $M^{\beta}$ in general.
\end{rem}

\begin{rem}
Submodules, quotient modules and extensions (in the category of $D_{E}$-modules) of monodromic $D_{E}$-modules are monodromic again.
\end{rem}

Let us consider a mixed Hodge module $\calM=(M,F_{\bullet}M, K, W_{\bullet}K)$ on $E$,
where $M$ is the underlying $D$-module, $F_{\bullet}M$ is the Hodge filtration, $K$ is the underlying $\QQ$-perverse sheaf and $W_{\bullet}K$ is the weight filtration.

\begin{defi}
We say that $\calM$ is monodromic if $M$ is monodromic.
\end{defi}

\subsection{{T}he case where $E$ is a trivial bundle of rank one}\label{subrankone}
Let us recall some results for a monodromic mixed Hodge module $\calM=(M,F_{\bullet}M, K, W_{\bullet}K)$ in the case where $E$ is a trivial bundle of rank one i.e. $E\simeq X\times \CC_{z}$.
We fix this trivialization in this subsection.
See \cite{SaiMon} for details.

Note that in this case $M^{\beta}(=\bigcup_{l\geq 0}\Ker((z\partial_{z}-\beta)^l\colon M\to M))$ is not only an $O_{X}$-module, it is a $D_{X}$-module (see Remark~\ref{basie}).
Moreover, we have a decomposition 
\begin{align}\label{treasure}
M=\bigoplus_{\beta\in \RR}M^{\beta},
\end{align}
where the $D_{X}[z]\langle{\partial_{z}}\rangle$-module structure is defined by using the morphisms
$z\colon M^{\beta}\to M^{\beta+1}$ and $\partial_{z}\colon M^{\beta}\to M^{\beta-1}$ (in fact, $M^{\beta}=0$ for $\beta\notin \QQ$).

We say that $M$ is $\KK$($=\QQ$ or $\RR$)-specializable if 
there exists the Kashiwara-Malgrange filtration along $z$ of $M$ indexed by $\KK$. We denote by $\{V^{\beta}_{z}M\}_{\beta\in \RR}$ the Kashiwara-Malgrange filtration of $M$ along $z$, where the index is defined to be $\GR^{\beta}_{V_{z}}M=V^{\beta}_{z}M/V^{>\beta}_{z}M$ is killed by $(z\pa_{z}-\beta)^l$ for some $l\geq 0$.
Sometimes $\GR^{\beta}_{V_{z}}M$ is abbreviated to $\GR^{\beta}_{V}M$.
The decomposition~(\ref{treasure}) leads to the following.

\begin{prop}[Proposition~1.15 of \cite{SaiMon}]\label{express}
Let $M$ be a monodromic coherent $D$-module on $X\times \CC_{z}$.
Then, $M$ is specializable, i.e. the Kashiwara-Malgrange filtration of $M$ along $z=0$ exists and we have
\[V_{z}^{\gamma}M=\bigoplus_{\beta\geq \gamma}M^{\beta}.\]
Therefore, we have
\[\GR^{\gamma}_{V}M=M^{\gamma}.\]
In particular, the $\alpha$-nearby cycle and the vanishing cycle of $M$ along $z$ are described as follows:
\begin{align}
    \psi_{z,\alpha}M(=\GR^{\alpha}_{V}M)&=M^{\alpha},\mbox{\quad and}\\
    \phi_{z,1}M(=\GR^{-1}_{V}M)&=M^{-1}.
\end{align}
Moreover, the morphism $\mathrm{can}\colon \psi_{z,0}M\to \phi_{z,1}M$ (resp. the morphism $\mathrm{var}\colon \phi_{z,1}M\to \psi_{z,0}M$) is $-\pa_{z}\colon M^{0}\to M^{-1}$ (resp. $z\colon M^{-1}\to M^{0}$).
\end{prop}

In this situation, the Hodge filtration is decomposed with respect to the decomposition~(\ref{treasure}).
\begin{prop}[Theorem~2.2 of \cite{SaiMon}]\label{rankoneFdecom}
For $p\in \ZZ$, the Hodge filtration $F_{p}M\subset M$ is decomposed as
\[F_{p}M=\bigoplus_{\beta\in \RR}F_{p}M^{\beta},\]
where $F_{p}M^{\beta}:=F_{p}M\cap M^{\beta}$ and the $O_{X}[z]$-module structure of the right hand side is defined by the morphisms $z\colon F_{p}M^{\beta}\to F_{p}M^{\beta+1}$.
\end{prop}

Let us recall the strict specializability for a filtered $D$-module.
See \cite{MHM} and \cite{MHMProj} for details.
Let $(M,F_{\bullet}M)$ be a holonomic $D$-module $M$ with a good filtration on $X\times \CC_{z}$.
We set $F_{p}\GR^{\beta}_{V}M:=F_{p}M\cap V^{\beta}_{z}M/F_{p}M\cap V^{>\beta}_{z}M$.
\begin{defi}\label{tron}
We say that $(M,F_{\bullet}M)$ is strictly $\KK$($=\QQ$ or $\RR$)-specializable along $z$ if $M$ is $\KK$-specializable and for any $p\in \ZZ$
\begin{enumerate}
\item[(i)] for any $\beta>-1$, $z\colon F_{p}\GR^{\beta}_{V}M\to F_{p}\GR^{\beta}_{V}M$\mbox{ is surjective, and}
\item[(ii)] for any $\beta<0$, $\partial_{z}\colon F_{p}\GR^{\beta}_{V}M\to F_{p+1}\GR^{\beta-1}_{V}M$ is surjective.
\end{enumerate}
\end{defi}
This property is one of the important constraints on mixed Hodge modules.
Proposition~\ref{rankoneFdecom} and the strict specializability lead the following.
\begin{lem}[Lemma~ 2.4 of \cite{SaiMon}]\label{kitune}
For the underlying filterd $D$-module $(M,F_{\bullet}M)$ of a monodromic mixed Hodge module on $X\times \CC_{z}$, $l\in \ZZ_{\geq 0}$ and $p\in \ZZ$ we have
\begin{align*}
F_{p}M^{\alpha+l}&=z^{l}F_{p}M^{\alpha}\quad (\alpha\in (-1,0]), \mbox{ and}\\
F_{p}M^{\alpha-l}&=\partial_{z}^{l}F_{p-l}M^{\alpha}\quad (\alpha\in [-1,0)).
\end{align*} 
In particular, we have
\begin{align}
F_{p}M=\left(\bigoplus_{l\geq 1}\bigoplus_{\alpha\in [-1,0)}\partial_{z}^{l}F_{p-l}M^{\alpha}\right)\oplus F_{p}M^{-1}\oplus \left(\bigoplus_{l\geq 0}\bigoplus_{\alpha\in (-1,0]}z^{l}F_{p}M^{\alpha}\right).
\end{align}
\end{lem}

We can describe the category $\mathrm{MHM}^{p}_{\mathrm{mon}}(X\times \CC_{z})$ of monodromic graded polarizable mixed Hodge modules on $X\times \CC_{z}$ as follows.
We consider a tuple $((\calM_{(-1,0]},T_{s},N), \calM_{-1},c,v)$, 
where $\calM_{(-1,0]}$ and $\calM_{-1}$ are graded polarizable mixed Hodge modules on $X$ and
$T_{s}\colon \calM_{(-1,0]}\simeq \calM_{(-1,0]}$, $N\colon \calM_{(-1,0]}\to \calM_{(-1,0]}(-1)$,
$c\colon \calM_{0}(:=\Ker(T_{s}-1)\subset \calM_{(-1,0]})\to \calM_{-1}$ and
$v\colon \calM_{-1}\to \calM_{0}(-1)$ are morphisms in the category of mixed Hodge modules with the following properties:
\begin{enumerate}
\item[(i)] $T_{s}$ commutes with $N$.
\item[(ii)] The underlying $D$-module $M_{(-1,0]}$ of $\calM_{(-1,0]}$ is decomposed as
\[M_{(-1,0]}=\bigoplus_{\alpha\in (-1,0]\cap \QQ}M_{\alpha},\]
where $M_{\alpha}:=\Ker(T_{s}-\exp(-2\pi\sqrt{-1}\alpha))\subset M_{(-1,0]}$.
\item[(iii)] $vc\colon \calM_{0}\to \calM_{0}(-1)$ is $-N$.
\end{enumerate}
We denote by $\mathscr{G}(X)$ the category of such tuples $((\calM_{(-1,0]},T_{s},N), \calM_{-1},c,v)$.

Let $\calM=(M,F_{\bullet}M, K, W_{\bullet}K)$ be a monodromic mixed Hodge module on $X\times \CC_{z}$.
We define $\calM_{(-1,0]}$ as the nearby cycle $\psi_{z}\calM$ of $\calM$ along $z$, $\calM_{-1}$ the unipotent vanishing cycle $\phi_{z,1}\calM$, $T_{s}$ (resp. $N$) the semisimple part (resp. 
$\frac{-1}{2\pi\sqrt{-1}}$ times the logarithm of the unipotent part) of the monodromy automorphism of $\psi_{z}\calM$ and $c$ (resp. $v$) the morphism $\mathrm{can}\colon \psi_{z,0}\calM\to \phi_{z,1}\calM$ (resp. $\mathrm{var}\colon \phi_{z,1}\calM\to \psi_{z,0}\calM(-1)$).
Then, the tuple $((\calM_{(-1,0]},T_{s},N), \calM_{-1},c,v)$ is an object in $\mathscr{G}(X)$.
In this way, we get a functor
\[F\colon \mathrm{MHM}^{p}_{\mathrm{mon}}(X\times \CC_{z})\to \mathscr{G}(X).\]

\begin{prop}[Theorem~3.5 of \cite{SaiMon}]\label{hoshinoouji}
The functor $F$ induces an equivalence of categories.
\end{prop}
In particular, we can reconstruct $\calM$ from 
the tuple $((\psi_{z}\calM,T_{s},N),\phi_{z,1}\calM,\mathrm{can},\mathrm{var})$.

\begin{rem}
As stated in Remark~\ref{magical}, we can also consider monodromic mixed Hodge modules on $X\times \CS_{z}$.
Then, we have a similar statement to Proposition~\ref{rankoneFdecom} and Proposition~\ref{hoshinoouji} for a monodromic mixed Hodge modules $\calM$ on $X\times \CS_{z}$ (see \cite{SaiMon}).
In this case,
$M$ is decomposed as
\[M=\bigoplus_{\beta\in \RR}M^{\beta},\]
where $M^{\beta}$ is defined as in the case of that on $X\times \CC_{z}$,
and 
for $\alpha\in (-1,0]$ and $k\in \ZZ$ we have
\[M^{\alpha+k}=z^{k}M^{\alpha}.\]
Moreover, as a corresponding assertion to Lemma~\ref{kitune}, we have
\[F_{\bullet}M=\bigoplus_{\substack{\alpha\in (-1,0]\\ k\in \ZZ}}z^{k}F_{\bullet}M^{\alpha}.\]
\end{rem}

\subsection{{E}xample: normal crossing type}
Let us consider monodromic $D$-modules on $E=X\times \CC^n$ with a stronger condition.
Let $(z_{1},\dots,z_{n})$ be the standard coordinates of $\CC^n$ and $\pi$ the projection $X\times \CC^n\to X$.

\begin{defi}
A $D$-module $M$ on $X\times \CC^n$ is of normal crossing type if 
for any section $m\in \pi_{*}M$ and $1\leq i\leq n$
there exists a polynomial $b(u)\in \CC[u]$ such that $b(z_{i}\partial_{z_{i}})m=0$.
In other words, for any $1\leq i\leq n$, $M$ is monodromic on $(X\times \CC_{z_{1}}\times \dots \CC_{z_{i-1}}\times \CC_{z_{i+1}}\times \dots \CC_{z_{n}})\times \CC_{z_{i}}$, where we regard $(X\times \CC_{z_{1}}\times \dots \CC_{z_{i-1}}\times \CC_{z_{i+1}}\times \dots \CC_{z_{n}})\times \CC_{z_{i}}$ as a rank one vector bundle over $X\times \CC_{z_{1}}\times \dots \CC_{z_{i-1}}\times \CC_{z_{i+1}}\times \dots \CC_{z_{n}}$.
\end{defi}

\begin{rem}
Let $M$ be a regular holonomic $D$-module $M$ of normal crossing type on $\CC^n$ and $K$ be the perverse sheaf corresponding to $M$. 
For $1\leq k\leq n$ and $1\leq i_{1},\dots,i_{k}\leq n$,
we set 
\[V_{i_{1},\dots, i_{k}}:=(\bigcap_{1\leq s\leq k}\{z_{s}=0\})\setminus (\bigcup_{s\notin \{i_{1},\dots,i_{k}\}}\{z_{s}=0\})\subset \CC^n.\]
Then, the restriction of each cohomology of $K$ to $V_{i_{1},\dots,i_{k}}$ or $\CC^n\setminus \bigcup_{1\leq s\leq n}\{z_{s}=0\}$ is locally constant by Remark~\ref{cyuusyo}.
Conversely, if $K$ has this property for a regular holonomic $D$-module on $\CC^n$, $M$ is of normal crossing type.
\end{rem}

For $\blbt=(\beta_{1},\dots,\beta_{n})\in \RR^n$ we set 
\[M^{\blbt}:=\bigcap_{i=1}^{n}\bigcup_{l\geq 0}(\Ker
((z_{i}\partial_{z_{i}}-\beta_{i})^l\colon \pi_{*}M\to \pi_{*}M).\]
Then, we can regard $M^{\blbt}$ as a $D_{X}$-module.

As mentioned, a $D$-module $M$ of normal crossing type is also a monodromic $D$-module with respect to any $z_{i}$-direction.
Therefore, we can apply the results in Subsection~\ref{subrankone} inductively.
For example, we have the following.

\begin{lem}\label{tanba}
$M$ is of normal crossing type if and only if $M$ is decomposed as
\[M=\bigoplus_{\blbt\in \RR^n}M^{\blbt}.\]
\end{lem}

For $\gamma\in \RR$, $M^{\gamma}=\bigcup_{l\geq 0}\Ker(\eu_{E}-\beta)^l\subset \pi_{*}M$ is a $D_{X}$-module (see Remark~\ref{basie}), where $\eu_{E}=\sum_{i=1}^nz_{i}\pa_{z_{i}}$.
Moreover, it is decomposed as follows.

\begin{lem}\label{potato}
We have
\[M^{\gamma}=\bigoplus_{\substack{\blbt\in \RR^n\\ \beta_{1}+\dots+\beta_{n}=\gamma}}M^{\blbt}\]
as a $D_{X}$-module.
\end{lem}

\begin{proof}
Let $m$ be a section of $M^{\blbt}$ and
$l_{0}$ an integer large enough so that 
$(z_{i}\partial_{z_{i}}-\beta_{i})^{l_{0}}m=0$ for any $1\leq i\leq n$.
We set $\gamma:=\beta_{1}+\dots+\beta_{n}$.
Then we have
\[(\eu_{E}-\gamma)^{nl_{0}}m=0,\]
and thus obtain 
\begin{align}\label{naniwa}
M^{\blbt}\subset M^{\gamma}.
\end{align}
Conversely, by Lemma~\ref{tanba},
$m\in M^{\gamma}$ is decomposed as
$m=m_{1}+\dots+m_{k}$, where $m_{s}(\neq 0)\in M^{\blbt_{s}}$ and $\blbt_{s}\neq \blbt_{t}$ for $s\neq t$.
By (\ref{naniwa}) and the decomposition~(\ref{macdonald}),
we have the converse inclusion $M^{\gamma}\subset \bigoplus_{\substack{\blbt\in \RR^n\\ \beta_{1}+\dots+\beta_{n}=\gamma}}M^{\blbt}$.
This implies the desired assertion.
\end{proof}

\begin{rem}\label{kochikame}
For $\blbt=(\beta_{1},\dots,\beta_{n})\in \RR^n$ and $1\leq i\leq n$,
it is easy to see that 
\begin{align*}
z_{i}M^{\blbt}&\subset M^{\blbt+\ble_{i}}\mbox{\ and}\\
\partial_{z_{i}}M^{\blbt}&\subset M^{\blbt-\ble_{i}},
\end{align*}
where $\ble_{i}=(0,\dots,0,\stackrel{i}{\hat{1}},0,\dots,0)$.
Moreover, in a similar way to the proof of Proposition~1.10 of \cite{SaiMon},
we can see that the morphism
\begin{align*}
z_{i}\colon M^{\blbt}\to & M^{\blbt+\ble_{i}}\\
(\mbox{resp.\ }\partial_{z_{i}}\colon M^{\blbt}\to &M^{\blbt-\ble_{i}})
\end{align*}
is an isomorphism if $\beta_{i}\neq -1$ (resp. $\beta_{i}\neq 0$).
Therefore, the $D$-module $M$ is determined by the following data:
\begin{enumerate}
\item[(i)] The family of $D_{X}$-modules $\{M^{\blal}\}_{\blal\in [-1,0]^n}$.
\item[(ii)] The nilpotent endomorphisms $z_{i}\partial_{z_{i}}-\alpha_{i}\colon M^{\blal}\to M^{\blal}$ for $\blal\in [-1,0]^n$ and $1\leq i\leq n$.
\item[(iii)] For $1\leq i\leq n$ and $\blal\in [-1,0]^n$ with $\alpha_{i}=-1$ (resp. $\alpha_{i}=0$),
the morphism $z_{i}\colon M^{\blal}\to M^{\blal+\ble_{i}}$ (resp. $\partial_{z_{i}}\colon M^{\blal}\to M^{\blal-\ble_{i}}$) such that the composition $\partial_{z_{i}}\circ z_{i}\colon M^{\blal}\to M^{\blal}$ (resp. $z_{i}\circ \partial_{z_{i}}\colon M^{\blal}\to M^{\blal}$) is equal to $z_{i}\partial_{z_{i}}+1$ (resp. $z_{i}\partial_{z_{i}}$) defined in (ii).
\end{enumerate}

\end{rem}

We assume that $M$ is coherent.
Let $V_{z_{i}}^{\bullet}M$ be the Kashiwara-Malgrange filtration of $M$ along $z_{i}$.
Then, by Proposition~\ref{express}, we have the following.
\begin{lem}
For a coherent $D$-module $M$ of normal crossing type,
the Kashiwara-Malgrange filtration of $M$ along $\ti$ exists and we have
\[V_{\ti}^{\gamma}M=\bigoplus_{\substack{
\blbt=(\beta_{1},\dots,\beta_{n})\in \RR^n\\
\beta_{i}\geq \gamma}}M^{\blbt}.\]
In particular, for $\alpha\in (-1,0]$ the $\alpha$-nearby cycle $\psi_{\ti,\alpha}M:=\GR_{V_{\ti}}^{\alpha}M$ 
 (resp. the unipotent vanishing cycle $\phi_{\ti,1}M:=\GR_{V_{\ti}}^{-1}M$) of $M$
can be described as
\begin{align*}
\psi_{\ti,\alpha}M=\bigoplus_{\substack{
\blbt=(\beta_{1},\dots,\beta_{n})\in \RR^n\\
\beta_{i}= \alpha}}M^{\blbt}\\
(\mbox{resp. }\phi_{\ti,1}M=\bigoplus_{\substack{
\blbt=(\beta_{1},\dots,\beta_{n})\in \RR^n\\
\beta_{i}= -1}}M^{\blbt}).
\end{align*}
\end{lem}

The previous lemma implies that the nearby cycle $\psi_{\ti,\alpha}M$ and the vanishing cycle $\phi_{\ti,1}M$ is again a $D$-module of normal crossing type on $X\times \CC_{z_{1}}\times \dots \CC_{z_{i-1}}\times \{0\}\times \CC_{z_{i+1}}\times \dots \CC_{z_{n}}$.
This allows us to prove the following proposition.
\begin{prop}\label{gozzira}
Let $\calM$ be a mixed Hodge module on $X\times \CC^n$ whose underlying $D$-module $M$ is of normal crossing type.
Then, the Hodge filtration $F_{p}M$ is decomposed as
\[F_{p}M=\bigoplus_{\blbt\in \RR^n}F_{p}M^{\blbt},\]
where $F_{p}M^{\blbt}:=F_{p}M\cap M^{\blbt}$ and the $O_{X}[z_{1},\dots,z_{n}]$-module structure of the right hand side is defined by the morphisms $z_{i}\colon F_{p}M^{\blbt}\to F_{p}M^{\blbt+\ble_{i}}$ for $\blbt\in \RR^n$.
\end{prop}

\begin{proof}
The proof is by induction on $n$.
The assertion for $n=1$ is Proposition~\ref{rankoneFdecom}.
Suppose that the statement is proved for $n=n_{0}-1$ ($n_{0}\geq 2$) and consider the case where $n=n_{0}$.
We set $M^{\beta}_{z_{n}}:=\bigcup_{l\geq 1}\Ker(z_{n}\partial_{z_{n}}-\beta)^l\subset \pi_{*}M$.
Since $M$ is monodromic with respect to the $z_{n}$-direction on $(X\times \CC^{n-1})\times \CC_{z_{n}}$,
we have
\begin{align}
M&=\bigoplus_{\beta\in \RR}M^{\beta}_{z_{n}} \mbox{\quad and}\\
\label{chance}
F_{p}M&=\bigoplus_{\beta\in \RR}F_{p}M^{\beta}_{z_{n}},
\end{align}
where $F_{p}M^{\beta}_{z_{n}}=F_{p}M\cap M^{\beta}_{z_{n}}$.
Note that for if $\beta$ is in  $(-1,0]$ (resp. $\beta$ is $-1$),
we have $M_{z_{n}}^{\beta}=\psi_{z_{n},\beta}M$ (resp. $M_{z_{n}}^{\beta}=\phi_{z_{n},1}M$).
As mentioned above, the $\alpha$-nearby cycle ($\alpha\in (-1,0]$) and the unipotent vanishing cycle of $M$ are of normal crossing type.
Moreover, they are a direct summand of a mixed Hodge module and their Hodge filtrations are $F_{\bullet}M_{z_{n}}^{\alpha}$ and $F_{\bullet}M_{z_{n}}^{\alpha}$ up to shift.
Therefore, by the induction hypothesis,
we have
\begin{align*}
F_{p}M_{z_{n}}^{\alpha}&=\bigoplus_{\substack{
\blbt=(\beta_{1},\dots,\beta_{n})\in \RR^n\\
\beta_{n}= \alpha}}F_{p}M^{\blbt}\mbox{\quad and,}\\
F_{p}M_{z_{n}}^{-1}&=\bigoplus_{\substack{
\blbt=(\beta_{1},\dots,\beta_{n})\in \RR^n\\
\beta_{n}= -1}}F_{p}M^{\blbt}.
\end{align*}
By (\ref{chance}) and the strict specializability along $z_{n}=0$,
we obtain the desired assertion.
\end{proof}

In particular, combining it with Lemma~\ref{potato},
we have the following.
\begin{cor}
In the situation of Proposition~\ref{gozzira},
we have
\[F_{p}M=\bigoplus_{\gamma\in \RR}F_{p}M^{\gamma},\mbox{ and}\]
\[F_{p}M^{\gamma}=\bigoplus_{\substack{\blbt\in \RR^n\\ \beta_{1}+\dots+\beta_{n}=\gamma}}F_{p}M^{\blbt}.\]
\end{cor}

In the next section,
we will generalize this assertion to the case for general monodromic mixed Hodge modules, which is not necessarily of normal crossing type.

\begin{rem}
For a mixed Hodge module of normal crossing type $M$ and $\blbt=(\beta_{1},\dots,\beta_{n})\in (-1,0]^n$,
it is easy to see
\[M^{\blbt}=\psi_{z_{1},\beta_{1}}\cdots \psi_{z_{n},\beta_{n}}M.\]
Moreover,
for example, for $\blbt=(-1,\beta_{2},\dots,\beta_{n})\in \{-1\}\times (-1,0]^{n-1}$,
we have
\[M^{\blbt}=\phi_{z_{1},1}\psi_{z_{2},\beta_{2}}\cdots \psi_{z_{n},\beta_{n}}M.\]
A similar statement holds for any $\blbt\in [-1,0]^n$.
Then, we can generalize the gluing: Proposition~\ref{hoshinoouji} to the normal crossing case.
In particular, 
the mixed Hodge module of normal crossing type $\calM$ can be reconstructed from
the family of mixed Hodge modules $\{\Psi_{1}\cdots \Psi_{n}\calM\}_{(\Psi_{1},\cdots ,\Psi_{n})}$ with some morphisms between them,
where $\Psi_{i}$ is $\psi_{z_{i},\alpha}$ ($\alpha\in (-1,0]$) or $\phi_{z_{i},1}$.
\end{rem}

\section{{T}he Hodge filtration of monodromic mixed Hodge modules}

Let $\calM=(M,F_{\bullet}M,K,W_{\bullet}K)$ be a monodromic mixed Hodge module on a vector bundle $\pi\colon E\to X$ on a smooth algebraic variety $X$.
By Proposition~\ref{decomprop},
we have the decomposition
\[M=\bigoplus_{\beta\in \RR}M^{\beta}.\]
For $p\in \ZZ$ and $\beta\in \RR$,
we define an $O_{X}$-submodule of $M^{\beta}$ as
\[F_{p}M^{\beta}:=\pi_{*}(F_{p}M)\cap M^{\beta} (\subset \pi_{*}M).\]
Then, the direct sum $\bigoplus_{\beta\in \RR}F_{p}M^{\beta}$ is a $\pi_{*}O_{E}$-submodule of $\bigoplus_{\beta\in \RR}M^{\beta}$.
Therefore, by Lemma~\ref{haruni}, we can also regard $\bigoplus_{\beta\in \RR}F_{p}M^{\beta}$ as an $O_{E}$-submodule of $M$.
The purpose of this section is to show the following theorem.

\begin{thm}\label{Fdecomp}
For $p\in \ZZ$, the Hodge filtration $F_{p}M$ is decomposed as
\[F_{p}M=\bigoplus_{\beta\in \RR}F_{p}M^{\beta}.\]
\end{thm}

\begin{rem}
This result was shown in a different way in a recent preprint \cite{ChenDirks} by Chen-Dirks.
\end{rem}

Since it is enough to show this theorem locally on $X$, 
we may assume that $E$ is a trivial bundle $X\times \CC^n$ and $X$ is affine (therefore, we identify $M$ with the module of its global sections).
Let $(z_{1},\dots,z_{n})$ be the coordinates of $\CC^n$.
We set $D_{1}:=\{z_{1}=0\}\subset X\times \CC^n$ and $V_{1}:=E\setminus D_{1}$.

Let us recall some basic properties of the localization $\calM[*D_{1}]$ and the dual localization $\calM[!D_{1}]$ of a mixed Hodge module $\calM$. For details, see \cite{Beilinson}, \cite{MHMProj} and \cite{MTM}.
We denote by $M[*D_{1}]$ (resp. $M[!D_{1}]$) the underlying $D$-module of $\calM[*D_{1}]$ (resp. $\calM[!D_{1}]$).
The stupid localization $M(*D_{1})$ (resp. $(M,F_{\bullet}M)(*D_{1})$) along $D_{1}$ of a $D$-module $M$ (resp. a filtered $D$-module $(M,F_{\bullet}M)$) is the $D_{E}(*D_{1})(=D_{E}\otimes_{\CC[z_{1}]}\CC[z_{1}^{\pm 1}])$-module (resp. the filtered $D_{E}(*D_{1})$-module) defined as
\[M\otimes_{\CC[z_{1}]}\CC[z_{1}^{\pm 1}]\]
\[(\mbox{resp.\ } (M\otimes_{\CC[z_{1}]}\CC[z_{1}^{\pm 1}],F_{\bullet}M\otimes_{\CC[z_{1}]}\CC[z_{1}^{\pm 1}])).\]
Let $V_{z_{1}}^{\bullet}M$ be a Kashiwara-Malgrange filtration of a $D$-module $M$ along $z_{1}$.

\begin{prop}[see \cite{Beilinson}, \cite{MHMProj} and \cite{MTM}]\label{localizfac}
Let $\calM=(M,F_{\bullet}M, K, W_{\bullet}M)$ be a mixed Hodge module on $E=X\times \CC^n$.
Then, we have the following.

\begin{enumerate}
\item[(i)]\label{locDmod}  The underlying $D$-modules are as follows:
\begin{align*}
M[*D_{1}]&=M(*D_{1})=M\otimes_{\CC[z_{1}]}\CC[z_{1}^{\pm 1}], \mbox{\quad and}\\
M[!D_{1}]&=\mathbf{D}(\mathbf{D}(M)(*D_{1}))=D_{E}\otimes_{V_{z_{1}}^{0}D_{E}}V^{>-1}_{z_{1}}M,
\end{align*}
where $V_{z_{1}}^{\bullet}D_{E}$ is the $V$-filtration of $D_{E}$ along $z_{1}$ and $\mathbf{D}$ is the duality functor between the category of mixed Hodge modules. 

\item[(ii)] We have an (canonical) isomorphism $\mathbf{D}(\calM[*D_{1}])\simeq (\mathbf{D}\calM)[!D_{1}]$.

\item[(iii)] There is natural morphisms $\calM\to \calM[*D_{1}]$ and $\calM[!D_{1}]\to \calM$ whose restriction to $V_{1}$ are isomorphisms.
In particular, the stupid localizations of the underlying filtered $D$-modules of $\calM[*D_{1}]$ and $\calM[!D_{1}]$ are the stupid localization $(M,F_{\bullet}M)(*D_{1})$ of the underlying filtered $D$-module of $\calM$, and we have 
\[\Vto^{>-1}M(*D_{1})=\Vto^{>-1}M[!D_{1}]=\Vto^{>-1}M.\]

\item[(iv)] \label{elite}We have
\begin{align*} 
\Vto^{\geq -1}M[*D_{1}]&=z_{1}^{-1}\Vto^{\geq 0}M\mbox{, and}\\
F_{p}\Vto^{\geq -1}M[*D_{1}]&=z_{1}^{-1}F_{p}\Vto^{\geq 0}M\quad (p\in \ZZ).
\end{align*}

\item[(v)] \label{hirakata}The Hodge filtrations are described as follows:
\begin{align*}
F_{p}(M[*D_{1}])&=\sum_{k\geq 0}\partial_{z_{1}}^{k}F_{p-k}V^{\geq -1}_{z_{1}}M, \mbox{\quad and}\\
F_{p}(M[!D_{1}])&=\sum_{k\geq 0}\partial_{z_{1}}^k\otimes F_{p-k}V^{>-1}_{z_{1}}M, 
\end{align*}
where $F_{p-k}V^{\geq -1}_{z_{1}}M=F_{p-k}M\cap V^{\geq -1}_{z_{1}}M$ and $F_{p-k}V^{>-1}_{z_{1}}M=F_{p-k}M\cap V^{>-1}_{z_{1}}M$.
With (iv), in particular, the filtered $D$-modules of $\calM[*D_{1}]$ and $\calM[!D_{1}]$ are determined only by the stupid localization $(M,F_{\bullet}M)(*D_{1})$.
\end{enumerate}
\end{prop}

The following is a simple consequence of (i) of Proposition~\ref{localizfac} 
\begin{lem}
If $M$ is monodromic, then $M(*D_{1})$ and $M[!D_{1}]$ are also monodromic.
\end{lem}

Let $\rho\colon \wt{\CC^n}\to \CC^n$ be the blowing up of $\CC^n$ at the origin.
Remark that $\wt{\CC^n}$ is a subvariety of $\CC^n\times \PP^{n-1}$ and $\rho$ is the projection to $\CC^n$.
We write $q\colon \wt{\CC^n}\to \PP^{n-1}$ the projection to $\PP^{n-1}$.
Let $[y_{1}\colon \dots\colon y_{n}]$ be the homogeneous coordinates of $\PP^{n-1}$.
Define $U_{1}$ as a local chart $\{y_{1}\neq 0\}\subset \PP^{n-1}$ of $\PP^{n-1}$.
We use the same symbol $(y_{2},\dots,y_{n})$ for the coordinates of $U_{1}$, i.e.
$(y_{2},\dots,y_{n})\in U_{1}$ is the point $[1\colon y_{2}\colon \dots \colon y_{n}]\in \PP^{n-1}$.
Then, we have
\begin{align}\label{kabi}
\CS_{s}\times U_{1}\simto q^{-1}&(U_{1})\simto V_{1}\\
 (\CS_{s}\times U_{1}\ni (s,y_{1},\dots,y_{n})\mapsto &(s,sy_{1},\dots,sy_{n})\in V_{1}).\notag
\end{align}

The following simple lemma reduces a problem for a monodromic $D$-module on a vector bundle to that for a monodromic $D$-module on a line bundle.

\begin{lem}\label{eegekai}
The isomorphism (\ref{kabi}) sends the vector field $s\partial_{s}$ on $\CS_{s}\times U_{1}$
to the vector field $\eu=\sum_{i=1}^nz_{i}\pa_{z_{i}}$ on $V_{1}$.
\end{lem}
\begin{proof}
Let $G:=(g_{1},\dots,g_{n})$ be the morphism (\ref{kabi}).
Then, 
(\ref{kabi}) sends the vector filed $s\partial_{s}$ to
\begin{align}\label{doragon}
    z_{1}\sum_{k=1}^{n}(\partial g_{k}/\partial s\circ G^{-1})\partial_{z_{k}}.
\end{align}
Since
\[\partial g_{k}/\partial s=\left\{ \begin{array}{ll}1&(k=1)\\ y_{k}&(k\neq 1)\end{array}\right.,\]
we have
\begin{align*}
(\ref{doragon})&=z_{1}\partial_{z_{1}}+z_{1}\sum_{k=2}^{n}(z_{k}/z_{1})\partial_{z_{k}}\\
&=\eu.
\end{align*}
\end{proof}

We write $\rho_{1}$ for the induced isomorphism $ q^{-1}(U_{1})\simto V_{1}$ by $\rho$ and $H^0L\rho^{*}_{1}$ (the $0$-th cohomology of) the pullback functor of the category of $D$-modules.
Since $\rho_{1}\colon q^{-1}(U_{1})\simto V_{1}$ is an isomorphism, $H^0L\rho^{*}_{1}M_{1}$ is just the pullback $\rho^{*}_{1}M_{1}=O_{ q^{-1}(U_{1})}\otimes_{\rho_{1}^{-1}O_{ V_{1}}}\rho^{-1}_{1}M_{1}$ as an $O$-module for a $D$-module $M_{1}$ on $V_{1}$.
We just write $\rho^{*}_{1}M_{1}$ for $H^0L\rho^{*}_{1}M_{1}$.
Note that any section $m\in \rho^{*}_{1}M_{1}$ can be expressed as $m=1\otimes m'$ for some $m'\in M_{1}$.
The morphisms $\rho$ and $\rho_{1}$ induces morphisms $X\times \wt{\CC^n}\to X\times \CC^n$ and $X\times q^{-1}(U_{1})\to X\times V_{1}$, denoted by the same symbols $\rho$ and $\rho_{1}$.
For a monodromic $D$-module $M_{1}$ on $X\times V_{1}$, 
we set $\wt{M_{1}}:=\rho^{*}M_{1}$.
Lemma~\ref{eegekai} immediately deduces the following. 

\begin{cor}\label{hunade}
A $D$-module $M_{1}$ on $X\times V_{1}$ is monodromic (in the sense of Remark~\ref{magical}) if and only if the $D$-module $\wt{M_{1}}(=\rho_{1}^{*}M_{1})$ on $X\times q^{-1}(U_{1})(\simeq X\times  \CS_{s}\times U_{1})$ is monodromic with respect to the $s$-direction. 
\end{cor}

By this corollary, we have a decomposition
\begin{align}\label{harukanaru}
\wt{\Mi}=\bigoplus_{\beta\in \RR}\wt{\Mi}^{\beta},
\end{align}
where $\wt{\Mi}^{\beta}=\bigcup_{l\geq 0}(\Ker((s\pa_{s}-\beta)^l\colon \wt{\Mi}\to \wt{\Mi})$.
We can regard $\wt{\Mi}^{\beta}$ as a $D_{X\times U_{1}}$-module (see Remark~\ref{basie}).
Recall that we also have
\begin{align}\label{porutoga}
M_{1}=\bigoplus_{\beta\in \RR}M_{1}^{\beta},
\end{align}
where $M_{1}^{\beta}=\bigcup_{l\geq 0}(\Ker((\eu-\beta)^l\colon M_{1}\to M_{1})$.
Let us see the raletionship between (\ref{harukanaru}) and (\ref{porutoga}).

\begin{lem}\label{sinpi}
For $\beta\in \RR$ we have
\[\wt{\Mi}^{\beta}=1 \otimes \rho^{-1}_{1}(M_{1}^{\beta})(=O_{X\times U_{1}}\otimes \rho^{-1}_{1}(M_{1}^{\beta})),\]
as $D_{X\times U_{1}}$-modules.
\end{lem}
\begin{proof}
By Lemma~\ref{eegekai},
for a section $m\in M_{1}$ and $1\otimes m\in \wt{M_{1}}$, we have
\[s\pa_{s} (1\otimes m)=1\otimes (\eu m).\]
Therefore,
if the section $m$ is in $M_{1}^{\beta}$,
the section $1\otimes m\in \wt{M_{1}}$ is in $\wt{M_{1}}^{\beta}$.
Hence, we have $1\otimes \rho^{-1}_{1}(M_{1}^{\beta})\subset \wt{\Mi}^{\beta}$.
Let us show the reverse inclusion.
Any section $m\in \wt{\Mi}$ can be expressed as 
$m=\sum_{i=0}^{i_{0}}(1\otimes m_{i})$ with $m_{i}\in M^{\beta_{i}}$ for some $\beta_{i}\in \RR$ by (\ref{porutoga}).
Assume that $m$ is in $\wt{M_{1}}^{\beta}$.
Since $1\otimes m_{i}$ is in $\wt{\Mi}^{\beta_{i}}$ as already shown,
the decomposition (\ref{harukanaru}) implies that $\sum_{\beta_{i}\neq \beta}(1\otimes m_{i})=0$.
Therefore, we have $m=\sum_{\beta_{i}=\beta}(1\otimes m_{i})$
and we thus obtain
$\wt{\Mi}^{\beta}\subset 1\otimes \rho^{-1}_{1}(M_{1}^{\beta})$.
\end{proof}

For a monodromic mixed Hodge module $\calM_{1}$ on $X\times V_{1}$,
we consider the pullback $H^{0}\rho_{1}^{*}\calM_{1}$ of $\calM_{1}$ by $\rho_{1}$ as a mixed Hodge module, whose underlying $D$-module is $\rho_{1}^*M_{1}$. 
We set $\wt{\calM_{1}}:=H^0\rho_{1}^{*}\calM_{1}$ and $(\wt{\Mi},F_{\bullet}\wt{\Mi})$ is the underlying filtered $D$-module on $X\times q^{-1}(U_{1})$.
Since $\rho_{1}$ is an isomorphism,
the Hodge filtration $F_{\bullet}\wt{\Mi}$ is just the pullback of the Hodge filtration $F_{\bullet}\Mi$ as that of $O_{X\times V_{1}}$-modules:
\begin{align}\label{yasuragi}
F_{p}\wt{\Mi}=O_{X\times q^{-1}(U_{1})}\otimes_{\rho_{1}^{-1}O_{X\times V_{1}}}\rho_{1}^{-1}F_{p}\Mi.
\end{align}

In order to prove Proposition~\ref{behomazun} below,
we need the following 
\begin{lem}\label{kaizu}
We have
\begin{align}\label{kouya}
F_{p}\wt{\Mi}\cap \wt{\Mi}^{\beta}=1\otimes \rho_{1}^{-1}(F_{p}M_{1}\cap M_{1}^{\beta}).
\end{align}
\end{lem}

\begin{proof}
By Lemma~\ref{sinpi} and (\ref{yasuragi}), the right hand side of (\ref{kouya}) is contained in the left hand side of (\ref{kouya}).
Let $m$ be a section in the left hand side.
By (\ref{yasuragi}),
we can write $m=1\otimes m'$ for some $m'\in F_{p}M_{1}$.
Let $m'=\sum_{i=0}^{i_{0}}m_{i}'$ be the decomposition, where $m_{i}'\in M^{\beta_{i}}$ for some $\beta_{i}\in \RR$ with $\beta_{i}\neq \beta_{j}$ ($i\neq j$).
By Lemma~\ref{sinpi}, $1\otimes m'_{i}$ is in $\wt{M_{1}}^{\beta_{i}}$.
Since $m\in \wt{M_{1}}^{\beta}$, we have $1\otimes m'_{i}=0$, i.e. $m_{i}'=0$ if $\beta_{i}\neq \beta$.
Hence, $m'\in M^{\beta}_{1}$ and we thus conclude that $m$ is in the right hand side of (\ref{kouya}).
\end{proof}

Combining Corollary~\ref{hunade} and Proposition~\ref{rankoneFdecom},
we have the following.

\begin{prop}\label{behomazun}
For a monodromic mixed Hodge module $\calM_{1}$ on $X\times V_{1}$,
we have a decomposition of the Hodge filtration as
\begin{align}
F_{p}M_{1}=\bigoplus_{\beta\in \RR}F_{p}\Mi^{\beta},
\end{align}
where $(M_{1}, F_{\bullet}\Mi)$ is the underlying filtered $D$-module of $\calM_{1}$ and $F_{p}M^{\beta}=F_{p}\Mi\cap \Mi^{\beta}$.
\end{prop}

\begin{proof}
By Corollary~\ref{hunade}, $\wt{\calM_{1}}$ is monodromic on $X\times q^{-1}(U_{1})\simeq X\times \CC_{s}^*\times U_{1}$ with respect to the $s$-direction.
Therefore, by Proposition~\ref{rankoneFdecom}, we have
\begin{align}\label{tenkuu}
    F_{p}\wt{\Mi}=\bigoplus_{\beta\in \RR}F_{p}\wt{\Mi}^{\beta},
\end{align}
where $F_{p}\wt{\Mi}^{\beta}=F_{p}\wt{\Mi}\cap \wt{\Mi}^{\beta}$.
Moreover, by Lemma~\ref{kaizu}, we have
\begin{align}\notag
F_{p}\wt{\Mi}&=\bigoplus_{\beta\in \RR}1\otimes \rho_{1}^{-1}(F_{p}M_{1}\cap M_{1}^{\beta})\\
&=1\otimes \rho_{1}^{-1}(\bigoplus_{\beta\in \RR}F_{p}M_{1}\cap M_{1}^{\beta}).\label{subarasii}
\end{align}
Note that $\bigoplus_{\beta\in \RR}F_{p}M_{1}\cap M_{1}^{\beta}$ is an $O_{X\times V_{1}}$-submodule of ${\Mi}$.
Then, since $\rho_{1}$ is an isomorphism,
from the equalities (\ref{yasuragi}) and (\ref{subarasii}) we get the desired equality.
\end{proof}

Let $\calM=(M,F_{\bullet}M,K,W_{\bullet}M)$ be a monodromic mixed Hodge module on $X\times \CC^n$.
We set $\calM_{1}:=\calM|_{X\times V_{1}}$.
Its underlying filtered $D$-module is $(M_{1},F_{\bullet}M_{1}):=(M,F_{\bullet}M)|_{X\times V_{1}}$.

\begin{cor}\label{oboro}
For $p\in\ZZ$,
$F_{p}V^{>-1}_{z_{1}}M$ is decomposed with respect to the decomposition $M=\bigoplus_{\beta\in \RR}M^{\beta}$, i.e. we have
\[F_{p}V^{>-1}_{z_{1}}M=\bigoplus_{\beta\in \RR}F_{p}V^{>-1}_{z_{1}}M\cap M^{\beta}.\]
\end{cor}
\begin{proof}
By the strict specializability along $z_{1}$ of the filtered $D$-module $(M,F_{\bullet}M)$ (Definition~\ref{tron}),
we have (see Proposition~3.2.2 of \cite{HM88} or Exercise~11.1 of \cite{Schnell})
\begin{align}\label{tuki}
F_{p}V^{>-1}_{z_{1}}M=j_{*}(F_{p}M_{1})\cap V_{z_{1}}^{>-1}M,
\end{align}
where $j$ is the inclusion $X\times V_{1}\hookrightarrow X\times \CC^n$ and the intersection in the right hand side is taken in $j_{*}M_{1}=M[z_{1}^{\pm 1}]$.
By Proposition~\ref{behomazun},
we have
\begin{align}\label{taiyou}
F_{p}M_{1}=\bigoplus_{\beta\in \RR}F_{p}M^{\beta}_{1}.
\end{align}
By Lemma~\ref{keihan} below,
we have
\begin{align*}
V^{>-1}_{z_{1}}j_{*}M_{1}\cap (j_{*}M_{1})^{\beta}=V^{>-1}_{z_{1}}j_{*}M_{1}\cap j_{*}(M_{1}^{\beta}),
\end{align*}
where $(j_{*}M_{1})^{\beta}$ is defined similarly to (\ref{gion}) and $j_{*}(M_{1}^{\beta})$ is a $\CC$-submodule of $j_{*}M_{1}$ generated by $\{j_{*}m\in j_{*}M_{1}\ |\ m\in M_{1}^{\beta}\}$.
Therefore, since $V^{>-1}_{z_{1}}j_{*}M_{1}=V^{>-1}_{z_{1}}M$,
we have
\begin{align*}
V^{>-1}_{z_{1}}M\cap M^{\beta}
= V^{>-1}_{z_{1}}j_{*}M_{1}\cap j_{*}(M^{\beta}_{1}).
\end{align*}
Hence, we have
\begin{align}\label{mokusei}
 V_{z_{1}}^{>-1}M\cap j_{*}(F_{p}M_{1}^{\beta})=
V^{>-1}_{z_{1}}M\cap F_{p}M^{\beta}.
\end{align}
Combining (\ref{tuki}), (\ref{taiyou}) and (\ref{mokusei}),
we obtain
\[F_{p}V^{>-1}_{z_{1}}M=\bigoplus_{\beta\in \RR}F_{p}M\cap V^{>-1}_{z_{1}}M\cap M^{\beta}.\]
\end{proof}

The following was used in the proof of Corollary~\ref{oboro}.
\begin{lem}\label{keihan}
We have
\begin{align}\label{kinsei}
V^{>-1}_{z_{1}}j_{*}M_{1}\cap (j_{*}M_{1})^{\beta}=V^{>-1}_{z_{1}}j_{*}M_{1}\cap j_{*}(M_{1}^{\beta}).
\end{align}
\end{lem}
\begin{proof}
It is obvious that the left hand side contained in the right hand side.
For a section $m$ in the right hand side of (\ref{kinsei}),
since $V^{>-1}_{z_{1}}j_{*}M_{1}=V^{>-1}_{z_{1}}M$, $m$ is a section of $M$ with $((\eu-\beta)^lm)|_{X\times V_{1}}=0$ for some $l\geq 0$.
Therefore, $(\eu-\beta)^lm$ is a section of $V^{>-1}_{z_{1}}j_{*}M_{1}$ whose support is contained in $z_{1}=0$.
Since the multiplication by $z_{1}$ on $V^{>-1}_{z_{1}}j_{*}M_{1}$ is injective,
we have $(\eu-\beta)^lm=0$, i.e. $m$ is in $(j_{*}M_{1})^{\beta}$.
\end{proof}

As mentioned, if $M$ is monodromic, the localizations $M[*D_{1}]$ and $M[!D_{1}]$ are also monodromic.
Corollary~\ref{oboro} deduces the following.
\begin{cor}\label{kaniseijin}
If $\calM \simeq \calM[*D_{1}]$ or $\calM\simeq \calM[!D_{1}]$,
the assertion stated in Theorem~\ref{Fdecomp} is true.
\end{cor}

\begin{proof}
Suppose that $\calM=\calM[*D_{1}]$.
Then, by Corollary~\ref{oboro},
$F_{p}V_{z_{1}}^{\geq 0}M$ is decomposed as
$F_{p}V_{z_{1}}^{\geq 0}M=\bigoplus_{\beta\in \RR}F_{p}V_{z_{1}}^{\geq 0}M\cap M^{\beta}$.
Hence, by (iv) of Proposition~\ref{localizfac},
$F_{p}V_{z_{1}}^{\geq -1}M[*D_{1}]$ is also decomposed.
Since $F_{p}M=\sum_{k\geq 0}\pa_{\ti}^{k}F_{p-k}V^{\geq -1}_{\ti}M$ by (v) of Proposition~\ref{localizfac},
we thus obtain the decomposition of $F_{p}M$ in this case.
The case of $\calM=\calM[!D_{1}]$ can be proved in the same way.
\end{proof}

Let us recall the Beilinson's maximal extension.
See \cite{Beilinson}, \cite{MHMProj}, \cite{MTM}, \cite{ReichKaisetu} and \cite{Morel} for details.
We consider the vector space $I^{\varepsilon ,k}:=\CC e_{\varepsilon}\oplus \dots\oplus \CC e_{k}$ with the basis $e_{\varepsilon},\dots,e_{k}$ for $\varepsilon=0,1$ and $k\in \ZZ_{\geq 1}$.
Moreover, we consider the nilpotent endomorphism $N$ of $I^{\varepsilon,k}$ so that $Ne_{j}=e_{j-1}$ (with $e_{-1}:=0$).
For a $D$-module $M$ on $X\times \CC^n$ with the assumption $M=M(*D_{1})$,
we consider the new $D$-module $M^{\varepsilon,k }:=M\otimes_{\CC}I^{\varepsilon,k}$ on $X\times \CC^n$ so that $z_{1}\pa_{z_{1}}(m\otimes e_{j})=z_{1}\pa_{z_{1}}m\otimes e_{j}+m\otimes e_{j-1}$ and $\pa_{z_{s}}(m\otimes e_{j})=\pa_{z_{s}}m\otimes e_{j}$ ($s\neq 1$).
The morphism $I^{0,k}\to I^{1,k}$ ($e_{j}\mapsto e_{j}$, $e_{0}\mapsto 0$) induces the morphism $M^{0,k}\to M^{1,k}$.
Therefore, we can consider the morphism $M^{0,k}[!D_{1}]\to M^{1,k}[*D_{1}]$ as the composition of the morphisms $M^{0,k}[!D_{1}]\to M^{0,k}[*D_{1}]$ and $M^{0,k}[*D_{1}]\to M^{1,k}[*D_{1}]$.
Then, the kernel of the natural morphism
\[M^{0,k}[!D_{1}]\to M^{1,k}[*D_{1}]\]
does not depend on sufficiently large $k\geq1$, i.e. 
the inductive limit $\varinjlim_{k}\Ker(M^{0,k}[!D_{1}]\to M^{1,k}[*D_{1}])$ exists.
So, we define
\[\Xi_{z_{1}}M:=\varinjlim_{k}\Ker(M^{0,k}[!D_{1}]\to M^{1,k}[*D_{1}]).\]
We can generalize this construction to mixed Hodge modules;
for a mixed Hodge module $\calM$ on $X\times \CC^n$,
we can define a mixed Hodge module $\calM^{\varepsilon,k}$ on $X\times \CC^n$, a morphism $\calM^{0,k}[!D_{1}]\to \calM^{1,k}[*D_{1}](-1)$ and
$\Xi_{z_{1}}\calM=\varinjlim_{k}\Ker(\calM^{0,k}[!D_{1}]\to \calM^{1,k}[*D_{1}])$,
which are compatible with the corresponding objects for the underlying $D$-module of $\calM$.
Note that the filtered $D$-module $(\Xi_{z_{1}}M,F_{\bullet}\Xi_{z_{1}}M)$ depends only on the stupid localization $(M,F_{\bullet}M)(*D_{1})$.

\begin{prop}[see loc. cit.]\label{bongore}
Let $\calM=(M,F_{\bullet}M, K, W_{\bullet}M)$ is a mixed Hodge module on $E=X\times \CC^n$.
Then, we have the following.

\begin{enumerate}
\item[(i)] There are natural morphisms between mixed Hodge modules
\begin{align*}
a\colon \psi_{z_{1},1}\calM\to \Xi_{z_{1}}\calM,\quad \mbox{and}\\
b\colon \Xi_{z_{1}}\calM\to \psi_{z_{1},1}\calM(-1).
\end{align*}

\item[(ii)] \label{kanipasta}Consider the complex:
\begin{align}\label{miito}
\psi_{z_{1},1}\calM\to \Xi_{z_{1}}\calM\oplus \phi_{z_{1},1}\calM\to \psi_{z_{1},1}\calM(-1),
\end{align}
where the first morphisms is $a\oplus \mathrm{can}$ and the second morphism is $b+\mathrm{var}$.
Then, the cohomology in the middle degree of this complex at $\Xi_{z_{1}}\calM\oplus \phi_{z_{1},1}\calM$ is isomorphic to $\calM$.

\item[(iii)] Let $\mathrm{Glue}(X\times \CC^n,D_{1})$ be the category of tuples $(\calM', \calM'', c,v)$, where 
$\calM'$ is a mixed Hodge module on $X\times V_{1}$ which is the restriction of a mixed Hodge module on $X\times \CC^n$, $\calM''$ is a mixed Hodge module on $X\times D_{1}$ and
$c$ (resp. $v$) is a morphism $\psi_{z_{1},1}\calM'\to \calM''$ (resp. $\calM''\to \psi_{z_{1},1}\calM'(-1)$) such that the endomorphism $v\circ c$ of $\psi_{z_{1},1}\calM'$ is the nilpotent endomorphism of $\psi_{z_{1},1}\calM'$.
Then, the functor $\calM\mapsto (\calM|_{X\times V_{1}}, \phi_{z_{1}, 1}\calM, \mathrm{can}, \mathrm{var})$ induces an equivalence of categories between the category of mixed Hodge modules and $\mathrm{Glue}(X\times \CC^n,D_{1})$.
\end{enumerate}

\end{prop}

If $M$ is monodromic,
it is easy to see that if $M^{\varepsilon ,k}$ is also monodromic.
Hence, $\calM^{\varepsilon, k}[*D_{1}]$ and $\calM^{\varepsilon, k}[!D_{1}]$ are also monodromic.
Therefore, $\Xi_{z_{1}}\calM$ is also monodromic.
Then, we have the following.
\begin{cor}\label{patrush}
For a monodromic mixed Hodge module $\calM$ on $X\times\CC^n$,
the Hodge filtration $F_{\bullet}\Xi_{z_{1}}M$ is decomposed as
$F_{p}\Xi_{z_{1}}M=\bigoplus_{\beta\in \RR}F_{p}\Xi_{z_{1}}M\cap (\Xi_{z_{1}}M)^{\beta}$.
\end{cor}
\begin{proof}
By Corollary~\ref{kaniseijin},
$F_{p}M^{\varepsilon ,k}[!D_{1}]$ and $F_{p}M^{\varepsilon ,k}[*D_{1}]$ are decomposed. 
Therefore, $F_{p}\Xi_{z_{1}}M$, i.e. the kernel of the morphism $F_{p}M^{\varepsilon ,k}[!D_{1}]\to F_{p}M^{\varepsilon ,k}[*D_{1}]$ for sufficiently large $k\geq 1$, is also decomposed.
\end{proof}

For the proof of Theorem~\ref{Fdecomp},
another lemma is needed.

\begin{lem}\label{rapp}
Let $M$ be a monodromic $D$-module on $X\times \CC^n$ and $V_{z_{1}}^{\bullet}M$ the Kashiwara-Malgrange filtration along $z_{1}=0$.
For $\gamma\in \RR$ and a section $m\in V_{z_{1}}^{\gamma}M$,
let $m$ be a decomposition $m=\sum_{k=1}^{k_{0}}m_{k}$, where $m_{k}$ is in $M^{\beta_{k}}$ for some $\beta_{k}\in \RR$ with the condition $\beta_{k_{1}}\neq\beta_{k_{2}}$ for $k_{1}\neq k_{2}$.
Then, we have $m_{k}\in V_{z_{1}}^{\gamma}M$ for any $1\leq k\leq k_{0}$.
\end{lem}

\begin{proof}
For each $1\leq k\leq k_{0}$,
let $\delta_{k}\in \RR$ be the biggest number such that $m_{k}\in V^{\delta_{k}}_{z_{1}}M$.
We may assume that $\delta_{1}\leq \delta_{2}\leq \dots \leq \delta_{k_{0}}$.
In particular, we have $m_{k}\in V^{\delta_{1}}_{z_{1}}M$ for any $k$.
If $\delta_{1}\geq \gamma$, the claim is obvious,
so we suppose $\delta_{1}<\gamma$.
For sufficiently large $l_{1}\geq 0$,
$(z_{1}\pa_{z_{1}}-\delta_{1})^{l_{1}}m_{k}$ is in $V^{>\delta_{1}}_{z_{1}}M$ for any $k$.
On the other hand, since $m_{k}$ is in $M^{\beta_{k}}$,
we can take sufficiently large $l_{2}\geq 0$ such that
$(\eu-\beta_{k})^{l_{2}}m_{k}=0$ for any $k$.
Set $\eu':=\sum_{i\geq 2}\ti\pa_{\ti}$.
Then, we have
\[\sum_{j=0}^{l_{2}}C_{l_{2},j}(z_{1}\pa_{z_{1}}-\delta_{1})^{j}(\eu'-(\beta_{k}-\delta_{1}))^{l_{2}-k}m_{k}=0,
\]
where $C_{l_{2},j}$ are the binomial coefficients.
Therefore, for any $k$,
there exists a polynomial $H_{k}$ in $z_{1}\pa_{1},\dots,z_{n}\pa_{n}$
such that
\[
(\eu'-(\beta_{k}-\delta_{1}))^{l_{2}}m_{k}
=H_{k}(z_{1}\pa_{z_{1}}-\delta_{1})m_{k}.
\]
Hence, there exists a polynomial $H'_{k}$ in $z_{1}\pa_{1},\dots,z_{n}\pa_{n}$
such that
\[
(\eu'-(\beta_{k}-\delta_{1}))^{l_{1}l_{2}}m_{k}
=H_{k}'(z_{1}\pa_{z_{1}}-\delta_{1})^{l_{1}}m_{k}.
\]
Therefore, 
the section $[m_{k}]\in \GR^{\delta_{1}}_{V_{z_{1}}}M$ is in 
$(\GR^{\delta_{1}}_{V_{z_{1}}}M)^{\beta_{k}-\delta_{1}}$,
where we set
\[(\GR^{\delta_{1}}_{V_{z_{1}}}M)^{\beta_{k}-\delta_{1}}=\bigcup_{l\geq 1}\Ker(\eu'-(\beta_{k}-\delta_{1}))^l(\subset \GR^{\delta_{1}}_{V_{z_{1}}}M).\]
Moreover, since $\delta_{1}<\gamma$,
we have
\begin{align}\label{inmemory}
    \sum_{i=1}^{k_{0}}[m_{k}]=[m]=0 \quad (\mbox{in $\GR^{\delta_{1}}_{V_{z_{1}}}M$}).
\end{align}
However, it is easy to check (in the same way as in the proof of Proposition~\ref{decomprop}) that $(\GR^{\delta_{1}}_{V_{z_{1}}}M)^{\beta}\cap (\GR^{\delta_{1}}_{V_{z_{1}}}M)^{\beta'}=0
$ for $\beta\neq \beta'$ (this is true even if we do not yet know that $\GR^{\delta_{1}}_{V_{z_{1}}}M$ is not monodromic).
Combining this fact with (\ref{inmemory}),
we have $[m_{k}]=0$ in $\GR^{\delta_{1}}_{V_{z_{1}}}M$.
However, this contradicts with $[m_{1}]\neq 0$.
This completes the proof.
\end{proof}

\begin{cor}\label{nobita}
For a monodromic $D$-module on $X\times \CC^n$ and $\alpha\in (-1,0]$,
the $\alpha$-nearby cycle $\psi_{z_{1},\alpha}M=\GR^{\alpha}_{V_{z_{1}}}M$ and the unipotent vanishing cycle $\phi_{z_{1},1}M=\GR^{-1}_{V_{z_{1}}}M$ are also monodromic on $X\times \CC^{n-1}(=X\times \{z_{1}=0\}\times \CC_{z_{2}}\times \dots \times \CC_{z_{n}})$.
\end{cor}

\begin{proof}
For $\alpha\in [-1,0]$,
let $[m]\in \GR^{\alpha}_{V_{z_{1}}}M$ be a section represented by a section $m\in V^{\alpha}_{z_{1}}M$.
By the previous lemma, we can decompose $m$ as 
\[m=\sum_{k=1}^{k_{0}}m_{k},\]
where $m_{k}$ is in $M^{\beta_{k}}$ for some $\beta_{k}\in \RR$ and $V_{z_{1}}^{\alpha}M$.
Therefore, it is enough to see $[m_{k}]$ is killed by some power of $(\eu'-\delta)$ for some $\delta\in \RR$.
In the same way as in the proof of Lemma~\ref{rapp},
there is an integer $l\geq 0$ and a polynomial $H_{k}$ in 
$z_{1}\pa_{1},\dots,z_{n}\pa_{n}$
such that
\[
(\eu'-(\beta_{k}-\alpha))^{l}m_{k}
=H_{k}'(z_{1}\pa_{z_{1}}-\alpha)m_{k}.
\]
Since $[m_{k}]$ is killed by some power of $z_{1}\pa_{z_{1}}-\alpha$,
this implies that $[m_{k}]$ is in $(\GR^{\alpha}_{V_{z_{1}}}M)^{\beta_{k}-\alpha}$, and the proof is complete.
\end{proof}

\begin{proof}[Proof of Theorem~\ref{Fdecomp}]
If $n=1$, the assertion is true by Propositrion~\ref{rankoneFdecom}.
We use the induction on $n$.
Consider the case $n\geq 2$.
By (ii) of Proposition~\ref{bongore}, $\calM$ is isomorphic to the cohomology in the middle degree of the complex (\ref{miito}).
By Corollary~\ref{patrush} and Corollary~\ref{nobita} with the inductive assumption,
all the terms of (\ref{miito}) are monodromic and the Hodge filtrations are decomposed with respect to the decomposition of the underlying $D$-modules.
Hence, so is its cohomology in the middle degree. 
This completes the proof.
\end{proof}

\section{{T}he Fourier-Laplace transform of a monodromic mixed Hodge module}\label{maboo}

In this section, we consider the Fourier-Laplace transform of a monodromic mixed Hodge module.

\subsection{{T}he Fourier-Laplace transform of a $D$-module}\label{toban}
First, let us recall the notion of the Fourier-Laplace transform of a $D$-module.
We refer to \cite{BryFou}.
Let $X$ be a smooth algebraic variety, $\pi\colon E\to X$ an algebraic vector bundle on $X$, ${\pi}^{\vee}\colon {E^{\vee}}\to X$ the dual vector bundle of $E$ and $\varphi\colon E\times_{X} {E^{\vee}}\to \CC$ the paring between $E$ and ${E^{\vee}}$. 
Moreover, let $\expo^{-\varphi}$ be the integrable connection $(O_{E\times_{X} {E^{\vee}}}, d-d\varphi)$; we regard it as a $D$-module.
We denote by $p$, $q$ the projections $E\times_{X} {E^{\vee}}\to E$ and $E\times_{X} {E^{\vee}}\to {E^{\vee}}$.
For a morphism $f\colon Y\to Z$ between the manifolds $Y$ and $Z$ and a complex of $D$-modules $N_{1}$ (resp. $N_{2}$) on $Y$ (resp. $Z$),
let $f_{\dag}N_{1}$ be the pushforward of $N_{1}$ (which is denoted by $\int_{f}N_{1}$ in \cite{HTT}),
and $f^{\dag}N_{2}$ the pullback of $N_{2}$, which is also expressed as $Lf^{*}N_{2}[\dim{Y}-\dim{Z}]$.
These are objects in the derived categories of $D$-modules.
Recall that $f_{\dag}$ (resp. $f^{\dag}$) corresponds to $Rf_{*}$ (resp. $f^{!}$) under the Riemann-Hilbert correspondence.

\begin{defi}\label{defiFou}
For a $D$-module $M$ on $E$,
we define the $D$-module on $E^{\vee}$ called the Fourier-Laplace transform $\whmunk{M}$ as
\[\whmunk{M}=H^0q_{\dag}(p^{*}M\otimes_{O_{E\times_{X} {E^{\vee}}}}\expo^{-\varphi} ).\]
\end{defi}

It is known that $\whmu{\ \cdot\ }$ defines an exact functor.


\begin{rem}
Since $p$ is a projection,
we have $H^{j}p^{\dag}M=0$ ($j\neq -n$) and
\[
H^{-n}p^{\dag}M=H^0Lp^*M=p^{*}M(=O_{E\times_{X} E^{\vee}}\otimes_{p^{-1}O_{E}}p^{-1}M).\]
\end{rem}

\begin{rem}
For a complex of $D$-modules $M^\bullet$, 
we also define the Fourier-Laplace transform of $\whmunk{(M^\bullet)}$ as
\[\whmunk{(M^\bullet)}=q_{\dag}(p^{*}M^\bullet\otimes_{O_{E\times_{X} {E^{\vee}}}}\expo^{-\varphi} ),\]
which is an object in the derived category of $D$-modules.
Since for a $D$-module (not complex) $M$ the $0$-th cohomology is the only non-trivial cohomology of $q_{\dag}(p^{*}M\otimes_{O_{E\times_{X} {E^{\vee}}}}\expo^{-\varphi})$,
therefore this definition is compatible with Definition~\ref{defiFou}, as identifying a $D$-module with the complex of $D$-modules concentrated in degree $0$.
\end{rem}

Let us consider the projective version of the above definition.
Define $\wt{E}$ (resp. $\wt{E^{\vee}}$) as the projective compactification of $E$ (resp. ${E^{\vee}}$) i.e. the projective bundle of the direct sum of $E$ (resp. $E^{\vee}$) and the trivial bundle over $X$.
We use the same symbol $\pi$ and $\pi^{\vee}$ for their projection to $X$.
Moreover, we denote by $j\colon E\hookrightarrow \wt{E}$ (resp. $j^\vee \colon E^{\vee}\hookrightarrow \wt{E^{\vee}}$) the inclusion of $E$ (resp. $E^{\vee}$) to $\wt{E}$ ($\wt{E^\vee}$) and $D_{\infty}$ (resp. $D_{\infty}^{\vee}$) the divisor $\wt{E}\setminus E$ (resp. $\wt{E^{\vee}}\setminus E^{\vee}$).
We use the same symbol $D_{\infty}$ (resp. $D_{\infty}^{\vee}$) for the divisor $D_{\infty}\times_{X}E^\vee$ (resp. $E\times_{X}D_{\infty}^{\vee}$) of $E\times_{X}E^{\vee}$.
Let $\wt{p}$ (resp. $\wt{q}$) be the projection $\wt{E}\times_{X}\wt{E^{\vee}}\to \wt{E}$ (resp. $\wt{E}\times_{X}\wt{E^{\vee}}\to \wt{E^{\vee}}$) and $\varphi$ the rational function on $\wt{E}\times_{X}\wt{E^{\vee}}$ defined as the pairing of $\wt{E}$ and $\wt{E^{\vee}}$, whose pole divisor is $D_{\infty}\cup D_{\infty}^{\vee}$ (we use the same symbol as $\varphi\colon E\times_{X} E^{\vee}\to \CC$). 
Let $\calE^{-\varphi}$ be the meromorphic connection $(O_{\wt{E}\times_{X}\wt{E^{\vee}}}(*D_{\infty}\cup D_{\infty}^{\vee}), d-d\varphi)$.

\begin{defi}
For a $D$-module $N$ on $\wt{E}$,
we define the Fourier-Laplace transform $\whmunk{N}$ as
\[\whmunk{N}=H^0\wt{q}_{\dag}(\wt{p}^{*}N\otimes \expo^{-\varphi}).\]
\end{defi}

Since our $D$-modules are algebraic,
$\whmunk{N}$ is expressed as follows.

\begin{lem}\label{train}
For a $D$-module $N$ on $\wt{E}$, we have
\[\whmunk{N}\simeq j^{\vee}_{*}({N|_{E}})^{\wedge}.\]
\end{lem}
\begin{proof}
By the definition of $\calE^{-\varphi}$,
we have 
\begin{align*}
    \wt{p}^{\dag}N\otimes \expo^{-\varphi}
    =
    (j\times j^{\vee})_{*}(p^{*}(N|_{E})\otimes \expo^{-\varphi}).
\end{align*}
Therefore, we obtain
\begin{align*}
    \whmunk{N}\simeq& H^0\wt{q}_{\dag}(    (j\times j^{\vee})_{*}(p^{*}(N|_{E})\otimes \expo^{-\varphi}).
)\\
\simeq &
j_{*}^{\vee}H^0{q}_{\dag}(p^{*}(N|_{E})\otimes \expo^{-\varphi})\\
=& j_{*}^{\vee}({N|_{E}})^{\wedge}.
\end{align*}
\end{proof}

Let us consider the case when $X$ is affine and $E$ is trivial, i.e. $E\simeq X\times \CC^n$. 
Let $(z_{1},\dots,z_{n})$ be the standard coordinates of $\CC^n$ (we sometimes write $\CC^n_{z}$ to emphasize the coordinates), $\CC^n_{\zeta}$ the dual vector space of $\CC^n_{z}$, $(\zeta_{1},\dots,\zeta_{n})$ the dual coordinates of $\CC^n_{\zeta}$.
Then, we have ${E^{\vee}}\simeq X\times \CC^{n}_{\zeta}$.
Remark that we can identify a $D$-module $M$ with the $\Gamma(E;D_{E})$-module $\Gamma(E;M)$.
Recall that since $q$ is a projection, the pushforward $q_{\dag}$ is described in terms of the relative de Rham complex (see Proposition~1.5.28 of \cite{HTT}).
The following is well-known. 
\begin{lem}\label{march}
\begin{enumerate}
    \item[(i)]
There is a ring isomorphism $\Gamma({E^{\vee}},D_{{E^{\vee}}})\simeq \Gamma(E;D_{E})$ which sends
$P\in D_{X}$ to the same element $P$
and $\zeta_{i}$ (resp. $\pa_{\zeta_{i}}$) to $\pa_{\ti}$ (resp. $-\ti$). 

\item[(ii)]
For a $D$-module $M$ on $E$,
the Fourier-Laplace transform $\whmunk{M}$ is $M$ as a $\CC$-module
and its $\Gamma({E^{\vee}},D_{{E^{\vee}}})$-module structure is induced from the original $\Gamma(E;D_{E})$-module structure via the isomorphism 
$\Gamma({E^{\vee}},D_{{E^{\vee}}})\simeq \Gamma(E;D_{E})$.
\end{enumerate}
\end{lem}

We will introduce a similar statement for $R$-modules in the next section (Lemma~\ref{walkin}).
We can prove this lemma in the same way as the proof written there.
We write $\whmunk{m}$ for the section of $\whmunk{m}$ corresponding to $m\in M$.
By this lemma, we have 
\begin{align}
\begin{aligned}\label{umiwo}
    \zeta_{i}\cdot\whmunk{m}&=\whmu{\pa_{\ti}m}\quad \mbox{and}\\
    \pa_{\zeta_{i}}\cdot\whmunk{m}&=-\whmu{\ti m}.
\end{aligned}
\end{align}

If we take two trivializations $\varphi_{i}\colon E\simeq X\times \CC^n$ ($i=1,2$) and a section $m\in M$,
the section $\whmunk{m}$ for the trivialization $\varphi_{1}$ (we write $(\whmunk{m})_{1}$ for it) does not coincide with $\whmunk{m}$ for $\varphi_{2}$ (we write $(\whmunk{m})_{2}$ for it), i.e. ``$\whmunk{m}$'' depends on the choice of the trivialization.
However, they are equal up to a multiplicative factor, i.e.
there is a holomorphic function $A(x)\in \Gamma(X;O_{X})(\subset \Gamma(E;O_{E}))$ such that we have
\[(\whmunk{m})_{2}={A(x)}(\whmunk{m})_{1}.\]
Therefore, for an $O_{X}$-submodule $F$ of $\pi_{*}M$,
the $O_{X}$-submodule 
\begin{align}\label{kyousoukyoku}
    \whmunk{F}:=\{\whmunk{m}\in \pi_{*}^{\vee}\whmunk{M}\ |\ m\in F\}
    \end{align}
of $\pi^{\vee}_{*}\whmunk{M}$ does not depend on the choice of the trivialization $E\simeq X\times \CC^n$.
Hence, the following definition is well-defined.

\begin{defi}\label{tetuko}
For a $D$-module $M$ on $E$ ($E$ is not necessary trivial)
and an $O_{X}$-submodule $F$ of $\pi_{*}M$,
we define an $O_{X}$-submodule $\whmunk{F}$ of $\pi_{*}^{\vee}\whmunk{M}$ so that for any local trivialization $\pi^{-1}(U)\simeq U\times \CC^n$ ($U\subset X$ is affine)
we have
\[(\whmunk{F})|_{U}=({F|_{U}})^{\wedge},\]
where the RHS is the one defined by (\ref{kyousoukyoku}).
\end{defi}

\subsection{{T}he Fourier-Laplace transform of a monodromic mixed Hodge module}\label{kanmen}

Next, we consider a monodromic $D$-module $M$ on a (not necessary trivial) vector bundle $E$.
We use the notation defined in the previous subsection.
Recall that $M^{\beta}$ is defined as 
\[M^{\beta}=\bigcup_{l\geq 0}\Ker((\eu-\beta)^l)\subset \pi_{*}M,\]
where $\eu$ is the Euler vector field on $E$.

\begin{prop}\label{fable}
If $M$ is monodromic, then so is $\whmunk{M}$. 
Moreover,
we have
\[(\whmunk{M})^{\beta}=({M^{-\beta-n}})^{\wedge},\]
as $O_{X}$-modules for any $\beta\in \RR$,
where the RHS is defined by Definition~\ref{tetuko}.
\end{prop}

\begin{proof}
We may assume $X$ is affine and $E$ is trivial, i.e. $E\simeq X\times \CC^n_{z}$.
Then, we use the description by Lemma~\ref{march}.
Consider a section $\whmunk{m}\in \whmunk{M}$ for $m\in M^{\beta}$.
We denote by ${\eu^{\vee}}$ the Euler vector field $\sum_{i=1}^{n}\zeta_{i}\pa_{\zeta_{i}}$ on ${E^{\vee}}$.
By (\ref{umiwo}),
we have
\[{\eu^{\vee}}\whmunk{m}=\whmu{(-\eu-n)m}.\]
Therefore, $({\eu^{\vee}}+n+\beta)^l\whmunk{m}$ is zero for some $l\geq 0$.
This implies the assertion.
\end{proof}

If $M$ is a holonomic $D$-module, so is $\whmunk{M}$.
On the other hand,
$\whmunk{M}$ may not be regular in general even if $M$ is regular since $\expo^{-\varphi}$ is not regular at infinity.
Hence, $\whmunk{M}$ may not be the underlying $D$-module of a mixed Hodge module.
In general, it is an underlying $D$-module of a mixed twistor $D$-module (see Subsection~\ref{micromacro}).
Nevertheless, we have the following.

\begin{lem}[Th\'eor\`eme~7.24 of Brylinski~\cite{BryFou}]\label{tobuneko}
If $M$ is monodromic regular holonomic $D$-module, so is $\whmunk{M}$.
\end{lem}

See also Proposition~1.24 of \cite{SaiMon}.
From this lemma, it may be possible to endow $\whmunk{M}$ with a mixed Hodge module structure.
In fact, when $E$ is of rank $1$, we constructed a mixed Hodge module whose underlying $D$-module is $\whmunk{M}$ in Subsection~3.7 of \cite{SaiMon}.
\begin{lem}[Subsection~3.7 of \cite{SaiMon}]\label{soutou}
Let $\calM=(M,F_{\bullet}M,K,W_{\bullet}K)$ be a monodromic mixed Hodge module.
Assume that $E$ is of rank $1$ and $E$ is trivialized by an isomorphism $E\simeq X\times \CC_{z}$.
Then, we can endow $\whmunk{M}$ with a natural mixed Hodge module structure, i.e. we can define a good filtration $F_{p}\whmunk{M}$ of $\whmunk{M}$ and
$(\whmunk{M},F_{\bullet}\whmunk{M})$ is the underlying filtered $D$-module of a mixed Hodge module $\whmunk{\calM}$,
with the following property:
for $\beta\in \RR$ we have
\begin{align}\label{dance}
F_{p}(\whmunk{M})^{\beta}(:=F_{p}\whmunk{M}\cap (\whmunk{M})^{\beta})=\whmu{F_{p+1+\lfloor\beta\rfloor}M^{-\beta-1}},
\end{align}
under the isomorphism $(\whmunk{M})^{\beta}=\whmu{M^{-\beta-1}}$ (Proposition~\ref{fable}).
\end{lem}

Let us recall the idea of this result and describe the Hodge filtration explicitly.
Assume that $E$ is of rank $1$ and $E$ is trivialized by an isomorphism $E\simeq X\times \CC_{z}$.
Let us consider the object in $\mathscr{G}(X)$ (defined in Subsection~\ref{subrankone}):
\begin{align}
    \label{pogy}
((\phi_{z,1}\calM\oplus \psi_{z,\neq 0}\calM, 1\oplus T_{s}^{-1}, \mathrm{can}\circ \mathrm{var}\oplus N), \psi_{z,0}\calM(-1),-\mathrm{var}, \mathrm{can}),
\end{align}
where $T_{s}$ (resp. $N$) is the semisimple part (resp. $\frac{-1}{2\pi\sqrt{-1}}$ times the logarithm of the unipotent part) of the monodromy automorphism.
By Proposition~\ref{hoshinoouji}, we get a mixed Hodge module which will be denoted by $\whmunk{\calM}$ on $X\times \CC_{\zeta}$. 
One can see that the underlying $D$-module of $\whmunk{\calM}$ is $\whmunk{M}$.
We set $\whmunk{\calM}=(\whmunk{M},F_{\bullet}\whmunk{M},\whmunk{K},W_{\bullet}\whmunk{K})$.
The perverse sheaf $\whmunk{K}$ is the Fourier-Sato transform of $K$.
In the setting of Lemma~\ref{soutou}, 
let $M=\bigoplus_{\beta\in \RR}M^{\beta}$ and
$\whmunk{M}=\bigoplus_{\beta\in \RR}(\whmunk{M})^{\beta}$ 
be the decompositions.
By Proposition~\ref{rankoneFdecom},
we have
$F_{\bullet}M=\bigoplus_{\beta\in \RR}F_{\bullet}M^{\beta}$ and
$F_{\bullet}\whmunk{M}=\bigoplus_{\beta\in \RR}F_{\bullet}(\whmunk{M})^{\beta}$.
By Proposition~3.25 of \cite{SaiMon}, 
we have (\ref{dance}) for $\beta\in \RR$.

\begin{rem}
In Proposition~3.25 of \cite{SaiMon},
only (\ref{dance}) for $\beta\in [-1,0]$ is stated.
However, it is easy to verify that (\ref{dance}) holds for any $\beta\in \RR$ by the strict specializability. 
\end{rem}

\begin{rem}
There are other possible mixed Hodge modules whose underlying $D$-modules are $\whmunk{M}$.
In this paper, we always take the one which corresponds to (\ref{pogy})
so that it coincides with the irregular Hodge filtration (see Theorem~\ref{takarajima}).
\end{rem}

\begin{rem}
If we take two trivializations $\varphi_{i}\colon E\simeq X\times \CC_{z}$ ($i=1,2$), we obtain two mixed Hodge modules
$\whmu{H^0(\varphi_{i})_{\dag}\calM}$ by Lemma~\ref{soutou} and get two Hodge module structures on $\whmunk{M}=H^0({}^t\varphi_{i})^\dag\whmu{H^0(\varphi_{i})_{\dag}M}$ ($i=1,2$).
However, one can see that these coincide by using Remark~\ref{sikan} below.
As a consequence, we can generalize Lemma~\ref{soutou} to the case of a (not necessary trivial) line bundle.
\end{rem}

\begin{rem}
The correspondence $\calM\mapsto \whmunk{\calM}$ defines an exact functor between the categories of mixed Hodge modules on the line bundles $E$ and $E^{\vee}$.
Moreover, this induces a functor between their derived categories.
\end{rem}

We will generalize Lemma~\ref{soutou} to the case of a (not necessary trivial) line bundle.
To that end, we will express the Fourier-Laplace transformation on a vector bundle of any rank in terms of the Fourier-Laplace transform on some vector bundle of rank $1$ with some functors between the categories of $D$-modules.
Note that we want to take those that fit the theory of mixed Hodge modules as those functors.
Therefore, for example, as a pullback functor for a morphism $f$, we use $f^{\dag}$ not $f^*$ in the following.

We will use the following lemmas.

\begin{lem}[Corollaire~6.7 of Brylinski~\cite{BryFou}]\label{suza}
Let $E$ and $F$ be a vector bundle over $X$ and $f\colon E\to F$ a morphism of vector bundles.
We denote by ${}^{t}f$ its transpose morphism ${}^{t}f\colon F^{\vee}\to E^{\vee}$ between the dual vector bundles.
Then, for a $D$-module $M$ on $E$
there is a natural isomorphism in the derived category of $D$-modules
\[\whmu{f_{\dag}M}\simeq ({}^tf)^{\dag}\whmunk{M}[n_{F}-n_{E}],\]
where $n_{E} $ (resp. $n_{F}$) is the rank of $E$ (resp. $F$).
\end{lem}

\begin{rem}\label{sikan}
Let $\varphi_{i}\colon E\simeq X\times \CC^n$ ($i=1,2$) be two trivializations of a trivial vector bundle $E$.
Then, $\varphi_{1}\circ \varphi_{2}^{-1}$ is an isomorphism between vector bundles.
The $D$-module structure of $\whmu{H^0(\varphi_{i})_{\dag}M}$ is described by Lemma~\ref{march}.
In this case, 
the $D$-module $\whmu{H^0(\varphi_{1})_{\dag}M}$ is isomorphic to $H^0({}^t(\varphi_{1}\circ \varphi_{2}^{-1}))^{\dag}\whmu{H^0(\varphi_{2})_{\dag}M}$ through the natural morphism in Lemma~\ref{suza}.
\end{rem}

\begin{lem}[Corollaire~6.7 of Brylinski~\cite{BryFou}]\label{donbee}
Let $X$ and $Y$ be smooth algebraic varieties, $f\colon Y\to X$ a morphism and $E$ a vector bundle over $X$.
We denote by $u$ (resp. $u^{\vee}$) the natural morphism from the pullback vector bundle $f^*E$ (resp. $f^{*}E^{\vee}(=(f^{*}E)^{\vee})$) of $E$ (resp. $E^{\vee}$) by $f$ to $E$ (resp. $E^{\vee}$).
Then, for a $D$-module $M$ on $X$ we have a natural isomorphism in the derived category of $D$-modules
\[\whmu{u^{\dag}M}\simeq (u^{\vee})^{\dag}\whmunk{M}.\]
\end{lem}

Let us consider the vector bundles $E\times_{X} E^{\vee}$ and $ \CC\times E^{\vee}$ over $E^{\vee}$.
We define the morphism $\omega$ between these vector bundles as
\begin{align}
\omega\colon E\times_{X} E^{\vee}\to \CC\times E^{\vee}\\
(\bld{z},\bld{\zeta})\to (\langle \bld{z},\bld{\zeta}\rangle,\bld{\zeta}),
\end{align}
where $\bld{z}$ (resp. $\bld{\zeta}$) is a section of $E$ (resp. $E^{\vee}$) and $\langle \bld{z},\bld{\zeta}\rangle$ is the pairing of $\bld{z}$ and $\bld{\zeta}$.
We can regard the projection $p\colon E\times_{X} E^{\vee}\to E$ as the base change of $E$ by the morphism $E^{\vee}\to X$:
\[
\xymatrix{
E\times_{X}E^{\vee}\ar[r]^-{p}\ar[d]& E\ar[d]\\
E^{\vee}\ar[r]& X.
}
\]
Similarly, we can regard the second projection $ E^{\vee}\times_{X}E^{\vee}\to E^{\vee}$ as the base change of $E^{\vee}$ by $E^{\vee}\to X$:
\[
\xymatrix{
E^{\vee}\times_{X}E^{\vee}\ar[r]\ar[d]& E^{\vee}\ar[d]\\
E^{\vee}\ar[r]& X.
}
\]
This morphism $E^{\vee}\times_{X}E^{\vee}\to E^{\vee}$ is $p^{\vee}$.
Let $\iota$ be the inclusion 
\[\iota\colon E^{\vee}\simeq \{1\}\times E^{\vee}\hookrightarrow \CC^{\vee}\times E^{\vee}.\]
Then, we have the following.

\begin{lem}[Proposition~6.11 of Brylinski~\cite{BryFou}]\label{scaret}
For a $D$-module $M$ on $E$ we have a natural isomorphism
\[\whmunk{M}\simeq H^{1}\iota^{\dag}(\whmu{H^{0}\omega_{\dag}H^{-n}p^{\dag}M}).\]
\end{lem}
\begin{proof}

In the derived category of $D$-modules, we have
\begin{align}
\iota^{\dag}\whmu{\omega_{\dag}p^{\dag}M}[1-n]
&\simeq
\iota^{\dag}({}^t\omega)^{\dag}\whmu{p^{\dag}M}\quad (\mbox{by Lemma~\ref{suza}})\notag\\
&\simeq \iota^{\dag}({}^t\omega)^{\dag}(p^{\vee})^{\dag}\whmunk{M}\quad (\mbox{by Lemma~\ref{donbee}})\notag\\
&\simeq \whmunk{M}. \label{kousin}
\end{align}
Note that we have $\whmunk{M}\simeq H^0\whmunk{M}$ and $p^{\dag}M\simeq H^{-n}p^{\dag}M[n]$.
Therefore, the $j$-th cohomology of the complex $\whmu{\omega_{\dag}p^{\dag}M}\simeq ({}^t\omega)^{\dag}\whmu{p^{\dag}M}[n-1]$ is $0$ for $j> -n$,
and hence by (\ref{kousin}) we have
\begin{align*}
\whmunk{M}=&H^0\whmunk{M}\\
\simeq &H^{1-n}\iota^\dag\whmu{\omega_{\dag}p^{\dag}M}\\
\simeq &H^{1}\iota^{\dag}H^{-n}\whmu{\omega_{\dag}p^{\dag}M}\\
\simeq &H^{1}\iota^{\dag}\whmu{H^{0}\omega_{\dag}H^{-n}p^{\dag}M}.
\end{align*}

\end{proof}

By Proposition~\ref{scaret}, the Fourier-Laplace transformation on a vector bundle of any rank can always be expressed in terms of the Fourier-Laplace transform on some vector bundle of rank $1$.
The following lemma was essentially shown in \cite{BryFou}.
For convenience, we present a proof.

\begin{lem}[Brylinski~\cite{BryFou}]\label{borero}
If $M$ is monodromic,
so is $H^{0}\omega_{\dag}H^{-n}p^{\dag}M$ on $\CC\times E^{\vee}$.
In particular, $\whmu{H^{0}\omega_{\dag}H^{-n}p^{\dag}M}$ is monodromic on $\CC^{\vee}\times E^{\vee}$.
\end{lem}

\begin{proof}
It is enough to show the assertion under the assumption that  $E$ is trivial and $X$ is one point: $E\simeq \CC^n$.
We express any object in an algebraic way.
Let $\bld{z}=(z_{1},\dots,z_{n})$ be the coordinates of $\CC^n$ and $\bld{\zeta}=(\zeta_{1},\dots,\zeta_{n})$ its dual coordinates.
Note that we have 
\begin{align*}
    H^{-n}p^{\dag}M\simeq \CC[\bld{z},\bld{\zeta}]\otimes_{\CC[\bld{z}]}M.
\end{align*}
We decompose $\omega$ into 
\begin{align*}
    i_{\omega}\colon E\times E\hookrightarrow& \CC_{s}\times E\times E^{\vee}\\
    (\bld{z},\bld{\zeta})\mapsto& (\langle\bld{z},\bld{\zeta}\rangle,\bld{z},\bld{\zeta})
\end{align*}
and 
\begin{align*}
    p_{\omega}\colon \CC_{s}\times E\times E\to & \CC_{s}\times E^{\vee}\\
    (s,\bld{z},\bld{\zeta})\mapsto& (s,\bld{\zeta}).
\end{align*}
Then, we have
\[H^{0}\omega_{\dag}H^{-n}p^{\dag}M
\simeq H^{0}(p_{\omega})_{\dag}H^{0}(i_{\omega})_{\dag}H^{-n}p^{\dag}M.
\]
Set $N:=H^{0}(i_{\omega})_{\dag}H^{-n}p^{\dag}M$.
Then, we can express $N$ as
\begin{align*}
    N\simeq (\CC[\bld{z},\bld{\zeta}]\otimes_{\CC[\bld{z}]}M)\otimes_{\CC}\CC[\partial_{s}].
\end{align*}
Let $\mathrm{DR}_{\CC_{s}\times E\times E^{\vee}/\CC_{s}\times E^{\vee}}(N)$ be the relative de Rham complex:
\[N\to \Omega^{1}_{\CC_{s}\times E\times E^{\vee}/\CC_{s}\times E^{\vee}}\otimes N\to 
\dots
\to \Omega^{n}_{\CC_{s}\times E\times E^{\vee}/\CC_{s}\times E^{\vee}}\otimes N,
\]
where $\Omega^{i}_{\CC_{s}\times E\times E^{\vee}/\CC_{s}\times E^{\vee}}\otimes N$ is in degree $0$.
Then, we can express $H^{0}(p_{\omega})_{\dag}N$ as
\begin{align*}
    H^{0}(p_{\omega})_{\dag}N=H^{0}(p_{\omega})_{*}(\mathrm{DR}_{\CC_{s}\times E\times E^{\vee}/\CC_{s}\times E^{\vee}}(N)),
\end{align*}
i.e. the cokernel of the morphism
\[\Omega_{\CC_{s}\times E\times E^{\vee}/\CC\times E^{\vee}}^{n-1}\otimes N\to \Omega^{n}_{\CC_{s}\times E\times E^{\vee}/\CC_{s}\times E^{\vee}}\otimes N.\]
Set $\bld{dz}:=dz_{1}\wedge\dots \wedge dz_{n}$.
We fix the isomorphism $\Omega^{n}_{\CC_{s}\times E\times E^{\vee}/\CC_{s}\times E^{\vee}}\simeq O_{\CC_{s}\times E\times E^{\vee}}\bld{dz}$.
Then, a section of $H^{0}(p_{\omega})_{\dag}N$ can be represented by
a sum of some sections in the form:
\[\bld{dz}\otimes f(\bld{\zeta})\otimes m\otimes \pa_{s}^{l},\]
where $f(\bld{\zeta})\in \CC[\bld{\zeta}]$, $m\in M$ and $l\in \ZZ_{\geq 0}$.
Note that a section of $H^{0}(p_{\omega})_{\dag}N$ in the form:
\begin{align*}
    \bld{dz}\otimes\pa_{z_{k}}(f(\bld{\zeta})\otimes m\otimes \pa_{s}^{l})
    =\bld{dz}\otimes f(\bld{\zeta})\otimes \pa_{z_{k}}m\otimes \pa_{s}^{l}+ \bld{dz}\otimes (-\zeta_{k})f(\bld{\zeta})\otimes m\otimes \pa^{l+1}_{s}
\end{align*}
is zero.
Therefore,
we have
\begin{align}\label{wario}
 [ \bld{dz}\otimes f(\bld{\zeta})\otimes \eu m\otimes \pa_{s}^{l}]
= [ \bld{dz}\otimes\langle{\bld{z},\bld{\zeta}}\rangle(f(\bld{\zeta})\otimes  m\otimes \pa_{s}^{l})],  
\end{align}
in $H^{0}(p_{\omega})_{\dag}N$, where $\eu=\sum_{i=1}^{n}z_{i}\pa_{z_{i}}$.
Moreover,
since  
\[(s\pa_{s}+l+1)\pa_{s}^l=\pa_{s}^{l+1}s,\]
we have
\begin{align*}
    (s\pa_{s}+l+1)[\bld{dz}\otimes f(\bld{\zeta})\otimes  m\otimes \pa_{s}^{l}]&=
    \bld{dz}\otimes \langle\bld{z},\bld{\zeta}\rangle (f(\bld{\zeta})\otimes  m\otimes \pa_{s}^{l+1})\\
    &=\bld{dz}\otimes f(\bld{\zeta})\otimes \calE m\otimes \pa_{s}^{l}\quad (\mbox{by (\ref{wario})}).
\end{align*}
By the assumption that $M$ is monodromic,
there exists a polynomial $b(u)\in \CC[u]$ such that $b(\eu)m=0$.
Hence,
we obtain
\begin{align*}
    b(s\pa_{s}+l+1)[\bld{dz}\otimes f(\bld{\zeta})\otimes  m\otimes \pa_{s}^{l}]=&[\bld{dz}\otimes f(\bld{\zeta})\otimes  b(\eu)m\otimes \pa_{s}^{l}]\\
    =&0,
\end{align*}
in $H^{0}(p_{\omega})_{\dag}N$.
We thus conclude that $H^{0}\omega_{\dag}H^{-n}p^{\dag}M$ is monodromic on $\CC_{s}\times E^{\vee}$ and this completes the proof.
\end{proof}

$H^{0}\omega_{\dag}H^{-n}p^{\dag}\calM$ is an object in the category of mixed Hodge modules on a line bundle $\CC\times E^{\vee}$ over $E^{\vee}$,
where
we use the same symbols $H^0\omega_{\dag}$ and $H^{-n}p^{\dag}$ as the functors between the categories of mixed Hodge modules.
Moreover, by Lemma~\ref{borero},
$H^0\omega_{\dag}H^{-n}p^{\dag}\calM$ is a monodromic mixed Hodge module.
Therefore,
by Lemma~\ref{soutou},
we can define a mixed Hodge module $\whmu{H^0\omega_{\dag}H^{-n}p^{\dag}\calM}$
on $\CC^{\vee}\times E^{\vee}$ whose underlying $D$-module is $\whmu{H^0\omega_{\dag}H^{-n}p^{\dag}M}$.
Applying the functor $H^1\iota^{\dag}$ to it,
we obtain a mixed Hodge module $H^1\iota^{\dag}(\whmu{H^0\omega_{\dag}H^{-n}p^{\dag}\calM})$ on $E^{\vee}$ whose underlying $D$-module is
$H^1\iota^{\dag}(\whmu{H^0\omega_{\dag}H^{-n}p^{\dag}M})$.
By Lemma~\ref{scaret}, we have
$\whmunk{M}\simeq H^1\iota^{\dag}(\whmu{H^0\omega_{\dag}H^{-n}p^{\dag}M})$.

\begin{defi}\label{ganbaruzo}
Let $p\colon E\to X$ be a vector bundle whose rank is greater than or equal to $2$ and $\calM$ a monodromic mixed Hodge module on $E$.
Then, we define a mixed Hodge module $\whmunk{\calM}$ whose underlying $D$-module is $\whmunk{M}$ as
\begin{align}\label{letsgogo}
    \whmunk{\calM}:=H^1\iota^{\dag}(\whmu{H^0\omega_{\dag}H^{-n}p^{\dag}\calM})(1).
\end{align}
\end{defi}

\begin{rem}
The Tate twist ``$(1)$'' is needed so that Theorem~\ref{takarajima} below holds.
\end{rem}

It is not easy to compute the Hodge filtration of $\whmunk{M}$ directly from the definition because the pushforward of the mixed Hodge module is complicated object in general.
However, in the next section,
we will compare the Hodge filtration $F_{\bullet}\whmunk{M}$ with the irregular Hodge filtration (see Theorem~\ref{takarajima}),
and by virtue of it, we will get a concrete description of the Hodge filtration of $\whmunk{M}$ (Corollary~\ref{nenmatu}).



\section{{I}rregular Hodge filtrations}
\subsection{{I}rregular Hodge filtrations}\label{subirrH}
As mentioned in the previous section,
the exponentially twisted module $\calE^{-\varphi}$,
in particular the Fourier-Laplace transform of a regular holonomic $D$-module, is not always regular,
and the Fourier-Laplace transform is not always equipped with any mixed Hodge module structure
since the underlying $D$-module of a mixed Hodge module is regular.
Nevertheless, we endowed a natural mixed Hodge module structure on the Fourier-Laplace transform of the underlying $D$-module of a monodromic mixed Hodge module (Lemma~\ref{soutou} and Definition~\ref{ganbaruzo}).
On the other hand, Esnault-Sabbah-Yu~\cite{ESY}, Sabbah-Yu~\cite{SabYuIrrFil} defined a natural filtration called the irregular Hodge filtration on the exponentially twisted module, in particular the Fourier-Laplace transform of the underlying $D$-module of a mixed Hodge module,
which are a generalization of the filtration on the twisted de Rham cohomologies defined by Deligne~\cite{DelIrr}, Yu~\cite{YuIrr} and Sabbah~\cite{SabFoutwo}.
Moreover, Sabbah~\cite{IrrHodge} established the category of irregular Hodge modules as a full subcategory of integrable mixed twistor $D$-modules to handle such filtrations in more functorial way like the theory of mixed Hodge module.
In this section, we will review the irregular Hodge theory briefly.

Let us recall the notion of of $R$-modules.
For details, see \cite{SabPolar}, \cite{MTM} and \cite{MHMProj}.
Let $X$ be a smooth algebraic variety and $p_{\lambda}$ the projection $X\times \CC_{\lambda}\to X$, where $\CC_{\lambda}$ is $\CC$ with the coordinates $\lambda$.
$R_{X}$ denote the sheaf of subalgebras in $D_{X\times \CC_{\lambda}}$ generated by $\lambda p_{\lambda}^*\Theta_{X}$ over $O_{X\times\CC_{\lambda}}$, where $\Theta_{X}$ is the sheaf of vector fields on $X$.
If we identify ${p_{\lambda}}_{*}O_{X\times \CC_{\lam}}$-modules with $O_{X}\otimes_{\CC} \CC[\lam]$-modules,
$R_{X}$ is the sheaf of ring associated to the Rees module $R_{F}D_{X}:=\bigoplus_{p\in \ZZ}F_{p}D_{X}\lambda^{p}\subset D_{X}\otimes_{\CC}\CC[\lam^{\pm 1}]$ of the filtered ring $(D_{X},F_{\bullet}D_{X})$, where $F_{\bullet}D_X$ is the order (with respect to differentials) filtration of $D_{X}$. 
For a local chart $(x_{1},\dots,x_{n})$ of $X$, we set $\eth_{x_{i}}:=\lambda \pa_{x_{i}}$, which is a section of $R_{X}$.
Moreover, we denote by $R_{X}^{\integ}$ the sheaf of subalgebras in $D_{X\times \CC_{\lambda}}$ generated by $R_{X}$ and $\lam^2\pa_{\lam}$.
A $R_{X}^{\integ}$-module $\scrM$ is called an integrable $R_{X}$-module.

\begin{ex}\label{hatou}
Let $(M,F_{\bullet}M)$ be a filtered $D$-module on $X$.
We set 
\[R_{F}M:=\sum_{p\in \ZZ}F_{p}M\lam^{p}\subset M\otimes \CC[\lam^{\pm 1}].\]
Then, $R_{F}M$ is an $R_{X}$-module.
An $R_{X}$-module is called strict if it has no $\CC[\lam]$-torsion.
$R_{F}M$ is strict.
Moreover, $R_{F}M$ has a natural $R_{X}^{\integ}$-module structure, i.e. $R_{F}M$ is an integrable $R_{X}$-module.
Remark that we have
\[R_{F}M/(\lambda-1)R_{F}M\simeq M.\]
\end{ex}

We can generalize the thoery of $D$-modules to the theory of $R$-modules.
For example, we can define the $6$-operations, the Kashiwara-Malgarange filtrations, the nearby-vanishing functors, localizations and Beilinson's gluing, even for $R$-modules, which are denoted by the same symbols as in the theory of $D$-modules (for example, like $f_{\dag}$ and $f^{\dag}$).
In particular, for an $R_{X}$-module $\scrM$ and a divisor $D\subset X$,
we can define a localization (resp. dual localization) $\scrM[*D]$ (resp. $\scrM[!D]$) of $\scrM$ along $D$, which has the same properties as described in Proposition~\ref{localizfac} (see loc. cit.).
Remark that $\scrM[*D]$ is not equal to the naive localization $\scrM(*D)=\scrM\otimes_{O_{X\times \CC_{\lam}}} O_{{X\times \CC_{\lam}}}(*D)$ in general.
If $\scrM$ is the Rees module $R_{F}M$ of a filtered $D$-module $(M,F_{\bullet}M)$ (see example~\ref{hatou}),
$R_{F}M[*D]$ (resp. $R_{F}M[!D]$) coincides with the Rees module of the filtered $D$-module $(M[*D], F_{\bullet}M[*D])$ (resp. $(M[!D], F_{\bullet}M[!D])$).
Moreover, the strict ($\QQ$ or $\RR$-)specializability explained in Definition~\ref{tron} can also be generalized for $R$-modules.
Moreover, we have the notion of holonomicity for $R$-modules.

The category $\mathrm{MTM}^{\mathrm{int}}_{\mathrm{good}}(X;\QQ)$ of integrable mixed twistor $D$-modules with good $\QQ$-structures, introduced by Mochizuki~\cite{MTM}, contains the category $\mathrm{MHM}(X)$ of mixed Hodge modules as a full subcategory,
which is a generalization of the category of pure twistor $D$-modules introduced by Sabbah~\cite{SabPolar} and the category of mixed twistor structures introduced by Simpson~\cite{SimTwistor}.
Moreover, Sabbah~\cite{IrrHodge} defined an abelian full subcategory 
$\mathrm{IrrMHM}(X;\QQ)$, called the category of irregular mixed Hodge modues (with $\QQ$-structures),
of $\mathrm{MTM}^{\mathrm{int}}_{\mathrm{good}}(X;\QQ)$ which contains $\mathrm{MHM}(X)$.
We can write
\[\mathrm{MHM}(X)\subset \mathrm{IrrMHM}(X;\QQ) \subset \mathrm{MTM}^{\mathrm{int}}_{\mathrm{good}}(X;\QQ).\]
We will explain them in a little more detail.

A mixed twistor $D$-module $\scrT\in \mathrm{MTM}^{\mathrm{int}}_{\mathrm{good}}(X;\QQ)$ is a pair of $R^{\integ}_{X}$-modules $\scrM_{1}$, $\scrM_{2}$, a sesqui-linear pairing $C$ of $\scrM_{1}$ and $\scrM_{2}$ and a weight filtration with a $\QQ$-structure satisfying some conditions.
For a mixed Hodge module $\calM=(M,F_{\bullet}M,K,W_{\bullet}K)$,
we can construct a natural mixed twistor $D$-module $\scrT=(\scrM_{1},\scrM_{2},C)\in \mathrm{MTM}^{\mathrm{int}}_{\mathrm{good}}(X;\QQ)$ such that $\scrM_{2}=R_{F}M$ (see Proposition~13.5.4 of \cite{MTM}),
and this construction defines the inclusion $\mathrm{MHM}(X)\subset \mathrm{MTM}^{\mathrm{int}}_{\mathrm{good}}(X;\QQ)$ above (i.e. a fully faithful exact functor $\mathrm{MHM}(X)\hookrightarrow \mathrm{MTM}^{\mathrm{int}}_{\mathrm{good}}(X;\QQ)$).

\begin{rem}
In \cite{MTM}, the ``underlying $R$-module'' of an algebraic mixed twistor $D$-module $\scrT$ on $X$ is an $R(*H)$-module $\wtkai{\scrM}$ on a compactification $\ov{X}$ of $X$, where we set $H:=\ov{X}\setminus X$ (see Definition~14.1.1).
However, since we believe the difference in the terminology will not cause any confusion,
we call $\scrM:={\wtkai{\scrM}}|_{X}$ the underlying $R$-module in this paper.
\end{rem}

\begin{rem}
In the following, we do not consider the weight filtrations and the $\QQ$-structures of mixed twistor $D$-modules.
So, we forget them and treat an object in $\mathrm{MTM}^{\mathrm{int}}_{\mathrm{good}}(X;\QQ)$ as a $R$-triple $(\scrM_{1},\scrM_{2},C)$ with some conditions.
\end{rem}

\begin{notation}
Let $Y$ be another smooth algebraic variety and $f\colon X\to Y$ be a morphism.
In Section~14 of \cite{MTM}, the functors
\begin{align*}
    \TO{f}_{*}, \TO{f}_{!}\colon \DBkai\mathrm{MTM}^{\mathrm{int}}_{\mathrm{good}}(X;\QQ)\to \DBkai\mathrm{MTM}^{\mathrm{int}}_{\mathrm{good}}(Y;\QQ)\\
    \TO{f}^{*}, \TO{f}^{!}\colon \DBkai\mathrm{MTM}^{\mathrm{int}}_{\mathrm{good}}(Y;\QQ)\to \DBkai\mathrm{MTM}^{\mathrm{int}}_{\mathrm{good}}(X;\QQ)
\end{align*}
are defined, each of which is compatible with the corresponding functor in the theory of mixed Hodge modules.
For an underlying integrable $R$-module $\scrM$ of a mixed twistor $D$-module $\scrT$ on $X$,
we denote by $\TO{f}_{*}\scrM$ the underlying complex of integrable $R$-modules of $\TO{f}_{*}\scrT$. 
$\TO{f}_{!}\scrM$, $\TO{g}^{*}\scrN$ and $\TO{g}^{!}\scrN$ are defined in the same way for an underlying $R$-module of a mixed twistor $D$-module $\scrN$ on $Y$.
\end{notation}

\begin{rem}\label{wakeittemo}
For a morphism $f\colon X\to Y$ and an underlying $R$-module of a mixed twistor $D$-module on $X$,
the object $\TO{f}_{*}\scrM$ is not the same ``the $D$-module theoretical pushforward'' $f_{\dag}\scrM$ in general, even though the underlying (complex of) $D$-modules are equal.
If $f$ is projective, we have $\TO{f}_{*}\scrM=f_{\dag}\scrM$.
In the general case,
we first take a smooth variety $\ov{X}$ (resp. $\ov{Y}$) containing $X$ (resp. $Y$) such that $H_{X}:=\ov{X}\setminus X$ (resp. $H_{Y}:=\ov{Y}\setminus Y$) is a divisor,
and a proper morphism $\ov{f}\colon \ov{X}\to \ov{Y}$ which induces $f\colon X\to Y$.
Moreover, let $\wtkai{\scrM}$ be the underlying $R$-module of a mixed twistor $D$-module on $\ov{X}$ whose restriction $\wtkai{\scrM}|_{X}$ is $\scrM$.
Then, $\TO{f}_{*}\scrM$ is expressed as
\[\TO{f}_{*}(\scrM)=\ov{f}_{\dag}(\wtkai{\scrM}[*H_{X}])|_{Y}.\]
The same result holds for $\TO{f}_{!}\scrM$.
Similarly, $\TO{f}^{*}$ is neither ``the scheme theoretical pullback'' $f^{*}$ or ``the $D$-module theoretical pullback'' $f^{\dag}$ in general.
\end{rem}

Next, we review about the rescaling of $R$-modules.
We consider a complex plane $\CC_{\tau}$ with the coordinates $\tau$ and set $\RES{X}:=X\times \CC_{\tau}$ and $\RES{X_{0}}:=X\times \{\tau=0\}$.
Let $j\colon X\times \CC_{\tau}^*\times \CC_{\lambda}\hookrightarrow \RES{X}\times \CC_{\lambda}$ be the inclusion,
$q\colon X\times \CC_{\tau}^*\times \CC_{\lam}\to X\times \CC_{\lam}$ the projection,
and $\mu$ a morphism 
\[\mu\colon X\times \CC_{\tau}^*\times \CC_{\lambda}\to X\times \CC_{\lambda}\quad ((x,\tau,\lambda)\mapsto (x,\lambda/\tau)).\]

Then, for an (algebraic) $R_{X}^{\integ}$-module $\scrM$ (an object on $X\times \CC_{\lambda}$)
we consider the pullback $\mu^*\scrM=O_{X\times \CC_{\tau}^*\times \CC_{\lambda}}\otimes \mu^{-1}\scrM$ as $O$-module and its pushforward (as an algebaric object) $j_{*}\mu^*\scrM$ by $j$.
The object $j_{*}\mu^*\scrM$ is an $O_{X\times \CC_{\tau}\times \CC_{\lam}}(*\{\tau=0\})$-module and denoted by $\RES{\scrM}$.
Remark that in the analytic setting we have to modify the definition a bit not to use the pushforward by the open embedding $j$ (2.2.a of \cite{IrrHodge}), but in the algebraic setting our definition is enough.
The $O_{X\times \CC_{\tau}\times \CC_{\lam}}(*\{\tau=0\})$-module $\RES{\scrM}$ can be endowed with a natural $R^{\integ}_{{\RES{ X}}}(*\{\tau=0\})$-module structure so that
for a section $m\in \scrM$ and a vector field $\theta$ on $X$, we have
\begin{align}
\begin{aligned}\label{wingspan}
\lam(1\otimes m)&=\tau\otimes \lam m,\\
\lam\theta(1\otimes m)&=\tau\otimes \lam\theta m,\\
\eth_{\tau}(1\otimes m)&=-1\otimes \lambda^2\pa_{\lam}m,\quad \mbox{and}\\
\lam^2\pa_{\lam}(1\otimes m)&=\tau\otimes {\lam}^2\pa_{\lam}m.
\end{aligned}
\end{align}
This $R^{\integ}_{{\RES{X}}}(*\{\tau=0\})$-module $\RES{\scrM}$ is called the rescaling of $\scrM$.
We say that $\scrM$ is well-rescalable if the $R^{\integ}_{{\RES{X}}}(*\{\tau=0\})$-module $\RES{\scrM}$ is strictly $\RR$-specializable and regular along $\tau=0$ (Definition~2.19 of \cite{IrrHodge}).

The notions of rescaling and well-rescalablity generalize to $R$-triples and filtered $R$-triples (2.3.d of \cite{IrrHodge}).
Then, the category of irregular mixed Hodge modules $\mathrm{IrrMHM}(X;\QQ)$ (with $\QQ$-structure) 
is defined as the full subcategory of $\mathrm{MTM}^{\mathrm{int}}_{\mathrm{good}}(X;\QQ)$, which consists of graded well-rescalable filtered $R$-triples whose rescaling are also in $\mathrm{MTM}^{\mathrm{int}}_{\mathrm{good}}(\RES{X};\QQ)$ (see Definitions~2.50 and 2.52 of \cite{IrrHodge}).
By Proposition~2.68 of \cite{IrrHodge}, the subcategory $\mathrm{MHM}(X)\subset \mathrm{MTM}^{\mathrm{int}}_{\mathrm{good}}(X;\QQ)$ is contained in $\mathrm{IrrMHM}(X;\QQ)$.

\begin{prop}[Theorem~2.62 and Proposition 2.67 of \cite{SabPolar}] \label{freeway}
The (cohomology of) projective pushforward and smooth pullback preserves the well-rescalability property for filtered $R$-triples. Therefore, each induces the functor between the category of irregular Hodge modules.
More generally,
the same holds for a non-characteristic inverse image
(Proposition~6.64 of \cite{MochiResc}).
\end{prop}

The important fact is that ``the exponential twist'' is contained in $\mathrm{IrrMHM}(X;\QQ)$ although not in $\mathrm{MHM}(X)$, as explained below.
Let $\varphi$ be a rational functioin on $X$ and $P(\subset X)$ its pole divisor.
Then, 
we can define a $R_{X}^{\integ}(*P\times \CC_{\lam})$-module $O_{X\times \CC_{\lam}}(*P\times \CC_{\lam})\cdot e^{\varphi/\lam}$ so that
$O_{X\times \CC_{\lam}}(*P\times \CC_{\lam})\cdot e^{\varphi/\lam}$ is $O_{X\times \CC_{\lam}}(*P\times \CC_{\lam})$ as $O_{X\times \CC_{\lam}}$-module
and
\begin{align*}
\lam\theta \cdot (m e^{\varphi/\lam})&=(\lam\theta m)e^{\varphi/\lam}+\theta(\varphi)e^{\varphi/\lam},\quad \mbox{and}\\
\lam^2\pa_{\lam}\cdot me^{\varphi/\lam}&=-(\varphi m)e^{\varphi/\lam}
\end{align*}
for $m\in O_{X\times \CC_{\lam}}(*P\times \CC_{\lam})$ and a vector field $\theta$ on $X$.
This module is twistor-specializable along $P\times \CC_{\lam}$ and we define $\calE^{\varphi/\lam}_{X}:=( O_{X\times \CC_{\lam}}(*P\times \CC_{\lam})e^{\varphi/\lam})[*(P\times \CC_{\lam})]$ (see Proposition~3.3 of \cite{SabYuIrrFil}) .

\begin{lem}[Proposition~3.3 of \cite{SabYuIrrFil}, Thoerem~0.2 and 2.4.g of \cite{IrrHodge}]\label{tabako}
The object $\calE^{\varphi/\lam}_{X}$ underlies an object of $\mathrm{MTM}^{\mathrm{int}}_{\mathrm{good}}(X;\QQ)$.
More strongly,
this belongs to $\mathrm{IrrMHM}(X;\QQ)$.
Moreover, for an $R^{\integ}$-module $\scrM$ which underlies an object of $\mathrm{MHM}(X)\subset \mathrm{MTM}^{\mathrm{int}}_{\mathrm{good}}(X;\QQ)$, i.e. 
the Rees module of an underlying filtered $D$-module of a mixed Hodge module on $X$,
the $R^{\integ}$-module $\scrM\otimes_{O_{X\times \CC_{\lam}}}\calE^{\varphi/\lam}_{X}$ underlies an object of 
$\mathrm{IrrMHM}(X;\QQ)$.
\end{lem}

For an $R_{X}$-module $\scrM$, we set 
\[\Xi_{\mathrm{DR}}(\scrM):=\scrM/(\lam-1)\scrM\in \mathrm{Mod}(D_{X}),\]
and call it the underlying $D$-module of $\scrM$.
An important feature of well-rescalable good $R^{\integ}_{X}$-module is that we can define a natural good filtration on $\Xi_{\mathrm{DR}}(\scrM)$ called the irregular Hodge filtration.
Let us recall the definition.

Through the identification $X\times \CC^*_{\lam}$ with the image of the diagonal embedding
$X\times \CC^*_{\lam}\hookrightarrow X\times \CC_{\tau}\times \CC_{\lam}$ ($\tau$ is the parameter for the rescaling),
we have
\begin{align}\label{kacyou}
i^*_{\tau=\lam}R_{X\times \CC_{\tau}/\CC_{\tau}}(*\{\tau=0\})=R_{X}\otimes_{\CC[\lambda]}\CC[\lam^{\pm 1}],
\end{align}
where $i_{\tau=\lam}$ is the inclusion $\{\tau=\lam\}\hookrightarrow X\times \CC_{\tau}\times \CC_{\lam}$ and
$R_{X\times \CC_{\tau}/\CC_{\tau}}$ is the subalgebra of $R_{X\times \CC_{\tau}}$ generated by $R_{X}$ and $O_{X\times \CC_{\tau}\times \CC_{\lam}}$ (which does not contain ``$\eth_{\tau}$'').

Let $\pi^{\circ}\colon X\times \CC_{\lam}^*\to X$ be the projection,
$\scrM$ a well-rescalable good $R_{X}^{\integ}$-module and $M$ its underlying $D$-module $\Xi_{\mathrm{DR}}(\scrM)$.

\begin{lem}[Remark~2.20 of \cite{IrrHodge}]\label{effel}
We have an isomorphism between $R_{X}\otimes_{\CC[\lambda]}\CC[\lam^{\pm 1}]$-modules (not $R_{X}^{\integ}$-modules)
\begin{align}\label{kougen}
i_{\tau=\lam}^*\RES{\scrM}\simto {\pi^{\circ}}^*M
\end{align}
which sends
a section $1\otimes (1\otimes m) \in O_{X\times \CC_{\lam}^*}\otimes i_{\tau=\lam}^{-1}\RES{\scrM}$
to 
$1\otimes [m]\in O_{X\times \CC_{\lam}^*}\otimes {\pi^{\circ}}^{-1}M$
where $m$ is a section of $\scrM$,
$1\otimes m$ is the one in $\RES{\scrM}$,
and $[m]$ is the image of $m$ under $\scrM\to \Xi_{\mathrm{DR}}(\scrM)$.
Moreover, 
under the isomorphism~(\ref{kougen}),
the natural action $\lam^2\pa_{\lam}$ on ${\pi^{\circ}}^*M$
corresponds to the action $\lam^2\pa_{\lam}+\tau\eth_{\tau}$ on $i_{\tau=\lam}^*\RES{\scrM}$.
More precisely,
for $k\in \ZZ$ and a section $m\in \scrM$,
the section $1\otimes (\lam^2\pa_{\lam}+\tau\eth_{\tau})(\lam^k\otimes m)$ of $i_{\tau=\lam}^*\RES{\scrM}$
corresponds to
$(\lam^2\pa_{\lam})(\lam^k\otimes m)(=k\lam^{k+1}\otimes m)$ under the isomorphism (\ref{kougen}).
\end{lem}

\begin{proof}
Since $\mu\circ i_{\tau=\lam}\colon X\times \CC_{\lam}\simeq X\times \{\tau=\lam\}\to X\times \CC_{\lam}$ is the morphism $(x,\lam)\to (x,1)$,
by (\ref{kacyou}) we have (\ref{kougen}).
We remark that
$1\otimes (\lam^k\otimes m)$ corresponds to $\lam^k\otimes m$ under (\ref{kougen}).
Then, the second statement follows from the definition of the actions (\ref{wingspan}) of $\RES{\scrM}$.
\end{proof}

${\pi^{\circ}}^*M$ has a natural grading induced by the $\lam$-adic filtration 
\[\lam^{k}O_{X\times \CC_{\lam}^*}\otimes M\subset {\pi^{\circ}}^*M.\]
Here, for brevity we write $\lam^{k}O_{X\times \CC_{\lam}^*}\otimes M$ for $\lam^{k}O_{X\times \CC_{\lam}^*}\otimes {\pi^{\circ}}^{-1}M$.
Then, the $k$-th graded piece is $\GR^{k}{\pi^{\circ}}^*M=\lam^k\otimes M$.
The corresponding graded module is denoted by $\GR({\pi^{\circ}}^*M)(=\bigoplus_{k\in \ZZ}\GR^{k}{\pi^{\circ}}^*M)$, which is $O_{X}[\lam^{\pm 1}]\otimes M$.
We can regard it as a $R_{F}D_{X}$-module.
Remark that by the definition of well-rescalability we can consider the Kashiwara-Malgrange filtration $V^{\bullet}_{\tau}(\RES{\scrM})$ along $\tau=0$ of $\RES{\scrM}$.

\begin{lem}[Lemma~2.21 of \cite{IrrHodge}]
For $\beta\in \RR$,
we have $(\tau-\lam)V^{\beta}_{\tau}(\RES{\scrM})=(\tau-\lam)\RES{\scrM}\cap V_{\tau}^{\beta}(\RES{\scrM})$.
Therefore, we obtain an inclusion
\[i^*_{\tau=\lam}V^{\beta}_{\tau}(\RES{\scrM})\hookrightarrow i^*_{\tau=\lam}\RES{\scrM}.\]
In particular, we can regard
$i^*_{\tau=\lam}V^{\beta}_{\tau}(\RES{\scrM})$ as a submodule of ${\pi^{\circ}}^{*}M$ by Lemma~\ref{effel}.
\end{lem}

For $\alpha\in [0,1)$, the $\lam$-adic filtration $\lam^k\otimes M\subset {\pi^{\circ}}^{*}M$ induces a filtration on $i^*_{\tau=\lam}V^{-\alpha}_{\tau}(\RES{\scrM})$.
The corresponding graded module $\GR(i^*_{\tau=\lam}V^{-\alpha}_{\tau}(\RES{\scrM}))$ is a graded $R_{F}D_{X}$-submodule of $\GR({\pi^{\circ}}^{*}M)=O_{X}[\lam^{\pm 1}]\otimes M$.
Since $\GR({\pi^{\circ}}^{*}M)$ is a strict graded $R_{F}D_{X}$-module,
so is $\GR(i^*_{\tau=\lam}V^{-\alpha}_{\tau}(\RES{\scrM}))$.
Therefore, $\GR(i^*_{\tau=\lam}V^{-\alpha}_{\tau}(\RES{\scrM}))$ comes from the Rees module associated to a filtration of $M$.

\begin{defi}[Definition~2.22 of \cite{IrrHodge}]\label{soten}
For $\alpha\in [0,1)$, the irregular Hodge filtration $\irrF_{\alpha+\bullet}M$ is the unique good filtration of the $D$-module $M$ indexed by $\ZZ$ such that the corresponding Rees module $R_{\irrF_{\alpha+\bullet}}M=\bigoplus_{p\in \ZZ}\irrF_{\alpha+p}M\lam^{p}(\subset M[\lam^{\pm 1}])$ is equal to $\GR(i^*_{\tau=\lam}V^{-\alpha}_{\tau}(\RES{\scrM}))$.
\end{defi}

We can regard the family $\{\irrF_{\alpha+p}M\}_{\alpha\in [0,1), p\in \ZZ}$ as a filtration of $M$ indexed by $\RR$.
If $\scrM$ comes from a filtered $D$-module,
the irregular Hodge filtration is equal to the original one as follows.

\begin{prop}[Proposition~2.40 of \cite{IrrHodge}]
For a filtered $D$-module $(M,F_{\bullet}M)$, 
the corresponding Rees module $\scrM=R_{F}M$ is a well-rescalable good $R_{X}^{\integ}$-module.
Moreover, the irregular Hodge filtration $\irrF_{\bullet}M$ is equal to the original filtration $F_{\bullet}M$.
In particular, $\irrF_{\bullet}M$ jumps only at the integers.
\end{prop}

For an irregular Hodge module $\scrT$,
the underlying $R^{\integ}_{X}$-module $\scrM$ is well-rescalable and good.
So, we can consider the irregular Hodge filtration on $\Xi_{\mathrm{DR}}(\scrM)$.
As already mentioned,
for a mixed Hodge module $\calM=(M,F_{\bullet}M, K, W_{\bullet}M)$,
we can regard it as an irregular Hodge module whose underlying $R^{\integ}_{X}$-module is the Rees module $R_{F}M$.
Therefore, by the proposition above, the irregular Hodge filtration on $\Xi_{\mathrm{DR}}(\scrM)=M$ is the original Hodge filtration $F_{\bullet}M$.

\subsection{{T}he Fourier-Laplace transforms of $R$-modules and the irregular Hodge filtrations}\label{micromacro}
In Section~\ref{maboo},
we introduced the Fourier-Laplace transform of a $D$-module on a vector bundle (or the projective compactification of a vector bundle).
For a monodromic mixed Hodge module, we endowed the Fourier-Laplace transform of its underlying $D$-module with a structure of mixed Hodge module (Definition~\ref{ganbaruzo}). 
As explained there, in general, we can not define ``the Fourier-Laplace transform of a (non-monodromic) mixed Hodge module'' in the category of mixed Hodge modules.
However, 
for a (not necessary monodromic) mixed Hodge module, if we regard it as an integrable mixed twistor $D$-module as explained in Subsection~\ref{subirrH}, we can naturally define
``the Fourier-Laplace transform'' of it in the category of irregular mixed Hodge module.
To explain it, we first recall the definition of the Fourier-Laplace transform of an $R$-module and its basic properties.

The following is the list of references for this subsection.
In \cite{SabPolar}, \cite{SabFouRiem}, \cite{SabMonInftyTwo}, \cite{SabFouOne} and \cite{SabFoutwo},
Sabbah considered the Fourier-Laplace transformation (as an $R$-triples) of a variation of Hodge structure or a twistor $D$-module on a complex line and proved that the $R$-triple is an integrable twistor $D$-module.
In \cite{ESY} and \cite{SabYuIrrFil},
they introduced and studied exponential $R$-module ``$\calE^{f}$'' as a twistor $D$-module.
Moreover, in Dom\'ingez-Reichelt-Sevenheck~\cite{DRS} and Mochizuki~\cite{MochiResc},
they defined (with a slightly different formulation in each paper) the Fourier-Laplace transform of an integrable $R$-module or an integrable mixed twistor $D$-module on $\CC^n$.
The content of this subsection is a review and restatement of these papers, so there are essentially no original contents.


Let $\pi\colon E\to X$ be a vector bundle on a smooth algebraic variety $X$ and $\pi^{\vee}\colon E^{\vee}\to X$ its dual bundle.
We use the notations defined in Section~\ref{maboo} for a vector bundle $E$.
For example, $\wt{E}$ (resp. $\wt{E^{\vee}}$) is the projective compactification of $E$ (resp. ${E^{\vee}}$).
Moreover, $\varphi$ the rational function on $\wt{E}\times_{X}\wt{E^{\vee}}$ defined as the pairing of $\wt{E}$ and $\wt{E^{\vee}}$, whose pole divisor is $D_{\infty}\cup D_{\infty}^{\vee}$. 
Recall that the $R^{\integ}$-modules
$\calE^{\varphi/\lam}_{{E}\times_{X}{E^{\vee}}}$ and
$\calE^{\varphi/\lam}_{\wt{E}\times_{X}\wt{E^{\vee}}}$ are the underlying $R$-modules of mixed twistor $D$-modules (Lemma~\ref{tabako}).

\begin{defi}
For the underlying $R^{\integ}$-module $\scrM$ (resp. ${\scrN}$) of a mixed twistor $D$-module on ${E}$ (resp. $\wt{E}$),
we define the Fourier-Laplace transform $\whmunk{\scrM}$ (resp. $\whmunk{\scrN}$) as
\begin{align}
\notag
\whmunk{\scrM}=&H^0\TO{q}_{*}(p^{*}\scrM\otimes \calE^{-\varphi/\lam}_{{E}\times_{X}{E^{\vee}}})
\\
(\mbox{resp.\ }
\label{kani2}\whmunk{\scrN}=&H^0\TO{\wt{q}}_{*}(\wt{p}^{*}\scrN\otimes \calE^{-\varphi/\lam}_{\wt{E}\times_{X}\wt{E^{\vee}}}[*(D_{\infty}\cup D_{\infty}^{\vee})])
).
\end{align}
\end{defi}

\begin{lem}\label{santa}
Let $\scrN$ be the underlying $R^{\integ}$-module of a mixed twistor $D$-module on $\wt{E}$.
Then, we have
\[(\whmunk{\scrN})|_{E^{\vee}}=\whmu{\scrN|_{E}}.\]
\end{lem}
\begin{proof}
By the definition (see Remark~\ref{wakeittemo}),
we have
\begin{align*}
    \whmu{\scrN|_{E}}&=H^{0}\TO{q}_{*}(p^{*}(\scrN|_{E})\otimes \calE^{-\varphi/\lam})\\
    &=
    H^{0}{\wt{q}}_{\dag}(\wt{p}^{*}\scrN\otimes \calE^{-\varphi/\lam}[*D_{\infty}\cup D_{\infty}^{\vee}])|_{\wt{E^{\vee}}}\\
    &=(\whmunk{\scrN})|_{E^{\vee}}.
\end{align*}
\end{proof}

\begin{cor}\label{hiphop}
In the setting of Lemma~\ref{santa},
we have
\[\whmunk{\scrN}\simeq (j^{\vee}_{\dag}\whmu{\scrN|_{E}})[*D_{\infty}^{\vee}].\]
\end{cor}
\begin{proof}
Since $\scrN$ is the underlying $R$-module of a mixed twistor $D$-module, we have
\begin{align}\label{jingle}(\wt{p}^{*}\scrN\otimes \calE^{-\varphi/\lam})[*(D_{\infty}\cup D_{\infty}^{\vee})]=(\wt{p}^{*}\scrN\otimes \calE^{-\varphi/\lam})[*D_{\infty}][*D_{\infty}^{\vee}].
\end{align}
By Lemma~3.2.12 of \cite{MTM} or Corollary~9.7.1 of \cite{MHMProj},
we have
\begin{align*}
\whmunk{\scrN}=&
H^0\wt{q}_{\dag}((\wt{p}^{*}\scrN\otimes \calE^{-\varphi/\lam})[*D_{\infty}\cup *D_{\infty}^{\vee}])    
\\
=&H^0\wt{q}_{\dag}((\wt{p}^{*}\scrN\otimes \calE^{-\varphi/\lam})[*D_{\infty}\cup D_{\infty}^{\vee}])[*D_{\infty}^{\vee}].   
\end{align*}
Since $H^0\wt{q}_{\dag}((\wt{p}^{*}\scrN\otimes \calE^{-\varphi/\lam})[*D_{\infty}\cup D_{\infty}^{\vee}])|_{E^{\vee}}$ is $\whmu{\scrN|_{E}}$,
we have the desired assertion.
\end{proof}

\begin{rem}
If $\scrM$ (resp. $\scrN$) is the Rees module of a filtered $D$-module $(M,F_{\bullet}M)$ (resp. $(N,F_{\bullet}N)$),
the underlying $D$-module of $\whmunk{\scrM}$ (resp. $\whmunk{\scrN}$) is the Fourier-Laplace transform $\whmunk{M}$ (resp. $\whmunk{N}$) defined in Section~\ref{maboo}.
\end{rem}

By Lemma~\ref{tabako} and Proposition~\ref{freeway},
we obtain the following.
\begin{prop}\label{hyakusyo}
If $\scrM$ (resp. $\scrN$) is the Rees module of the underlying filtered $D$-module of a mixed Hodge module,
$\whmunk{\scrM}$ (resp. $\whmunk{\scrN}$) is the underlying $R^{\integ}$-module of an irregular mixed Hodge module.
\end{prop}

\begin{rem}
For a monodromic mixed Hodge module, we defined a mixed Hodge module whose underlying $D$-module is the Fourier-Laplace transform of its underlying $D$-module (Definition~\ref{ganbaruzo}).
As explained in Lemma~\ref{tabako}, we can regard it as an irregular mixed Hodge module.
On the other hand, we have another ``Fourier-Laplace transform'' made from a monodromic mixed Hodge module, which appeared in Proposition~\ref{hyakusyo}.  
So we have two definitions of ``the Fourier-Laplace transform of a monodromic mixed Hodge module'' in the category of irregular mixed Hodge modules.
In general, the two are different, but they are related to each other.
We will observe it in Subsection~\ref{doronokawa}.
\end{rem}

Let us see $\whmunk{\scrM}$ and $\whmunk{\scrN}$ have better descriptions.
We need the following lemma.

\begin{lem}\label{asyura}
For the underlying $R^{\integ}$-module $\scrN$ on $\wt{E}$ of a mixed twistor $D$-module,
we have
\[(\wt{p}^{*}\scrN\otimes \calE^{-\varphi/\lam})[*D_{\infty}\cup D_{\infty}^{\vee}]
=(\wt{p}^{*}\scrN\otimes \calE^{-\varphi/\lam})(*D_{\infty})[*D_{\infty}^{\vee}].\]
\end{lem}

\begin{proof}
This proof is inspired by the proof of Proposition~A.2.7 of \cite{SabPolar} and the one of Lemma~3.1 of \cite{SabYuIrrFil}.
We assume that $X$ is one point variety.
We can prove in the general case in the same way.
In this case, $E$ and $E^{\vee}$ are vector spaces of rank $n$.
Let $(z_{1},\dots,z_{n})$ be the coordinates of $E$
and $(\zeta_{1},\dots,\zeta_{n})$ its dual coordinates of $E^{\vee}$.
We write $\CC^n_{z}$ (resp. $\CC^n_{\zeta}$) for $E$ (resp. $E^{\vee}$) with the coordinates $(z_{1},\dots,z_{n})$ (resp. $(\zeta_{1},\dots,\zeta_{n})$).
Moreover,
$\PP^n_{z}$ (resp. $\PP^n_{\zeta}$) is the projective compactification of $E=\CC^n_{z}$ (resp. $E=\CC^n_{\zeta}$).
Remark that $D_{\infty}$ (resp. $D_{\infty}^{\vee}$) is the divisor $\PP^n_{z}\setminus \CC^n_{z}$ (resp. $\PP^n_{\zeta}\setminus \CC^n_{\zeta}$).

By the equality (\ref{jingle}),
it is enough to show
\begin{align}\label{denji}(\wt{p}^{*}\scrN\otimes \calE^{-\varphi/\lam})[*D_{\infty}]=(\wt{p}^{*}\scrN\otimes \calE^{-\varphi/\lam})(*D_{\infty}).
\end{align}

Let $[z_{0}:z_{1}:\dots :z_{n}]$ (resp. $[\zeta_{0}:\zeta_{1}:\dots:\zeta_{n}]$) be the homogeneous coordinates of $\Pnt$ (resp. $\Pnx$) and
$U_{i}:=\{z_{i}=1\}(\simeq \CC^n)$ (resp. $U_{j}^{\vee}:=\{\zeta_{j}=1\}$) an open subset of $\Pnt$ (resp. $\Pnx$) with the coordinates $(z_{0},\dots,z_{i-1},z_{i+1},\dots,z_{n})$ (resp. $(\zeta_{0},\dots,\zeta_{j-1},\zeta_{j+1},\dots,\zeta_{n})$) for $i,j=0,\dots, n$.
Then, $\{U_{i}\}_{i}$ (resp. $\{U_{j}^{\vee}\}_{j}$) is a covering of $\Pnt$ (resp. $\Pnx$).
Note that $D_{\infty}$ (resp. $D_{\infty}^{\vee}$) is defined by $z_{0}$ (resp. $\zeta_{0}$).
Therefore, the assertion is clear on $U_{0}\times \Pnx$.
So, it is enough to prove the equality (\ref{denji}) on $U_{1}\times U_{0}^{\vee}$, $U_{1}\times U_{1}^{\vee}$ and $U_{1}\times U_{2}^{\vee}$.
We may assume $n=2$.

\noindent \textit{On $U_{1}\times U_{0}^{\vee}$.}\quad
Let $V^{\bullet}_{z_{0}}(\wt{p}^{*}\scrN\otimes \calE^{-\varphi/\lam})$ be the Kashiwara-Malgarange filtration along $z_{0}$.  
For a section $m\in \scrN|_{U_{1}}$ and $m\otimes 1\in \wt{p}^{*}\scrN\otimes\calE^{-\varphi/\lam}|_{U_{1}\times U_{0}^{\vee}}$,
we have $z_{0}^k(m\times 1)\in V^{>-1}_{z_{0}}(\wt{p}^{*}\scrN\otimes_{\Pnt\times \Pnx}\calE^{-\varphi/\lam})$ for some $k\geq 0$ by an standard property of the $V$-filtration.
Let $k_{0}\geq 0$ be the smallest $k$ and assume $k_{0}\geq 1$. 
Remark that on $U_{1}\times U_{0}^{\vee}$ (with the coordinates $(z_{0},z_{2},\zeta_{1},\zeta_{2})$), we have $\varphi=(1/z_{0})(\zeta_{1}+z_{2}\zeta_{2})$.
Then, we have
\begin{align*}
(\eth_{\zeta_{1}})z^{k_{0}}_{0} ( m\otimes 1)
=&-(1/z_{0})\cdot z_{0}^{k_{0}} ( m\otimes 1)\\
=&-z_{0}^{k_{0}-1}(m\otimes 1).
\end{align*}
Since the operators $\eth_{\zeta_{1}}$ preserves the filtration $V^{\bullet}_{z_{0}}$,
the section $z_{0}^{k_{0}-1}(m\otimes 1)$ is also in $V^{>-1}_{z_{0}}$.
This contradicts the definition of $k_{0}$.
Therefore, we have $k_{0}=0$, i.e.
$m\otimes 1$ is in $V^{>-1}_{z_{0}}(\wt{p}^{*}\scrN\otimes \calE^{-\varphi/\lam})$.
This implies that $V^{\bullet}_{z_{0}}(\wt{p}^{*}\scrN\otimes \calE^{-\varphi/\lam})$ is constant on $U_{1}\times U_{0}^{\vee}$ and
we thus obtain the equality (\ref{denji}).

\noindent \textit{On $U_{1}\times U_{1}^{\vee}$.}\quad 
Similarly to the previous case, for a section $m\in \scrN|_{U_{1}}$,
we take the minimum $k_{0}$ such that $z_{0}^{k_{0}}(m\otimes 1)\in V^{>-1}_{z_{0}}(\wt{p}^{*}\scrN\otimes \calE^{-\varphi/\lam})$ and assume $k_{0}\geq 1$. 
We have $\varphi=(1/(z_{0}\zeta_{0}))(1+z_{2}\zeta_{2})$ on $U_{1}\times U_{1}^{\vee}$.
Then, 
we have
\begin{align*}
(\zeta_{0}^2\eth_{\zeta_{0}}+\zeta_{2}\zeta_{0}\eth_{\zeta_{2}})z^{k_{0}}_{0} ( m\otimes 1)
= z_{0}^{k_{0}-1}(m\otimes 1).
\end{align*}
Therefore, $k_{0}$ is $0$ and thus by the same argument in the previous case, we obtain the equality (\ref{denji}).

\noindent \textit{On $U_{1}\times U_{2}^{\vee}$.}\quad
We can prove it in the same way.

\end{proof}

As mentioned in Remark~\ref{wakeittemo},
$\TO{q}_{*}$ is not $q_{\dag}$ in general.
Nevertheless, the following holds.

\begin{cor}\label{siawase}
For the underlying $R^{\integ}$-module $\scrM$ of a mixed twistor $D$-module on $E$, we have
\[\whmunk{\scrM}=H^{0}q_{\dag}(p^{*}\scrM\otimes \calE^{-\varphi/\lam}).\]
\end{cor}
\begin{proof}
Let $\scrN$ be the underlying $R^{\integ}$-module of a mixed twistor $D$-module on $\wt{E}$ whose restriction $\scrN|_{E}$ is $\scrM$.
Then, by Lemma~\ref{asyura} we have 
\begin{align*}
    \whmunk{\scrM}&=H^{0}\wt{q}_{\dag}(\wt{p}^{*}\scrN\otimes \calE^{-\varphi/\lam}[*D_{\infty}\cup D_{\infty}^{\vee}])|_{E^{\vee}}\\
&=H^{0}\wt{q}_{\dag}(\wt{p}^{*}\scrN\otimes \calE^{-\varphi/\lam}(*D_{\infty})[* D_{\infty}^{\vee}])|_{E^{\vee}}.
\end{align*}
By Lemma~3.2.12 of \cite{MTM} or Corollary~9.7.1 of \cite{MHMProj} again,
the last term is equal to
\begin{align*}
    H^{0}\wt{q}_{\dag}((\wt{p}^{*}\scrN\otimes \calE^{-\varphi/\lam})(*D_{\infty}))[*D_{\infty}^{\vee}]|_{E^{\vee}},
\end{align*}
i.e. 
\begin{align*}\label{diamond}
    H^{0}\wt{q}_{\dag}((\wt{p}^{*}\scrN\otimes \calE^{-\varphi/\lam})(*D_{\infty}))|_{E^{\vee}}.
\end{align*}
This is equal to $H^{0}{q}_{\dag}(({p}^{*}\scrM\otimes \calE^{-\varphi/\lam})$, which completes the proof.
\end{proof}


Finally, let us describe the Fourier-Laplace transform when $X$ is affine and $E$ is trivial.
Fix the trivialization $E\simeq X\times \CC^n$.
As above, we write $\CC^n_{z}$ (resp. $\CC^n_{\zeta}$) for $\CC^n$ with the coordinates $(z_{1},\dots,z_{n})$ (resp. the dual coordinates $(\zeta_{1},\dots,\zeta_{n})$).
Due to Corollary~\ref{hiphop},
in order to know $\whmunk{\scrN}$,
we need to know $\whmu{\scrN|_{E}}$.
Under the assumption above, since $E$ and $E^{\vee}$ are affine,
we can identify the $R^{\integ}$-modules on them with the modules of their global sections.
Therefore, we sometimes write $\whmu{\scrN|_{E}}$ for $\Gamma(E;\whmu{\scrN|_{E}})$.

\begin{lem}\label{walkin}
Let $\scrM$ be the underlying $R^{\integ}$-module of an mixed twistor $D$-module on $E=X\times \CC^n_{z}$ for a smooth affine variety $X$.
Then, as a $\CC[\lam]$-module,
$\Gamma(X\times \CC^n_{\zeta}; \whmunk{\scrM})$ is isomorphic to $\Gamma(X\times\CC^n_{z};\scrM)$
and under this identification,
the vector field $\theta$ on $X$ acts as the same $\theta$,
$\zeta_{i}$ acts as $\eth_{z_{i}}$
and $\eth_{\zeta_{i}}$ acts as $-z_{i}$.
Moreover, $\lam^2\pa_{\lam}$ acts as $\lam^2\pa_{\lam}+\lam\eu_{\Cnt}$,
where $\eu_{\Cnt}=\sum_{i=1}^nz_{i}\pa_{z_{i}}$.
\end{lem}

\begin{proof}
We may assume $X$ is one point variety, i.e. $E=\CC^n_{z}$
We set 
$\scrM_{1}:={p}^{*}\scrM\otimes \calE^{-\varphi/\lam}$ and
$\scrA^k:=\Omega_{\Cnt\times \Cnx/\Cnx}^k\otimes \CC[\lam]\lam^{-k}$, where $\Omega_{\Cnt\times \Cnx/\Cnx}^k$ is the sheaf of relative holomorphic $k$-forms.
Then, the $R$-module structure of $\scrM_{1}$ defines the connection
$\scrM_{1}\to \scrA^1\otimes \scrM_{1}$ ($m\mapsto \sum_{i=1}^{n}dz_{i}/\lam\otimes \eth_{z_{i}}m$).
This morphism is naturally extended to the relative de Rham complex
\[\scrM_{1}\to \scrA^1\otimes \scrM_{1} \to \dots\to \scrA^n\otimes \scrM_{1},\]
where the rightmost term is of degree $0$.
Fixing the isomorphisms ${\scrA^1}\simeq \bigoplus_{i=1}^{n}O_{\Cnt\times \Cnx}dz_{i}/\lam$ and $\scrA^k\simeq \bigwedge^k(\bigoplus_{i=1}^{n}O_{\Cnt\times \Cnx}dz_{i}/\lam)$,
one can see that this complex is the Koszul complex of the $R$-module $\scrM$ with respect to the regular sequence $\eth_{z_{1}},\dots,\eth_{z_{n}}$.
Therefore, the only non-trivial cohomology is the $0$-th one and
that is
\[\scrM_{1}/\sum_{i=1}^{n}\eth_{i}\scrM_{1}\]
by the identification $\scrA^n\simeq O_{\Cnt\times \Cnx}dz_{1}\wedge\dots \wedge dz_{n}/\lam^n$.
The pushforward of an $R$-module by a projection can be expressed as a relative de Rham cohomology.
Therefore, by Corollary~\ref{siawase} we have
\begin{align*}
\whmunk{\scrM}\simeq& H^0q_{*}(\scrM_{1}\to \scrA^1\otimes \scrM_{1} \to \dots\to \scrA^n\otimes \scrM_{1})\\
\simeq& q_{*}(\scrM_{1}/\sum_{i=1}^{n}\eth_{i}\scrM_{1})
\end{align*}
in the category of $R^{\integ}$-modules.
Taking the global sections of $q_{*}(\scrM_{1}/\sum_{i=1}^{n}\eth_{i}\scrM_{1})$,
since we have $\Gamma(\Cnt\times\Cnx;\scrM_{1})=\Gamma(\Cnt;\scrM)[\zeta]$ (with the twisted actions by $\calE^{-\varphi/\lam}$),
we obtain
\[\Gamma(\Cnx;\whmunk{\scrM})\simeq \Gamma(\Cnt;\scrM)[\zeta]/\sum_{i=1}^n\eth_{z_{i}}\Gamma(\Cnt;\scrM)[\zeta].\]
By looking at the actions on $\Gamma(\Cnt;\scrM)[\zeta]$,
we thus get the second assertion.
\end{proof}

Combining Lemma~\ref{walkin} and Corollary~\ref{hiphop}, 
we now understand the $R^{\integ}$-module structure of $\whmunk{\scrN}$.

\begin{defi}\label{konpeitou}
In the setting of Lemma~\ref{walkin},
for a section $m\in \Gamma(X\times \CC^n_{z};{\scrM})$
we denote by $\whmunk{m}$ the corresponding section of $\whmunk{\scrM}$ under the isomorphism $\Gamma(X\times \CC^n_{z};\scrM)\simeq \Gamma(X\times \CC_{\zeta}^n;\whmunk{\scrM})$.
\end{defi}

In terms of the terminology in the proof of Lemma~\ref{walkin}, $\whmunk{m}$ is the class represented by a section $(dz_{1}\wedge \dots \wedge dz_{n}/\lam^n)\otimes m\in \scrA^n\otimes \scrM_{1}$.
Then, by Lemma~\ref{walkin}, we have
\begin{align}
\begin{aligned}\label{gideon}
\lam\cdot\whmunk{m}=&\whmu{\lam m},\\
\zeta_{i}\cdot\whmunk{m}=&\whmu{\eth_{z_{i}}m},\\
\eth_{\zeta_{i}}\cdot\whmunk{m}=&-\whmu{z_{i}m},\quad \mbox{and}\\
\lam^2\pa_{\lam}\cdot\whmunk{m}=&\whmu{(\lam^2\pa_{\lam}+\lam\eu_{\Cnt})m}.
\end{aligned}
\end{align}

\begin{rem}
The section $\whmunk{m}$ depends on the choice of the trivialization of $E$.
However, similarly to Definition~\ref{tetuko},
we can define an $O_{X\times \CC_{\lam}}$-module $\whmunk{\mathscr{F}}$ of $\pi_{*}^{\vee}\whmunk{\scrM}$ for an $O_{X\times \CC_{\lam}}$-submodule $\mathscr{F}$ of $\pi_{*}\scrM$.
\end{rem}

\subsection{Fourier-Laplace transforms of a monodromic mixed Hodge modules} \label{otukaresamadesu}
In this subsection, we will compute the irregular Hodge filtration of the Fourier-Laplace transform of a monodromic mixed Hodge module.
To simplify the description, we will consider the Fourier-Laplace transforms of mixed Hodge modules on $\CC^n_{z}$ (with the coordinates $(z_{1},\dots,z_{n})$).
However, we can generalize the results to the Fourier-Laplace transform on a general vector bundle (see Remark~\ref{pikmin}).

We will use the notation defined in the previous subsection.
Let $\calM=(M,F_{\bullet}M,K,W_{\bullet}K)$ be a mixed Hodge module on $\Cnt$ and $\calN=(N,F_{\bullet}N,K',W_{\bullet}K')$ the pushforward of $\calM$ by the inclusion $j\colon \Cnt\hookrightarrow \Pnt$ (since our module is always algebraic, we can consider such an object).
We denote by $\scrM$ (resp. $\scrN$) the corresponding $R^{\integ}_{\Cnt}$-module (resp. $R^{\integ}_{\Pnt}$) of $(M,F_{\bullet}M)$ (resp. $(N,F_{\bullet}N)$) (see Example~\ref{hatou}). 
Remark that we have $\whmunk{\scrN}|_{\Cnx}=\whmunk{\scrM}$ and $\whmunk{N}|_{\Cnx}=\whmunk{M}$ (see Corollary~\ref{hiphop}). 
By Proposition~\ref{hyakusyo},
$\whmunk{\scrN}$ is also the underlying $R^{\integ}$-module of an irregular mixed Hodge module. 
Then, as explained in subsection~\ref{subirrH},
$\Xi_{\mathrm{DR}}(\whmunk{\scrN})=\whmunk{N}$ is equipped with the irregular Hodge filtration $\irrF_{\bullet}\whmunk{N}$ (Definition~\ref{soten}).

In the following, we assume that $M$ is monodromic.
Then, by Proposition~\ref{decomprop} and Theorem~\ref{Fdecomp} we have the decompositions 
\begin{align*}
M=\bigoplus_{\beta\in \RR}M^{\beta}\quad \mbox{and}\quad F_{\bullet}M=\bigoplus_{\beta\in \RR}F_{\bullet}M^{\beta},
\end{align*}
where $M^{\beta}=\bigcup_{l\geq 0}\Ker(\eu_{\Cnt}-\beta)^l$.
Therefore, we have 
\[\scrM=\bigoplus_{\substack{\beta\in \RR\\ p\in \ZZ}}F_{p}M^{\beta}\lam^{p}.\]
$\CC^n_{\zeta}$ is the dual space of $\CC^n_{z}$ with the dual coordinates $(\zeta_{1},\dots,\zeta_{n})$.
Let $[\zeta_{0}:\dots:\zeta_{n}]$ be the homogeneous coordinates of $\Pnx(=\PP(\CC^n_{\zeta}))$ and $\{U_{i}^{\vee}\}_{i=0}^n$ ($U_{i}^{\vee}=\{\zeta_{i}\neq 0\}(\simeq \CC^n)\subset \Pnx$) the affine open covering of $\Pnx$.
Note that $U_{0}^{\vee}=\Cnx$.
To understand the irregular Hodge filtration $\irrF_{\bullet}\whmunk{N}$,
we will compute the restriction of $\irrF_{\bullet}\whmunk{N}$ to each affine open subset $U_{i}^{\vee}$ respectively.

\subsubsection{{The irregular Hodge filtration on $\whmunk{M}$}}
First, we compute $\irrF_{\bullet}\whmunk{N}|_{\Cnx}(=\irrF_{\bullet}\whmunk{M})$.
In order to do that, we need to compute $V_{\tau}^{-\alpha}(\RES{(\whmunk{\scrM}}))$, 
where $\RES{(\whmunk{\scrM})}$ is the rescaled module of $\whmunk{\scrM}$ and $\tau$ the rescaling parameter (see Subsection~\ref{subirrH}).
Since $\Cnx\times \CC_{\tau}\times \CC_{\lam}$ is affine, 
we identify the sheaves on it with the modules of global sections of them and they are represented by the same symbol by abuse of notation.
Then, $\RES{(\whmunk{\scrM})}$ is $O_{\Cnx}[\tau^{\pm 1},\lam]\otimes_{O_{\Cnx}[\lam]} \whmunk{\scrM}$ as an $O_{\Cnx\times \CC_{\tau}\times \CC_{\lam}}(*\{\tau=0\})$-module with an $R^{\integ}$-module action defined as (\ref{wingspan}).
Moreover, recall that (a global section of) $\whmunk{\scrM}$ can be expressed as $\whmunk{m}$ for $m\in \scrM$ (see Definition~\ref{konpeitou}).

\begin{lem}\label{golgo}
In this setting, we have
\begin{align}\label{yoccyan}
V_{\tau}^{\gamma}(\RES{(\whmunk{\scrM})})=\bigoplus_{\substack{i,p\in \ZZ, \beta\in \RR\\i-p-\beta\geq \gamma}}\tau^{i}\otimes \whmu{F_{p}M^{\beta}\lam^{p}}.
\end{align}
\end{lem}

\begin{proof}
We remark that since $\RES{(\whmunk{\scrM})}$ is strictly $\RR$-specializable along $\tau$ (by the definitions of the well-rescalability and the irregular Hodge module),
Kashiwara-Malgrange filtration $V_{\tau}^{\bullet}(\RES{(\whmunk{\scrM})})$ exists and each graded piece $\GR_{V}^{\gamma}(\RES{(\whmunk{\scrM})})$ is strict.
Therefore, if a non-zero section $s\in \RES{(\whmunk{\scrM})}$ is killed by $(\tau\eth_{\tau}-\gamma\lam)^l$ for some $l\geq 0$,
the section $s$ is in $V^{\gamma}_{\tau}(\RES{(\whmunk{\scrM})})$ and
not in $V^{>\gamma}_{\tau}(\RES{(\whmunk{\scrM})})$.

Let $m$ be a section of $F_{p}M^{\beta}$.
Then, $m\lam^p$ is in $\scrM$.
By (\ref{gideon}),
we have
\[\lam^2\pa_{\lam}\whmu{m\lam^p}=\whmu{(\eu_{\Cnt}+p)m\lam^{p+1}}.\]
Therefore, by (\ref{wingspan}), for $i\geq 0$ we have
\begin{align}
\tau\eth_{\tau}(\tau^i\otimes \whmu{m\lam^p})
&=(i\lam\tau^i+\tau^{i+1}\eth_{\tau})(1\otimes \whmu{m\lam^p})\notag\\
&=i\tau^{i+1}\otimes \whmu{m\lam^{p+1}}
-\tau^{i+1}\otimes \lam^2\pa_{\lam}(\whmu{m\lam^p})\notag\\
&=
i\tau^{i+1}\otimes \whmu{m\lam^{p+1}}
-\tau^{i+1}\otimes \whmu{(p+\eu_{\Cnt}) m\lam^{p+1}}.
\end{align}
Hence, we have
\[(\tau\eth_{\tau}-(i-p-\beta)\lam)(\tau^i\otimes \whmu{m\lam^p})
=-\tau^{i+1}\otimes \whmu{(\eu_{\Cnt}-\beta)m\lam^{p+1}}.\]
By induction, for $l\geq 1$ we obtain
\[(\tau\eth_{\tau}-(i-p-\beta)\lam)^l(1\otimes \whmu{m\lam^p})=(-1)^l\tau^{i+l}\otimes \whmu{(\eu_{\Cnt}-\beta)^lm\lam^{p+l}}.\]
Therefore, $\tau^i\otimes \whmu{m\lam^p}\in \RES{(\whmunk{\scrM})}$ is killed by 
$(\tau\eth_{\tau}-(i-p-\beta)\lam)^l$ for sufficiently large $l\geq 0$
and hence $\tau^i\otimes \whmu{m\lam^p}$ is in $V^{i-p-\beta}_{\tau}(\RES{(\whmunk{\scrM})})$ for the reason stated at the beginning of this proof.
We thus conclude that the RHS of (\ref{yoccyan}) is contained in the LHS.

Any section $s\in V^{\gamma}_{\tau}(\RES{(\whmunk{\scrM})})$ is a sum of some sections $\tau^{i}\otimes \whmu{m\lam^p}$ for some $i\in \NN$, $p\in \ZZ$ and $m\in F_{p}M^{\beta}$.
Let $s=s_{\gamma_{1}}+\dots+s_{\gamma_{k}}$ be the decomposition of $s$ such that $s_{\gamma_{j}}(\neq 0)\in \sum_{i-p-\beta=\gamma_{j}}\tau^i\otimes \whmu{F_{p}M^{\beta}\lam^p}$ and $\gamma_{1}\leq \dots \leq \gamma_{k}$.
As we proved, $s_{\gamma_{1}}$ is in $V^{\gamma_{1}}_{\tau}(\RES{(\whmunk{\scrM})})$ and not in $V^{>\gamma_{1}}_{\tau}(\RES{(\whmunk{\scrM})})$.
Hence, $\gamma_{1}$ is greater than $\gamma$.
This implies that the LHS of (\ref{yoccyan}) is contained in the RHS.

This completes the proof.
\end{proof}

Recall that
for $\alpha\in [0,1)$
we can regard $\GR(i_{\tau=\lam}^{*}V^{-\alpha}_{\tau}(\RES{(\whmunk{\scrM})}))$ as a submodule of
$\GR(i_{\tau=\lam}^{*}(\RES{(\whmunk{\scrM})}))=\GR({\pi^{\circ}}^{*}\whmunk{M})\simeq O_{\Cnx}[\lam^{\pm 1}]\otimes_{O_{\Cnx}} \whmunk{M}$ (see Subsection~\ref{subirrH}).
Note that the isomorphism
\begin{align}\label{orutega}
\GR(i_{\tau=\lam}^{*}(\RES{(\whmunk{\scrM})}))\simeq O_{\Cnx}[\lam^{\pm 1}]\otimes_{O_{\Cnx}} \whmunk{M}
\end{align}
is defined so that (a section of the LHS represented by) the section $1\otimes \whmu{m\lam^p}\in \RES{(\whmunk{\scrM})}$ corresponds to $1\otimes \whmunk{m}$.
Moreover, the Rees module of the irregular Hodge filtration $\irrF_{\alpha+\bullet}\whmunk{M}$ is equal to $\GR(i_{\tau=\lam}^{*}V^{-\alpha}_{\tau}(\RES{(\whmunk{\scrM})}))$ (see Definition~\ref{soten}). 
Therefore, in order to know $\irrF_{\alpha+\bullet}\whmunk{M}$ it remains to see the image of 
$\GR(i_{\tau=\lam}^{*}V^{-\alpha}_{\tau}(\RES{(\whmunk{\scrM})}))$ under the isomorphism (\ref{orutega}).

\begin{thm}\label{prosyuto}
For $\alpha\in [0,1)$ and $\beta\in \RR$,
we have
\[\irrF_{\alpha+p}\whmunk{M}=\bigoplus_{\beta\in \RR}\whmu{F_{p+\lfloor\alpha-\beta\rfloor}M^{\beta}}.\]
\end{thm}

\begin{proof}
By Lemma~\ref{golgo},
we have
\begin{align*}
i_{\tau=\lam}^{*}V^{-\alpha}_{\tau}(\RES{(\whmunk{\scrM})})
=
i_{\tau=\lam}^{*}(\bigoplus_{\substack{i,p\in \ZZ, \beta\in \RR\\ i-p-\beta\geq -\alpha}}\tau^i\otimes \whmu{F_{p}M^{\beta}\lam^p}).
\end{align*}
Under the identification (\ref{orutega}),
the RHS is (as a subset of $O_{\Cnx}[\lam^{\pm 1}]\otimes \whmunk{M}$)
\begin{align}\label{smasher}
\bigoplus_{i\in \ZZ}\sum_{\substack{p\in \ZZ, \beta\in \RR\\ i-p-\beta\geq -\alpha}}\lam^i\otimes \whmu{F_{p}M^{\beta}}.
\end{align}
Note that the condition $i-p-\beta\geq -\alpha$ is equivalent to
\[i+\lfloor\alpha-\beta\rfloor\geq p,\]
where $\lfloor\alpha-\beta\rfloor$ is the largest integer less than or equal to $\alpha-\beta$.
Therefore, (\ref{smasher}) is equal to
\[\bigoplus_{{i\in \ZZ, \beta\in \RR}}\lam^i\otimes \whmu{F_{i+\lfloor\alpha-\beta\rfloor}M^{\beta}}.\]
This implies the desired result.
\end{proof}

\subsubsection{The irregular Hodge filtration of $\whmunk{M}$ at infinity}
Next, let us consider the irregular Hodge filtration on $\whmunk{N}|_{U_{i}^{\vee}}$ for $i=1,\dots,n$.
Since they can all be computed in the same way, we will consider the case where $i=n$.
In this subsection, we assume that $n\geq 2$.
However this assumption is not essential; the argument proceeds in exactly the same way also for the case $n=1$.
Actually, all the results hold also in that case (after changing the notations appropriately).
Let $(\zeta_{0}',\zeta_{1}',\dots,\zeta_{n-1}')$ be the coordinates of $\UV_{n}(\simeq \CC^n)$
so that the point of $\Pnx$ corresponding to $(\zeta_{0}',\zeta_{1}',\dots,\zeta_{n-1}')$ is $[\zeta_{0}':\zeta_{1}':\dots:\zeta_{n-1}':1]$.
Then, 
we have
$\UV_0\cap \UV_n=\{\zeta_{0}'\neq 0\}=\{\zeta_{n}\neq 0\}$, $D_{\infty}^{\vee}\cap \UV_n=\{\zeta_{0}'=0\}$,
$\zeta_{0}'=1/\zeta_{n}$ and $\zeta_{i}'=\zeta_{i}/\zeta_{n}$ for $i=1,\dots,n-1$ on $\UV_0\cap \UV_n$.
Moreover, we have 
\begin{align}
\eu_{\UV_0}(=\sum_{i=1}^n\zeta_{i}\pa_{\zeta_{i}})=-\zeta_{0}'\pa_{\zeta_{0}'}.
\end{align}

Let $j_{n}^{\vee}$ be the inclusion $\UV_0\cap \UV_n\hookrightarrow \UV_n$.
By Lemma~\ref{hiphop},
we have
\begin{align}\label{macd}
(\whmunk{\scrN})|_{\UV_n}\simeq (j^{\vee}_{n})_{\dag}((\whmunk{\scrM})|_{\UV_{0}\cap \UV_n})[*\{\zeta_{0}'=0\}].
\end{align}
Since $\UV_{n}\cap \UV_{0}$ (resp. $\UV_{0}$) is affine, 
we may identify $(\whmunk{\scrM})|_{\UV_{0}\cap \UV_n}$ (resp. $(j^{\vee}_{n})_{\dag}((\whmunk{\scrM})|_{\UV_{0}\cap \UV_n})$) 
with the module of its global sections
and regard it as a $\Gamma(\UV_{0}\cap \UV_n; R_{\UV_{0}\cap \UV_n})$-module
(resp. a $\Gamma(\UV_n; R_{\UV_n})$-module).
Then,
we can write $(j^{\vee}_{n})_{\dag}(\whmunk{\scrM}|_{\UV_{0}\cap \UV_n})$ in an algebraic way as
\[\whmunk{\scrM}\otimes_{\CC[\zeta_{n}]}\CC[\zeta_{n}^{\pm 1}].\]
Moreover, we can write the RHS of (\ref{macd}) as
\[\whmunk{\scrM}\otimes_{\CC[\zeta_{n}]}\CC[\zeta_{n}^{\pm 1}][*\{\xizp=0\}].\]
Its underlying $D$-module is $\whmunk{M}\otimes_{\CC[\zetn]}\CC[\zetn^{\pm 1}]$.
Remark that 
the underlying $D$-module of $(\whmunk{\scrM})|_{\UV_{0}\cap \UV_n}$
is also expressed as the same $\whmunk{M}\otimes_{\CC[\zetn]}\CC[\zetn^{\pm 1}]$ (under the identification of the sheaf of module and the module of its global sections).
However, in the following, we always regard $\whmunk{M}\otimes_{\CC[\zetn]}\CC[\zetn^{\pm 1}]$ as the underlying $D$-module of $(j^{\vee}_{n})_{\dag}((\whmunk{\scrM})|_{\UV_{0}\cap \UV_n})$,
i.e. $\whmunk{M}\otimes_{\CC[\zetn]}\CC[\zetn^{\pm 1}]$ is a $D$-module on $\UV_{n}$ (not $\UV_{0}\cap \UV_{n}$).
Moreover, for a section $m\in M$,
the section $\whmunk{m}\otimes 1\in \whmunk{M}\otimes_{\CC[\zetn]}\CC[\zetn^{\pm 1}]$ is simply denoted by $\whmunk{m}$ if there is no confusion. 

\begin{lem}\label{kinomi}
If we regard $\UV_n$ as a trivial line bundle $\CC_{\zeta_{0}'}\times (\CC_{\zeta_{1}'}\times\dots \times \CC_{\zeta_{n-1}'})$ over $(\CC_{\zeta_{1}'}\times \dots \times \CC_{\zeta_{n-1}'})$,
the $D$-module $\whmunk{M}\otimes_{\CC[\zetn]}\CC[\zetn^{\pm 1}]$ is monodromic on this line bundle, i.e.
we have
\[\whmunk{M}\otimes_{\CC[\zetn]}\CC[\zetn^{\pm 1}]=\bigoplus_{\beta\in \RR}(\whmunk{M}\otimes_{\CC[\zetn]}\CC[\zetn^{\pm 1}])^{\beta}_{\xizp},\]
where we set 
\[(\whmunk{M}\otimes_{\CC[\zetn]}\CC[\zetn^{\pm 1}])^{\beta}_{\zeta'_{0}}:=\bigcup_{l\geq 0}\Ker(\zeta_{0}'\pa_{\zeta_{0}'}-\beta)^{l}(\subset \whmunk{M}\otimes_{\CC[\zetn]}\CC[\zetn^{\pm 1}]).\]
In particular,
for $\beta\in \RR$ we have
\[\GR^{\beta}_{V}(\whmunk{M}\otimes_{\CC[\zetn]}\CC[\zetn^{\pm 1}])\simeq (\whmunk{M}\otimes_{\CC[\zetn]}\CC[\zetn^{\pm 1}])^{\beta}_{\xizp},\]
as $\Gamma(\CC_{\zeta_{1}'}\times\dots \times \CC_{\zeta_{n-1}'};O)$-modules,
where $\GR^{\beta}_{V}$ is the graded piece of the Kashiwara-Malgrange filtration $V^{\bullet}_{\xizp}$ along $\xizp=0$.
Moreover, for $\beta\in \RR$ we have
\begin{align}\label{ohaka}
(\whmunk{M}\otimes_{\CC[\zetn]}\CC[\zetn^{\pm 1}])^{\beta}_{\zeta'_{0}}=\sum_{\substack{j\in \ZZ,\gamma\in \RR\\ j+\gamma+n=\beta}}{\zeta_{0}'}^{j}\whmu{M^{\gamma}},
\end{align}
where 
${\zeta_{0}'}^{j}\whmu{M^{\gamma}}$ is the subset of $\whmunk{M}\otimes_{\CC[\zetn]}\CC[\zetn^{\pm 1}]$ generated by $\{\xizp^j\whmunk{m}(=\whmunk{m}\otimes \xizp^j)\in \whmunk{M}\otimes_{\CC[\zetn]}\CC[\zetn^{\pm 1}]\ |\ m\in M^{\gamma}\}$
and the RHS of (\ref{ohaka}) is the subset
$\{\sum_{i=1}^{k}s_{i}\ |\ s_{i}\in \xizp^{j_{i}}\whmu{M^{\gamma_{i}}}\ (j_{i}\in \ZZ, \gamma_{i}\in \ZZ , j_{i}+\gamma_{i}+n=\beta)\}$
of $\whmunk{M}\otimes_{\CC[\zetn]}\CC[\zetn^{\pm 1}]$
(in other words, it is not the $\Gamma(U_{n}^{\vee};O)$-module generated by $\{\xizp^j\whmu{M^{\gamma}}\}$, but the $\Gamma(\CC_{\zeta_{1}'}\times\dots \times \CC_{\zeta_{n-1}'}; O)$-module generated by them).
\end{lem}
\begin{proof}
As we already remarked, 
for a section $m\in M$,
we write $\whmunk{m}$ for the section $\whmunk{m}\otimes 1\in \whmunk{M}\otimes_{\CC[\zetn]}\CC[\zetn^{\pm 1}]$.
For $m\in M^{\gamma}$ and $j\in \ZZ$, consider a section ${\xizp}^j\whmunk{m}$.
Then, we have
\begin{align*}
\xizp\pa_{\xizp}({\xizp}^j\whmunk{m})&=
\xizp^j(j+\xizp\pa_{\xizp})\whmunk{m}\\
&=\xizp^j(j-\eu_{\UV_0})\whmunk{m}\\
&=\xizp^j\whmu{(j+\eu_{\Cnt}+n)m}.
\end{align*}
This implies the desired assertions.
\end{proof}

\begin{rem}\label{together}
Similary to $\sum_{\substack{j\in \ZZ,\gamma\in \RR\\ j+\gamma+n=\beta}}{\zeta_{0}'}^{j}\whmu{M^{\gamma}}$ above,
for a family $\{A_{i}\}_{i}$ of subsets of $(\whmunk{M}\otimes_{\CC[\zetn]}\CC[\zetn^{\pm 1}])^{\beta}_{\zeta'_{0}}$
we denote by $\sum_{i}A_{i}$ the $\Gamma(\CC_{\zeta_{1}'}\times\dots \times \CC_{\zeta_{n-1}'}; O)$-submodule of 
$(\whmunk{M}\otimes_{\CC[\zetn]}\CC[\zetn^{\pm 1}])^{\beta}_{\zeta'_{0}}$
generated by $\{A_{i}\}_{i}$, not the 
$\Gamma(U_{n}^{\vee};O)$-module generated by them,
when no confusion arises.
\end{rem}

For $\beta\in\RR$, we define a positive integer $j_{\beta}\in \ZZ_{\geq 0}$ by
\[j_{\beta}:=\max\{\lceil-\beta\rceil -n-1,0\}.\]
We will use the following elementary lemma.
\begin{lem}\label{mikkiikai}
\begin{enumerate} 
    \item[(i)] For any $\beta\in \RR$ and $j\in \ZZ_{\geq 0}$, the inequality $j+\beta+n\geq -1$ holds if and only if the inequality $j\geq j_{\beta}$ holds.
    \item[(ii)] For any $\beta\in \RR$, we have 
    \[j_{\beta}+\beta+n\geq -1.\]
    \item[(iii)] For $\beta\in \RR$ and $r\geq 0$, 
if $j_{\beta}+\beta+n\geq r$,
we have $j_{\beta}=j_{\beta-1}=\cdots =j_{\beta-r-1}=0$.
\end{enumerate}

\end{lem}

\begin{cor}\label{wakuseinine}
We have
\begin{align}\label{soutai}
    V^{-1}_{\xizp}(\whmunk{\scrM}\otimes_{\CC[\zetn]}\CC[\zetn^{\pm 1}][*\{\xizp=0\}])
=
\sum_{\substack{j\geq 0, p\in \ZZ, \beta\in \RR\\ j\geq j_{\beta}}}\xizp^{j}\whmu{F_{p}M^{\beta}\lam^{p}}.
\end{align}
\end{cor}

\begin{proof}
By Proposition~\ref{localizfac}, for $\gamma>-1$ we have 
\begin{align*}
V^{\gamma}_{\xizp}(\whmunk{\scrM}\otimes_{\CC[\zetn]}\CC[\zetn^{\pm 1}][*\{\xizp=0\}])
=&V^{\gamma}_{\xizp}(\whmunk{\scrM}\otimes_{\CC[\zetn]}\CC[\zetn^{\pm 1}]),
\end{align*}
By Lemma~\ref{kinomi}, the RHS is equal to
\[\sum_{\substack{j,p\in \ZZ, \beta\in \RR\\ j+\beta+n\geq \gamma}}\xizp^{j}\whmu{F_{p}M^{\beta}\lam^{p}}.
\]
Since $V^{-1}_{\xizp}(\whmunk{\scrM}\otimes_{\CC[\zetn]}\CC[\zetn^{\pm 1}][*\{\xizp=0\}])
$ is $\xizp^{-1}V^{0}_{\xizp}(\whmunk{\scrM}\otimes_{\CC[\zetn]}\CC[\zetn^{\pm 1}][*\{\xizp=0\}])$ (see Proposition~\ref{localizfac}),
we have
\begin{align}\label{osumou}
V^{-1}_{\xizp}(\whmunk{\scrM}\otimes_{\CC[\zetn]}\CC[\zetn^{\pm 1}][*\{\xizp=0\}])
=
\sum_{\substack{j,p\in \ZZ, \beta\in \RR\\ j+\beta+n\geq -1}}\xizp^{j}\whmu{F_{p}M^{\beta}\lam^{p}}.
\end{align}
Moreover, for $j<0$, $\beta\in \RR$ with $j+\beta+n\geq -1$ and $p\in \ZZ$, 
we have
\begin{align*}
\xizp^{j}\whmu{F_{p}M^{\beta}\lam^{p}}
=&
\zeta_n^{-j}\whmu{F_{p}M^{\beta}\lam^{p}}\\
=&
\whmu{\eth_{z_{n}}^{-j}F_{p}M^{\beta}\lam^{p}}\\
\subset &
\whmu{F_{p-j}M^{\beta+j}\lam^{p-j}}\\
= &
\xizp^{0}\whmu{F_{p-j}M^{\beta+j}\lam^{p-j}}.
\end{align*}
The last term is contained in
the RHS of (\ref{soutai}).
Therefore, together with (i) of Lemma~\ref{mikkiikai}, the RHS of (\ref{osumou}) is equal to the RHS of (\ref{soutai}).
This completes the proof.
\end{proof}


Recall again that $\whmunk{\scrM}\otimes_{\CC[\zetn]}\CC[\zetn^{\pm 1}][*\{\xizp=0\}]$ is a submodule of $\whmunk{\scrM}\otimes_{\CC[\zetn]}\CC[\zetn^{\pm 1}]$ generated by $V^{-1}_{\xizp}(\whmunk{\scrM}\otimes_{\CC[\zetn]}\CC[\zetn^{\pm 1}])$ (Proposition~\ref{localizfac}).
Therefore, by Corollary~\ref{wakuseinine}, we have the following.

\begin{cor}\label{meikyu}
We have
\begin{align*}
\whmunk{\scrM}\otimes_{\CC[\zetn]}\CC[\zetn^{\pm 1}][*\{\xizp=0\}]
=\sum_{\substack{k, j\geq 0, p\in \ZZ, \beta\in \RR\\ j\geq j_{\beta}}}\eth_{\xizp}^k\xizp^{j}\whmu{F_{p}M^{\gamma}\lam^{p}}.
\end{align*}
\end{cor}

Note that the rescaled module $\RESB{(\whmunk{\scrM}\otimes_{\CC[\zetn]}\CC[\zetn^{\pm 1}][*\{\xizp=0\}])}$ 
is 
\[O_{U_{n}^{\vee}}[\tau^{\pm 1},\lam]\otimes_{O_{U_{n}^{\vee}}[\lam]}(\whmunk{\scrM}\otimes_{\CC[\zetn]}\CC[\zetn^{\pm 1}][*\{\xizp=0\}]),\]
as an $O_{U_{n}^{\vee}\times \CC_{\tau}\times \CC_{\lam}}(*\{\tau=0\})$-module with an $R^{\integ}$-module action (\ref{wingspan}).

\begin{lem}\label{faiman}
We have
\begin{align}\label{ryosan}V_{\tau}^{\gamma}(\RESB{(\whmunk{\scrM}\otimes_{\CC[\zetn]}\CC[\zetn^{\pm 1}][*\{\xizp=0\}])})
=
\bigoplus_{i\in \ZZ}
\sum_{\substack{
k,j\geq 0, p\in \ZZ, \beta\in \RR
\\
j\geq j_{\beta}\\
i-k-p-\beta\geq \gamma}
}\tau^i\otimes \eth_{\xizp}^k\xizp^{j}\whmu{F_{p}M^{\beta}\lam^{p}}.
\end{align}
\end{lem}
\begin{proof}
As in the proof of Lemma~\ref{golgo},
if a section $s\in \RESB{(\whmunk{\scrM}\otimes_{\CC[\zetn]}\CC[\zetn^{\pm 1}][*\{\xizp=0\}])}$ is killed by $(\tau\eth_{\tau}-\gamma\lam)^l$ for some $l\geq 0$, $s$ is in $V_{\tau}^{\gamma}(\RESB{(\whmunk{\scrM}\otimes_{\CC[\zetn]}\CC[\zetn^{\pm 1}][*\{\xizp=0\}])})$.

Let $\beta$ be a real number and $m$ a section of $F_{p}M^{\beta}$.
For $j,k\in \ZZ_{\geq 0}$ with $j+\beta+n\geq -1$ and $i\in \ZZ$,
we consider a section $\tau^i\otimes \eth_{\xizp}^k\xizp^j\whmu{m\lam^p}\in \RES{(\whmunk{\scrM}\otimes_{\CC[\zetn]}\CC[\zetn^{\pm 1}][*\{\xizp=0\}])}$. 
Recall that we have
\begin{align*}
\eth_{\xizp}&=-\lam\zeta_{n}\eu_{\Cnx}\quad \mbox{and}\\
\xizp&=\zeta_{n}^{-1},
\end{align*}
on $U_{0}^{\vee}\cap U_{n}^{\vee}$ and
\[\lam^2\pa_{\lam}\whmu{m\lam^p}=\whmu{(\lam^2\pa_{\lam}+\lam\eu_{\Cnt})m\lam^p}
(=\whmu{(p+\eu_{\Cnt})m\lam^{p+1}}).\]
Therefore,
we have
\begin{align*}
\lam^2\pa_{\lam}(\eth_{\xizp}^k\xizp^j\whmu{m\lam^p})
&=\lam^2\pa_{\lam}((-\lam\zeta_{n}\eu_{\Cnx})^k\zeta_{n}^{-j}\whmu{m\lam^p})\notag\\
&=(k\lam(-\lam\zeta_{n}\eu_{\Cnx})^k\zeta_{n}^{-j}+(-\lam\zeta_{n}\eu_{\Cnx})^k\zeta_{n}^{-j}\lam^2\pa_{\lam})\whmu{m\lam^p}\notag\\
&=k(-\lam\zeta_{n}\eu_{\Cnx})^k\zeta_{n}^{-j}\whmu{m\lam^{p+1}}
+(-\lam\zeta_{n}\eu_{\Cnx})^k\zeta_{n}^{-j}
\whmu{(p+\eu_{\Cnt})m\lam^{p+1}}\notag\\
&=
\eth_{\xizp}^k\xizp^{j}
\whmu{(k+p+\eu_{\Cnt})m\lam^{p+1}}.
\end{align*}
By using this, we obtain
\begin{align*}
\tau\eth_{\tau}(\tau^i\otimes \eth_{\xizp}^k\xizp^j\whmu{m\lam^p})
&=i\tau^{i+1}\otimes \eth_{\xizp}^k\xizp^j\whmu{m\lam^{p+1}}
-\tau^{i+1}\otimes \lam^2\pa_{\lam}(\eth_{\xizp}^k\xizp^j\whmu{m\lam^{p}})\\
&=
i\tau^{i+1}\otimes \eth_{\xizp}^k\xizp^j\whmu{m\lam^{p+1}}
-\tau^{i+1}\otimes \eth_{\xizp}^k\xizp^{j}
\whmu{(k+p+\eu_{\Cnt})m\lam^{p+1}}\\
&=\tau^{i+1}\otimes \eth_{\xizp}^k\xizp^j\whmu{(i-k-p-\eu_{\Cnt})m\lam^{p+1}}.
\end{align*}
Therefore, we have 
\[
(\tau\eth_{\tau}-(i-k-p-\beta)\lam)(\tau^i\otimes \eth_{\xizp}^k\xizp^j\whmu{m\lam^p})
=-\tau^{i+1}\otimes \eth_{\xizp}^k\xizp^j\whmu{(\eu_{\Cnt}-\beta)m\lam^{p+1}},
\]
and hence
\[
(\tau\eth_{\tau}-(i-k-p-\beta)\lam)^l(\tau^i\otimes \eth_{\xizp}^k\xizp^j\whmu{m\lam^p})
=(-1)^l\tau^{i+l}\otimes \eth_{\xizp}^k\xizp^j\whmu{(\eu_{\Cnt}-\beta)^lm\lam^{p+l}},
\]
for $l\geq 0$.
Since the RHS is zero for sufficiently large $l\geq 0$,
we conclude that
$\tau^i\otimes \eth_{\xizp}^k\xizp^j\whmu{m\lam^p}$ is killed by $(\tau\eth_{\tau}-(i-k-p-\beta)\lam)^l$ for some $l\geq 0$
and hence 
$\tau^i\otimes \eth_{\xizp}^k\xizp^j\whmu{m\lam^p}$ is in
$V^{i-k-p-\beta}_{\tau}(\RESB{(\whmunk{\scrM}\otimes_{\CC[\zetn]}\CC[\zetn^{\pm 1}][*\{\xizp=0\}])})$.
We thus conclude that the RHS of (\ref{ryosan}) is contained in the LHS.

In the same way as the proof of Lemma~\ref{golgo},
we can show the LHS of (\ref{ryosan}) is contained in the RHS.
\end{proof}

Now, we can compute the irregular Hodge filtration $\irrF_{\bullet}(\whmunk{M}\otimes_{\CC[\zetn]}\CC[\zetn^{\pm 1}])(=\irrF_{\alpha+p}\whmunk{N}|_{U_{n}^{\vee}})$ in the same way as for $\irrF_{\bullet}\whmunk{M}$.

\begin{thm}\label{dengaku}
For $\alpha\in [0,1)$ and $p\in \ZZ$ we have
\begin{align}
\begin{aligned}
\label{saikoro}
\irrF_{\alpha+p}(\whmunk{M}\otimes_{\CC[\zetn]}\CC[\zetn^{\pm 1}])
=
(\sum_{\substack{
j\geq 0, \beta\in \RR
}
}\xizp^{j_{\beta}+j}&\whmu{F_{
p+\lfloor \alpha-\beta\rfloor
}M^{\beta}})\\
&+
(
\sum_{\substack{
k\geq 0, \beta\in \RR
}
}\pa_{\xizp}^{k}\xizp^{j_{\beta}}\whmu{F_{
p-k+\lfloor\alpha-\beta\rfloor
}M^{\beta}}).
\end{aligned}
\end{align}
\end{thm}

\begin{proof}
The proof is similar to that of Theorem~\ref{prosyuto}.
By Lemma~\ref{faiman}, we have
\[i^{*}_{\tau=\lam}V^{-\alpha}_{\tau}(\RESB{(\whmunk{\scrM}\otimes_{\CC[\zetn]}\CC[\zetn^{\pm 1}][*\{\xizp=0\}])})
=i^{*}_{\tau=\lam}
(\bigoplus_{i\in \ZZ}\sum_{\substack{
k,j\geq 0, p\in \ZZ,\beta\in \RR
\\
j\geq j_{\beta}\\
i-k-p-\beta\geq -\alpha}
}\tau^i\otimes \eth_{\xizp}^k\xizp^{j}\whmu{F_{p}M^{\beta}\lam^{p}}).
\]
Therefore, the Rees module of the irregular Hodge filtration (a submodule of $O_{U_{n}^{\vee}}[\lam^{\pm 1}]\otimes (\whmunk{M}\otimes_{\CC[\zetn]}\CC[\zetn^{\pm 1}]))$ is
\[
\bigoplus_{i\in \ZZ}\sum_{\substack{
k,j\geq 0, p\in \ZZ,\beta\in \RR
\\
j\geq j_{\beta}\\
i-k-p-\beta\geq -\alpha}
}\lam^i\otimes \pa_{\xizp}^k\xizp^{j}\whmu{F_{p}M^{\beta}}.
\]
Hence, after replacing $i$ by $p$
we have
\begin{align}\label{beach}
\irrF_{\alpha+p}(\whmunk{M}\otimes_{\CC[\zetn]}\CC[\zetn^{\pm 1}])
=
\sum_{\substack{
k,j\geq 0, \beta\in \RR
\\
j\geq j_{\beta}}
}\pa_{\xizp}^k\xizp^{j}\whmu{F_{
\lfloor p-k-\beta+\alpha\rfloor
}M^{\beta}}.
\end{align}
Therefore, by (i) and (ii) of Lemma~\ref{mikkiikai}, the RHS of (\ref{saikoro}) is contained in the RHS of (\ref{beach}).

We show the converse inclusion.
For $k_{0},j_{0}\in \ZZ_{\geq 0}$, $\beta_{0}\in \RR$ with $j_{0}\geq j_{\beta_{0}}$ ($\iff$ $j_{0}+\beta_{0}+n\geq -1$) and $m\in F_{\lfloor p-k_{0}-\beta_{0}+\alpha\rfloor} M^{\beta_{0}}$,
we consider a section $\pa_{\xizp}^{k_{0}}\xizp^{j_{0}}\whmunk{m}$, which is in the RHS of (\ref{beach}).
By (i) of Lemma~\ref{mikkiikai}, we have $j_{0}\geq j_{\beta_{0}}$.
We set $j_{0}':=j_{0}-j_{\beta_{0}}\in \ZZ_{\geq 0}$ so that $j_{\beta_{0}}+j_{0}'=j_{0}$.

\noindent \textit{The case: $j_{0}'\geq k_{0}$}
\quad In this case, we have
\begin{align}
\begin{aligned}\label{nintai}
\pa_{\xizp}^{k_{0}}\xizp^{j_{\beta_{0}}+j_{0}'}\whmunk{m}
&=\sum_{l=0}^{k_{0}}a_{s}\xizp^{j_{\beta_{0}}+j_{0}'-l}\pa_{\xizp}^{k_{0}-l}\whmunk{m}\quad (\mbox{for some $a_{s}\in \ZZ$})\\
&=\sum_{l=0}^{k_{0}}a_{s}\xizp^{j_{\beta_{0}}+j_{0}'-l}\whmu{(\pa_{z_{n}}(\eu_{\Cnt}+n))^{k_{0}-l}m}\\
&\in \sum_{l=0}^{k_{0}}\xizp^{j_{\beta_{0}}+j_{0}'-l}\whmu{F_{\lfloor p-k_{0}-\beta_{0}+\alpha\rfloor+2(k_{0}-l)} M^{\beta_{0}-(k_{0}-l)}}\\
&\subset \sum_{l=0}^{k_{0}}\xizp^{j_{\beta_{0}}+j_{0}'-l}\whmu{F_{\lfloor p-(\beta_{0}-(k_{0}-l))+\alpha\rfloor} M^{\beta_{0}-(k_{0}-l)}}.
\end{aligned}
\end{align}
The last term is contained in the first part of RHS of (\ref{saikoro}) since $j_{0}'-l\geq j_{0}'-k_{0}\geq 0$.

\noindent \textit{The case: $j_{0}'\leq k_{0}$}
\quad 
If $j_{0}'=0$, $\pa_{\xizp}^{k_{0}}\xizp^{j_{\beta_{0}}+j_{0}'}\whmunk{m}$ is in the RHS of (\ref{saikoro}).
So, we assume $j_{0}'\geq 1$.
In this case, 
we ``divide'' $\pa_{\xizp}^{k_{0}}$ as
\begin{align*}
\pa_{\xizp}^{k_{0}}\xizp^{j_{\beta_{0}}+j_{0}'}\whmunk{m}
=\pa_{\xizp}^{k_{0}-j_{0}'}\pa_{\xizp}^{j_{0}'}\xizp^{j_{\beta_{0}}+j_{0}'}\whmunk{m}.
\end{align*}
Then, by (\ref{nintai}) for $\pa_{\xizp}^{j_{0}'}\xizp^{j_{\beta_{0}}+j_{0}'}\whmunk{m}$, the section $\pa_{\xizp}^{k_{0}-j_{0}'}\pa_{\xizp}^{j_{0}'}\xizp^{j_{\beta_{0}}+j_{0}'}\whmunk{m}$ is contained in 
\begin{align}\label{gogogo}
\sum_{l=0}^{j_{0}'}\pa_{\xizp}^{k_{0}-j_{0}'}\xizp^{j_{\beta_{0}}+j_{0}'-l}\whmu{F_{\lfloor p-(\beta_{0}-(k_{0}-l))+\alpha\rfloor} M^{\beta_{0}-(k_{0}-l)}}.
\end{align}
Then, by induction on the exponent of $\pa_{\xizp}$ (remark that $k_{0}-j_{0}'< k_{0}$ since $j_{0}'\geq 1$),
we conclude that 
the RHS of (\ref{gogogo}) is contained in the RHS of (\ref{saikoro}).
\end{proof}

Moreover, we have the following.
\begin{cor}\label{kenmei2cor}
\begin{enumerate}
    \item[(i)]
For $\alpha\in [0,1)$, $p\in \ZZ$ and $\gamma\in \RR_{\geq -1}$,
we have
\begin{align}\label{vino}
\irrF_{\alpha+p}V^{\gamma}_{\xizp}(\whmunk{M}\otimes_{\CC[\zetn]}\CC[\zetn^{\pm 1}])
=\sum_{\substack{j\geq 0, \beta\in \RR\\ j_{\beta}+\beta+n+j\geq \gamma}}
\xizp^{j_{\beta}+j}\whmu{F_{p+\lfloor \alpha-\beta\rfloor}M^{\beta}}.
\end{align}
\item[(ii)] For $\alpha\in [0,1)$, $p\in \ZZ$ and $\gamma\in \RR_{< 0}$,
we have
\begin{align}\label{ukurere}
\irrF_{\alpha+p}\GR^{\gamma}_{V}(\whmunk{M}\otimes_{\CC[\zetn]}\CC[\zetn^{\pm 1}])
=\sum_{\substack{k \geq 0, \beta\in \RR\\ j_{\beta}+\beta+n-k=\gamma}}\pa_{\xizp}^{k}\xizp^{j_{\beta}}\whmu{F_{p-k+\lfloor\alpha-\beta\rfloor}M^{\beta}},
\end{align}
where the sum $\sum$ in the RHS is the one defined in Remark~\ref{together}.
\end{enumerate}
\end{cor}
\begin{proof}
Recall that $M^{\wedge}\otimes_{\CC[\zetn]}\CC[\zetn^{\pm 1}]$ is monodromic with respect to the $\xizp$-direction by Lemma~\ref{kinomi}
and we have
\[(M^{\wedge}\otimes_{\CC[\zetn]}\CC[\zetn^{\pm 1}])^{\beta}_{\xizp}=
\sum_{\substack{j\in \ZZ,\gamma\in \RR\\ j+\gamma+n=\beta}}{\zeta_{0}'}^{j}\whmu{M^{\gamma}},
\]
where the sum $\sum$ is the one defined in Lemma~\ref{kinomi} or Remark~\ref{together}.
Therefore,
the term $\xizp^{j_{\beta}+j}\whmu{F_{
p+\lfloor \alpha-\beta\rfloor
}M^{\beta}}$ 
(resp. $\pa_{\xizp}^{k}\xizp^{j_{\beta}}\whmu{F_{
p-k+\lfloor\alpha-\beta\rfloor
}M^{\beta}}$)
in the first (resp. second) term of (\ref{saikoro})
is contained in $V^{\gamma}_{\xizp}(\whmunk{M}\otimes_{\CC[\zetn]}\CC[\zetn^{\pm 1}])$ if and only if $j_{\beta}+\beta+n+j\geq \gamma$ (resp. $j_{\beta}+\beta+n-k\geq \gamma$).
Hence, by Theorem~\ref{dengaku}, we have
\begin{align}\begin{aligned}\label{konomichi2}
\irrF_{\alpha+p}V^{\gamma}_{\xizp}(\whmunk{M}\otimes_{\CC[\zetn]}\CC[\zetn^{\pm 1}])
=
(\sum_{\substack{
j\geq 0, \beta\in \RR\\
j_{\beta}+\beta+n+j\geq \gamma
}
}&\xizp^{j_{\beta}+j}\whmu{F_{
p+\lfloor \alpha-\beta\rfloor
}M^{\beta}})
\\&+
(
\sum_{\substack{
k\geq 0, \beta\in \RR\\
j_{\beta}+\beta+n-k\geq \gamma
}
}\pa_{\xizp}^{k}\xizp^{j_{\beta}}\whmu{F_{
p-k+\lfloor\alpha-\beta\rfloor
}M^{\beta}}).
\end{aligned}
\end{align}

To prove (i), we assume that $\gamma\in \RR_{\geq -1}$.
Then,
for $k\geq 1$ (not $0$) with
$j_{\beta}+\beta+n-k\geq \gamma$,
we have
$j_{\beta}+\beta+n\geq k-1(\geq 0)$.
Therefore, 
we get $j_{\beta}=\dots =j_{\beta-k}=0$ by (iii) of Lemma~\ref{mikkiikai}.
Therefore, we have
\begin{align*}
\pa^{k}_{\xizp}\xizp^{j_{\beta}}\whmu{F_{ p-k+\lfloor\alpha-\beta\rfloor}M^{\beta}}
=&\pa^{k}_{\xizp}\whmu{F_{p-k+\lfloor\alpha-\beta\rfloor}M^{\beta}}\\
=&\whmu{(\pa_{z_{n}}(\eu_{\CC^n}+n))^kF_{ p-k+\lfloor\alpha-\beta\rfloor}M^{\beta}}\\
\subset &\whmu{F_{ p+k+\lfloor\alpha-\beta\rfloor}M^{\beta-k}}\\
=&\xizp^{j_{\beta-k}}\whmu{F_{ p+k+\lfloor\alpha-\beta\rfloor}M^{\beta-k}}\quad (\mbox{since $j_{\beta-k}=0$}).
\end{align*}
Hence, the term $\pa^{k}_{\xizp}\xizp^{j_{\beta}}\whmu{F_{ p-k+\lfloor\alpha-\beta\rfloor}M^{\beta}}$ for $k\geq 1$ in the second part of the RHS of (\ref{konomichi2}) 
is contained in its first part
(for $s=0$).
This completes the proof of (i).

To prove (ii), we assume that $\gamma\in \RR_{< 0}$.
The first term of (\ref{konomichi2}) is contained in $V^{0}_{\xizp}(\whmunk{M}\otimes_{\CC[\zetn]}\CC[\zetn^{\pm 1}])$ unless $s=0$.
Hence, for $\gamma<0$, we have
\begin{align}\label{atosukosida}
\irrF_{\alpha+p}\GR^{\gamma}_{V}(\whmunk{M}\otimes_{\CC[\zetn]}\CC[\zetn^{\pm 1}])
\simeq \sum_{\substack{k\geq 0, \beta\in \RR\\ j_{\beta}+\beta+n-k=\gamma}}\pa_{\xizp}^{k}\xizp^{j_{\beta}}\whmu{F_{p-k+\lfloor\alpha-\beta\rfloor}M^{\beta}}.
\end{align}
\end{proof}

\begin{cor}\label{sincyan}
For $\alpha\in [0,1)$, the irregular Hodge filtration $\irrF_{\alpha+\bullet}(\whmunk{M}\otimes_{\CC[\zetn]}\CC[\zetn^{\pm 1}])$ satisfies the strict specializability property along $\xizp=0$.  
\end{cor}

\begin{proof}
First, let us see the condition (i) in Definition~\ref{tron}.
For $\gamma>-1$, 
by (i) of Corollary~\ref{kenmei2cor}, we have
\begin{align}\label{rohan}
\irrF_{\alpha+p}\GR^{\gamma}_{V}(\whmunk{M}\otimes_{\CC[\zetn]}\CC[\zetn^{\pm 1}])
\simeq \sum_{\substack{j\geq 0, \beta\in \RR\\ j_{\beta}+\beta+n+j=\gamma}}\xizp^{j_{\beta}+j}\whmu{F_{p+\lfloor \alpha-\beta\rfloor}M^{\beta}}.
\end{align}
It is enough to see that 
for $\gamma=\gamma_{0}>0$ and a section $\sigma$ in the RHS of (\ref{rohan}),
$\xizp^{-1}\sigma$ is in the RHS of (\ref{rohan}) for $\gamma=\gamma_{0}-1$.
Consider a section $\xizp^{j_{\beta}+j}\whmunk{m}$ for $j_{\beta}+\beta+n+j=\gamma_{0}$ and $m\in F_{p+\lfloor \alpha-\beta\rfloor}M^{\beta}$.
If $j\geq 1$,
it is clear that $\xizp^{-1}(\xizp^{j_{\beta}+j}\whmunk{m})$ is in the RHS of (\ref{rohan}) for $\gamma=\gamma_{0}-1$.
So, let us assume that $j=0$.
By (iii) of Lemma~\ref{mikkiikai}, we have $j_{\beta}=j_{\beta-1}=0$.
Therefore, we have
\begin{align*}
    \xizp^{-1}\cdot \xizp^{j_{\beta}}\whmunk{m}
    &= \xizp^{j_{\beta}}\whmu{\pa_{z_{n}}m}\\
    &\subset \xizp^{j_{\beta}}\whmu{F_{p+\lfloor\alpha-\beta\rfloor+1 }M^{\beta-1}}\\
    &= \xizp^{j_{\beta-1}}\whmu{F_{p+\lfloor\alpha-(\beta-1)\rfloor }M^{\beta-1}}\quad (\mbox{by $j_{\beta}=j_{\beta-1}=0$}).
\end{align*}
The last term is contained in the RHS of (\ref{rohan}) for $\gamma=\gamma_{0}-1$.
This completes the proof of the condition (i) in Definition~\ref{tron}.

Let us check the condition (ii) in Definition~\ref{tron}.
By (ii) of Corollary~\ref{kenmei2cor}, for $\gamma<0$ we have
\begin{align}\label{atosukosida2}
\irrF_{\alpha+p}\GR^{\gamma}_{V}(\whmunk{M}\otimes_{\CC[\zetn]}\CC[\zetn^{\pm 1}])
=\sum_{\substack{k \geq 0, \beta\in \RR\\ j_{\beta}+\beta+n-k=\gamma}}\pa_{\xizp}^{k}\xizp^{j_{\beta}}\whmu{F_{p-k+\lfloor\alpha-\beta\rfloor}M^{\beta}}.
\end{align}
For $\gamma_{0}<-1$ (not $<0$), consider a section $\pa_{\xizp}^k\xizp^{j_{\beta}}\whmunk{m}$ with
$j_{\beta}+\beta+n-k=\gamma_{0}$ and $m\in F_{p-k+\lfloor\alpha-\beta\rfloor}M^{\beta}$.
Since $j_{\beta}+\beta+n\geq -1$ by (ii) of Lemma~\ref{mikkiikai},
we have $k\geq 1$.
Moreover, $\pa_{\xizp}^{k-1}\xizp^{j_{\beta}}\whmunk{m}$ is in
the RHS of (\ref{atosukosida2}) for $\gamma=\gamma_{0}+1$.
Therefore, we have
\[\irrF_{\alpha+p}\GR^{\gamma}_{V}(\whmunk{M}\otimes_{\CC[\zetn]}\CC[\zetn^{\pm 1}])
\subset \pa_{\xizp}\cdot \irrF_{\alpha+p-1}\GR^{\gamma-1}_{V}(\whmunk{M}\otimes_{\CC[\zetn]}\CC[\zetn^{\pm 1}]),
\]
which is the condition (ii) in Definition~\ref{tron}.

This completes the proof.
\end{proof}
\begin{rem}
Corollary~\ref{sincyan} is derived by Theorem~1.6 of Mochizuki~\cite{MochiResc}, which is an assertion about the strict specializability in a more general setting.
The above proof is a concrete verification of this fact.
\end{rem}

Finally, we check that ``the irregular Hodge filtration at infinity is localized''.

\begin{cor}\label{cyobi}
\begin{enumerate}
  \item[(i)]  For $\alpha\in [0,1)$, we have
\[\irrF_{\alpha+p}V^{-1}_{\xizp}(\whmunk{M}\otimes_{\CC[\zetn]}\CC[\zetn^{\pm 1}])
=\xizp^{-1}\irrF_{\alpha+p}V^{0}_{\xizp}(\whmunk{M}\otimes_{\CC[\zetn]}\CC[\zetn^{\pm 1}]).
\]
\item[(ii)]
We have
\[\irrF_{\alpha+p}(\whmunk{M}\otimes_{\CC[\zetn]}\CC[\zetn^{\pm 1}])=
\sum_{k\geq 0}\pa_{\xizp}^k\irrF_{\alpha+p-k}V^{-1}_{\xizp}(\whmunk{M}\otimes_{\CC[\zetn]}\CC[\zetn^{\pm 1}]).
\]
\end{enumerate}
\end{cor}
\begin{proof}
The assertion (i) of Corollary~\ref{kenmei2cor} implies that
\[\irrF_{\alpha+p}V^{-1}_{\xizp}(\whmunk{M}\otimes_{\CC[\zetn]}\CC[\zetn^{\pm 1}])\subset \xizp^{-1}\irrF_{\alpha+p}V^{0}_{\xizp}(\whmunk{M}\otimes_{\CC[\zetn]}\CC[\zetn^{\pm 1}]).\]
Conversely, we consider $\xizp^{-1}\cdot \xizp^{j_{\beta}}\whmu{F_{ p+\lfloor \alpha-\beta\rfloor}M^{\beta}}$,
where $\xizp^{j_{\beta}}\whmu{F_{p+\lfloor \alpha-\beta\rfloor}M^{\beta}}$ is the term for $s=0$ and $l=0$ in (\ref{vino}).
Since $j_{\beta}+\beta+n\geq 0$, 
by (iii) of Lemma~\ref{mikkiikai}, we have $j_{\beta}=j_{\beta-1}=0$.
Therefore, we have
\begin{align*}
    \xizp^{-1}\cdot \xizp^{j_{\beta}}\whmu{F_{\lfloor p-\beta+\alpha\rfloor}M^{\beta}}
    =&\xizp^{-1} \whmu{F_{p+\lfloor \alpha-\beta\rfloor}M^{\beta}}\\
    =&\whmu{\pa_{z_{n}}F_{p+\lfloor \alpha-\beta\rfloor}M^{\beta}}\\
    \subset &\whmu{F_{p+\lfloor \alpha-\beta\rfloor+1}M^{\beta-1}}\\
    =& \xizp^{j_{\beta-1}}\whmu{F_{ p+\lfloor \alpha-(\beta-1)\rfloor}M^{\beta-1}} \quad (\mbox{by $j_{\beta}=j_{\beta-1}$}).
\end{align*}
The last term is contained in $\irrF_{\alpha+p}V^{-1}_{\xizp}(\whmunk{M}\otimes_{\CC[\zetn]}\CC[\zetn^{\pm 1}])$.
This completes the proof of (i). 

(ii) follows from the strict specializability Corollary~\ref{sincyan}.
\end{proof}

For a filtered $D$-module $(M,F_{\bullet}M)$ on $U_{n}^{\vee}$,
we define $(M,F_{\bullet}M)[*\{\xizp=0\}]$ as the filtered $D$-module $M(*\{\xizp=0\})$ with the filtration defined by the same formula as (v) of Proposition~\ref{localizfac}. 
Then, we can write
\[(\whmunk{M}\otimes_{\CC[\zetn]}\CC[\zetn^{\pm 1}],\irrF_{\alpha+\bullet}(\whmunk{M}\otimes_{\CC[\zetn]}\CC[\zetn^{\pm 1}]))=(\whmunk{M}\otimes_{\CC[\zetn]}\CC[\zetn^{\pm 1}],\irrF_{\alpha+\bullet}(\whmunk{M}\otimes_{\CC[\zetn]}\CC[\zetn^{\pm 1}]))[*\{\xizp=0\}]\]
This corollary means that the irregular Hodge filtration has the same properties as the Hodge filtration of the localization of an usual mixed Hodge module.

Obviously,
we have the same statement also for the irregular Hodge filtration of $\whmunk{N}|_{U_{i}^{\vee}}$ for $i=1,\dots,n-1$.
Therefore, we can restate Corollaries~\ref{sincyan} and \ref{cyobi} as follows.
\begin{cor}\label{nonnonbaa}
For $\alpha\in [0,1)$, the irregular Hodge filtration $\irrF_{\alpha+\bullet}\whmunk{N}$ has the strict specializability property along $D^{\vee}_{\infty}$.
Moreover, we have
\[(\whmunk{N}, \irrF_{\alpha+\bullet}\whmunk{N})=(\whmunk{N},\irrF_{\alpha+\bullet}\whmunk{N})[*D^{\vee}_{\infty}].\]
\end{cor}

\begin{rem}
We will later prove that the filtration $\{\irrF_{p}\whmunk{N}\}_{p\in \ZZ}$ is the Hodge filtration of a mixed Hodge module (Corollary~\ref{rattlechain}).
Since the Hodge filtration of a mixed Hodge module is strictly specializable along any divisor,
Corollary~\ref{nonnonbaa} for $\alpha=0$ follows also from this fact.
Corollary~\ref{nonnonbaa} is an improvement on it since it says that the strict specializability properties hold also for other $\alpha\in [0,1)$.
\end{rem}

\begin{rem}
By (\ref{saikoro}),
the restriction of $\irrF_{\alpha+p}(\whmunk{M}\otimes_{\CC[\zetn]}\CC[\zetn^{\pm 1}])$ to $U_{0}^{\vee}\cap U_{n}^{\vee}$ is
\[\sum_{\beta\in \RR}\whmu{F_{p+\lfloor\alpha-\beta\rfloor}M^{\beta}}\otimes_{\CC[\zetn]}\CC[\zetn^{\pm 1}].\]
Therefore, the computation (\ref{saikoro}) is consistent with Theorem~\ref{prosyuto}.
\end{rem}

\begin{rem}\label{pikmin}
We remark that we can generalize all the results in this subsection, especially Theorems~\ref{prosyuto} and \ref{dengaku}, to mixed Hodge modules on a vector bundle on a smooth algebraic variety.
For this purpose, it is enough to prove them in the case of trivial vector bundles.
We omit the details.
\end{rem}

\subsection{The irregular Hodge filtration and the mixed Hodge module structure of $M^{\wedge}$}\label{doronokawa}
We continue to consider the setting of the previous subsection.
In Section~\ref{maboo}, we defined a mixed Hodge module structure on $\whmunk{M}$ and thus $\whmunk{M}$ is equipped with the Hodge filtration $F_{\bullet}\whmunk{M}$.
On the other hand, in the previous subsection we computed the irregular Hodge filtration $\irrF_{\alpha+\bullet}\whmunk{M}$ on $\whmunk{M}$ for $\alpha\in [0,1)$.
In this subsection, we will prove the following.

\begin{thm}\label{takarajima}
We have the equality
\[\irrF_{p}\whmunk{M}=F_{p}\whmunk{M}\]
for any $p\in \ZZ$.
\end{thm}

By Theorem~\ref{prosyuto}, we have the following.
\begin{cor}\label{nenmatu}
For $p\in \ZZ$, we have
\[F_{p}\whmunk{M}=\bigoplus_{\beta\in \RR}\whmu{F_{p+\lfloor-\beta\rfloor}M^{\beta}}.\]
\end{cor}

Recall that the mixed Hodge module structure on $\whmunk{M}$ is defined by the formula (\ref{letsgogo}).
Since it is difficult to compute the Hodge filtration of the pushforward of a mixed Hodge module in general,
it is also difficult to compute $F_{p}\whmunk{M}$ just following the definition.
So, we take a different approach, which takes the advantage of the strength of the theory of mixed twistor $D$-modules and the irregular Hodge theory.

\begin{notation}
For an $R^{\integ}$-module $\scrM$ and an integer $l\in \ZZ$,  
we set
\[\scrM(l):=\lam^{l}\scrM.\]
Note that if $\scrM$ is the Rees module $R_{F}M$ corresponding to a filtered $D$-module $(M,F_{\bullet}M)$,
we have
\[\scrM(l)=R_{F}(M(l)),\]
where $M(l)$ is the Tate twist of the filtered $D$-module $(M,F_{\bullet}M)$.

\end{notation}

We need to generalize Lemmas~\ref{suza} and \ref{donbee} to $R$-modules.

\begin{lem}\label{koukyou}
Let $\scrM_{1}$ and $\scrM_{2}$ be the underlying $R^{\integ}$-modules of mixed twistor $D$-modules on a smooth algebraic variety $X$.
Assume that the intersection of the characteristic varieties $\mathrm{Ch}(\scrM_{1})$ and $\mathrm{Ch}(\scrM_{2})$ is contained in a zero section $\CC_{\lam}\times T^{*}X$.
Then, for an algebraic variety $Y$ with a morphism $f\colon Y\to X$,
we have
\[\TO{f}^{!}\scrM_{1}\otimes \TO{f}^{!}\scrM_{2}=\TO{f}^{!}(\scrM_{1}\otimes \scrM_{2})(d_{f})[d_{f}],\]
where we put $d_{f}:=\dim{Y}-\dim{X}$.
\end{lem}

\begin{rem}
In \cite{MochiResc}, we say ``$\scrM_{1}$ and $\scrM_{2}$ are non-characteristic'' if the assumption in Lemma~\ref{koukyou} holds.
\end{rem}

\begin{proof}
Let $\Delta_{X}\colon X\to X\times X$ and $\Delta_{Y}\colon Y\to Y\times Y$ be the diagonal embeddings.
By the assumption, $\Delta_{X}$ is non-characteristic with respect to $\scrM_{1}\boxtimes \scrM_{2}$.
Therefore, by Corollary~4.56 of \cite{MochiResc}, $\dim{X}(=2\dim{X}-\dim{X})$-th one is the only non-trivial cohomology of $\TO{\Delta_{X}^{!}}(\scrM_{1}\boxtimes \scrM_{2})$,
and we have
\begin{align}\label{huyunookurimono}
H^{\dim{X}}(\TO{\Delta_{X}^{!}}(\scrM_{1}\boxtimes \scrM_{2}))
\simeq {\Delta_{X}^{*}}(\scrM_{1}\boxtimes \scrM_{2})(-\dim{X}).
\end{align}
The RHS is $\scrM_{1}\otimes \scrM_{2}(-\dim{X})$.
By using this fact, we have
\begin{align*}
    \TO{f}^{!}\scrM_{1}\otimes \TO{f}^{!}\scrM_{2}
    \simeq &\TO{\Delta_{Y}^{!}}(    \TO{f}^{!}\scrM_{1}\boxtimes \TO{f}^{!}\scrM_{2}
)(\dim{Y})[\dim{Y}]\\
\simeq &  \TO{\Delta_{Y}^{!}}  \TO{(f\times f)}^{!}(\scrM_{1}\boxtimes \scrM_{2}
)(\dim{Y})[\dim{Y}]\\
\simeq & \TO{f}^{!}  \TO{\Delta_{X}^{!}}(\scrM_{1}\boxtimes \scrM_{2})(\dim{Y})[\dim{Y}]
\\
\simeq & \TO{f}^{!}(\scrM_{1}\otimes \scrM_{2})(d_{f})[d_{f}].
\end{align*}
\end{proof}

Let $E$ be an algebraic vector bundle over a smooth algebraic variety $X$.

\begin{lem}\label{kaitaku}
Let $Y$ be a smooth algebraic variety and $f\colon Y\to X$ a morphism.
We denote by $u$ (resp. $u^{\vee}$) the natural morphism $f^{*}E\to E$ (resp. $f^{*}E^{\vee}\to E^{\vee}$).
For the underlying $R^{\integ}$-module $\scrM$ of an integrable mixed twistor $D$-module on $E$,
we have a natural isomorphism in the category of $R^{\integ}$-modules
\[\whmu{H^{j}(\TO{u}^{!}\scrM)}\simeq H^j(\TO{(u^{\vee})}^{!}\whmunk{\scrM})\quad (j\in \ZZ).\]
\end{lem}

\begin{proof}
We consider the following diagram
\[
\xymatrix{
f^{*}E\ar[d]^-{u}&f^{*}E\times_{Y}f^{*}E^{\vee}\ar[d]^-{u\times u^{\vee}}\ar[r]^-{q'}\ar[l]_-{p'}&f^{*}E^{\vee}\ar[d]^-{u^{\vee}}\\
E&E\times_{X} E^{\vee}\ar[r]^-{q}\ar[l]_-{p}&E^{\vee}}.
\]
Note that since any projection is non-characteristic with respect to any $R$-module,
we have
\[\TO(p')^{!}\simeq (p')^{*}(n_{E})[n_{E}],\]
where $n_{E}$ is the rank of $E$.
Moreover, we have
\begin{align*}
\calE^{-\varphi/\lam}_{f^{*}E\times_{Y}f^{*}E^{\vee}}\simeq & (u\times u^{\vee})^*\calE^{-\varphi/\lam}_{E\times_{X} E}\\
\simeq & 
\TO(u\times u^{\vee})^{!}\calE^{-\varphi/\lam}_{E\times_{X} E}(-d_{f})[-d_{f}].
\end{align*}
Therefore, we have
\begin{align*}
\whmu{H^{j}(\TO{u}^{!}\scrM)}
&\simeq 
\TO{q'}_{*}((p')^{*}(H^{j}\TO{u}^{!}\scrM)\otimes \calE^{-\varphi/\lam})\\
&\simeq 
H^{j}\TO{q'}_{*}((p')^{*}(\TO{u}^{!}\scrM)\otimes \calE^{-\varphi/\lam})\\
&\simeq 
H^j\TO{q'}_{*}(\TO(p')^{!}(\TO{u}^{!}\scrM)\otimes\TO{(u\times u^{\vee})}^{!} \calE^{-\varphi/\lam})(-d_{f}-n_{E})[-d_{f}-n_{E}]\\
&\simeq 
H^j\TO{q'}_{*}(\TO{(u\times u^{\vee})}^{!}(\TO{p}^{!}\scrM)\otimes\TO{(u\times u^{\vee})}^{!} \calE^{-\varphi/\lam})(-d_{f}-n_{E})[-d_{f}-n_{E}]\\
&\simeq 
H^j\TO{q'}_{*}(\TO{(u\times u^{\vee})}^{!}((\TO{p}^{!}\scrM)\otimes  \calE^{-\varphi/\lam}))(-n_{E})[-n_{E}]\\
&\simeq H^j\TO{(u^{\vee})}^{!}(\TO{q}_{*}((\TO{p}^{!}\scrM)\otimes  \calE^{-\varphi/\lam}))(-n_{E})[-n_{E}]\\
&\simeq  H^j\TO{(u^{\vee})}^{!}(\TO{q}_{*}(({p}^{*}\scrM)\otimes  \calE^{-\varphi/\lam}))\\
&=  H^j\TO{(u^{\vee})}^{!}\whmunk{\scrM},
\end{align*}
where the second isomorphism follows from the exactness of Fourier-Laplace transform,
the 4-th isomorphism follows from Lemma~\ref{koukyou} and
the 6-th isomorphism follows from the base change: Proposition~14.3.27 of \cite{MTM}.

\end{proof}

\begin{lem}\label{projform}
Let $f\colon X\to Y$ be a morphism between smooth algebraic variety $X$ and $Y$ and $\scrM$ (resp. $\mathscr{L}$) be the underlying $R^{\integ}$-module of an integrable mixed twistor $D$-module (resp. a smooth integrable mixed twistor $D$-module, i.e. an admissible variation of mixed twistor structure) on $X$ (resp. $Y$).
Then, we have the following isomorphism in the derived category of $R^{\integ}$-modules:
\[\TO{f}_{*}(\scrM\otimes f^{*}\mathscr{L})\simeq
\TO{f}_{*}\scrM\otimes \mathscr{L}.\]
\end{lem}

\begin{proof}
Take a smooth variety $\ov{X}$ containing $X$ such that $H_{X}:=\ov{X}\setminus X$ is a divisor in $\ov{X}$,
and a proper morphism $\ov{f}\colon \ov{X}\to Y$ which induces $f\colon X\to Y$.
Moreover, take the underlying $R^{\integ}$-module $\wtkai{\scrM}$ of a mixed twistor $D$-module on $\ov{X}$ whose restriction $\wtkai{\scrM}|_{X}$ is $\scrM$.
Then, we have
\begin{align*}
    \TO{f}_{*}(\scrM\otimes f^{*}\mathscr{L})
    =&\TO{f}_{*}((\wtkai{\scrM}\otimes \ov{f}^{*}{\mathscr{L}})[*H_{X}]).
\end{align*}
Let $\Delta_{\ov{X}}\colon \ov{X}\hookrightarrow \ov{X}\times \ov{X}$
and $\Delta_{Y}\colon Y\hookrightarrow Y\times Y$ be the diagonal embedding.
Then, by Proposition~4.58 of \cite{MochiResc},
we have
\begin{align*}
    (\wtkai{\scrM}\otimes \ov{f}^*{\mathscr{L}})[*H_{X}]
    \simeq \TO{\Delta_{\ov{X}}^{!}}(\wtkai{\scrM}[*H_{X}]\boxtimes \ov{f}^*{\mathscr{L}})(\dim{X})[\dim{X}].
\end{align*}
Therefore, we have
\begin{align}
\TO{f}_{*}(\scrM\otimes f^{*}\mathscr{L})
    \simeq &\TO{f}_{*}\TO{\Delta_{\ov{X}}^{!}}(\wtkai{\scrM}[*H_{X}]\boxtimes \ov{f}^*{\mathscr{L}})(\dim{X})[\dim{X}]\notag\\
    \simeq &\TO{\Delta_{Y}^{!}}\TO{(\ov{f}\times \ov{f})}_{*}(\wtkai{\scrM}[*H_{X}]\boxtimes \ov{f}^*{\mathscr{L}})(\dim{X})[\dim{X}]\notag\\
    \simeq &\TO{\Delta_{Y}^{!}}{(\ov{f}\times \ov{f})}_{\dag}(\wtkai{\scrM}[*H_{X}]\boxtimes \ov{f}^*{\mathscr{L}})(\dim{X})[\dim{X}].\label{pinknohuku}
\end{align}
where we used the base change for the second isomorphism.
Let $p_{\ov{X},i}$ (resp. $p_{Y,i}$) be the $i$-th projection ($i=1,2$) of $\ov{X}\times \ov{X}$ (resp. $Y\times Y$).
Then, we have
\begin{align}
    {(\ov{f}\times \ov{f})}_{\dag}(\wtkai{\scrM}[*H_{X}]\boxtimes \ov{f}^*{\mathscr{L}})
    &={(\ov{f}\times \ov{f})}_{\dag}(p_{\ov{X},1}^*\wtkai{\scrM}[*H_{X}]\otimes p_{\ov{X},2}^*\ov{f}^*{\mathscr{L}})\notag\\
    &={(\ov{f}\times \ov{f})}_{\dag}(p_{\ov{X},1}^*\wtkai{\scrM}[*H_{X}]\otimes (\ov{f}\times \ov{f})^*p_{Y,2}^*{\mathscr{L}})\notag\\
    &=
    {(\ov{f}\times \ov{f})}_{\dag}(p_{\ov{X},1}^*\wtkai{\scrM}[*H_{X}])\otimes p_{Y,2}^*{\mathscr{L}},\label{reiten}
\end{align}
where the final isomorphism follows from the projection formula.
Moreover, 
since $\TO{p_{\ov{X},1}^!}\simeq {p_{\ov{X},1}^*}(\dim{X})[\dim{X}]$,
we have
\begin{align*}
    {(\ov{f}\times \ov{f})}_{\dag}(p_{\ov{X},1}^*(\wtkai{\scrM}[*H_{X}]))
    &\simeq 
\TO{(\ov{f}\times \ov{f})}_{*}\TO{p_{\ov{X},1}^!}(\wtkai{\scrM}[*H_{X}])(-\dim{X})[-\dim{X}]\\
&\simeq 
\TO{(\ov{f}\times \ov{f})}_{*}\TO{p_{\ov{X},1}^!}(\wtkai{\scrM}[*H_{X}])(-\dim{X})[-\dim{X}]\\
&\simeq 
\TO{p_{Y,1}^{!}}\TO{\ov{f}}_{*}(\wtkai{\scrM}[*H_{X}])(-\dim{X})[-\dim{X}]\\
&\simeq 
{p_{Y,1}^{*}}\TO{f}_{*}(\scrM)(d_{f})[d_{f}],
\end{align*}
where we used the base change formula for the third isomorphism.
Combining it with (\ref{pinknohuku}) and (\ref{reiten}),
we obtain
\begin{align*}
    \TO{f}_{*}(\scrM\otimes f^{*}\mathscr{L})\simeq&
    \TO{\Delta_{Y}^{!}}(\TO{f}_{*}(\scrM)\boxtimes \mathscr{L})(\dim{Y})[\dim{Y}]\\
    \simeq &\TO{f}_{*}(\scrM)\otimes \mathscr{L},
\end{align*}
where the last isomorphism follows from (\ref{huyunookurimono}) (or Proposition~4.58 of \cite{MochiResc}).
\end{proof}

\begin{lem}\label{abare}
Let $F$ be another vector bundle over $X$, $f\colon E\to F$ a morphism and $n_{E}$ (resp. $n_{F}$) the rank of the vector bundle $E$ (resp. $F$).
We denote by ${}^tf\colon F^{\vee}\to E^{\vee}$ its transpose morphism.
For the underlying $R^{\integ}$-module $\scrM$ of an integrable mixed twistor $D$-module on $E$,
we have a natural isomorphism in the category of $R^{\integ}$-modules
\[\whmu{H^{j}\TO{f}_{*}\scrM}\simeq H^{j+n_{f}}(\TO{({}^tf)}^{!}\whmunk{\scrM})(n_{f}),\]
where we put $n_{f}:=n_{E}-n_{F}$.
\end{lem}

\begin{proof}
We consider the following diagram
\[
\xymatrix{
E\ar[dd]^-{f}&E\times_{X} E^{\vee}
\ar[r]^-{q}\ar[l]_-{p}
&E^{\vee}\\
&E\times_{X}F^{\vee}\ar[d]^-{f\times 1}\ar[rd]^-{q''}\ar[lu]^-{p''}\ar[u]^-{1\times {}^tf}& \\
F&F\times_{X} F^{\vee}\ar[r]^-{q'}\ar[l]_-{p'}&F^{\vee}\ar[uu]^-{{}^tf}.
}
\]

Then, in the same way as the argument in the proof of Lemma~\ref{kaitaku},
we have
\begin{align}
    \whmu{H^{j}\TO{f}_{*}\scrM}\simeq &
    H^j\TO{q'}_{*}((p')^*\TO{f}_{*}\scrM\otimes \calE^{-\varphi/\lam})\notag\\
    \simeq &
    H^j\TO{q'}_{*}(\TO(p')^{!}\TO{f}_{*}\scrM\otimes \calE^{-\varphi/\lam})(-n_{F})[-n_{F}]\notag\\
    \simeq &    
    H^j\TO{q'}_{*}(\TO{(f\times 1)}_{*}\TO{(p'')}^{!}\scrM\otimes \calE^{-\varphi/\lam})(-n_{F})[-n_{F}]\notag\\
    \simeq &
    H^j\TO{q'}_{*}\TO{(f\times 1)}_{*}(\TO{(p'')}^{!}\scrM\otimes (f\times 1)^*\calE^{-\varphi/\lam})(-n_{F})[-n_{F}]\notag\\
    \simeq &
H^j\TO{q''}_{*}(\TO{(p'')}^{!}\scrM\otimes (f\times 1)^*\calE^{-\varphi/\lam})(-n_{F})[-n_{F}]\label{odayakani}
\end{align}
where we used the base change formula for the third isomorphism
and Lemma~\ref{projform} for the 4-th isomorphism.
Since $1\times {}^tf$ is non-characteristic with respect to  $\calE^{-\varphi/\lam}_{E\times_{X}E^{\vee}}$,
we have
\begin{align*}
(f\times 1)^{*}\calE^{-\varphi/\lam}_{F\times_{X}F^{\vee}}\simeq& 
(1\times {}^tf)^{*}\calE^{-\varphi/\lam}_{E\times_{X}E^{\vee}}\\
\simeq& 
\TO(1\times {}^tf)^{!}\calE^{-\varphi/\lam}_{E\times_{X}E^{\vee}}
(n_{f})[n_{f}].
\end{align*}
Therefore, we have
\begin{align}
    \TO{(p'')}^{!}\scrM\otimes (f\times 1)^*\calE^{-\varphi/\lam}&\simeq 
    \TO{(1\times {}^tf)}^{!}\TO{p}^{!}\scrM\otimes 
    \TO(1\times {}^tf)^{!}\calE^{-\varphi/\lam}(n_{f})[n_{f}]\notag\\
    &\simeq 
    \TO{(1\times {}^tf)}^{!}(\TO{p}^{!}\scrM\otimes \calE^{-\varphi/\lam})(-n_{f})[-n_{f}]
    (n_{f})[n_{f}]\notag\\
    &\simeq     \TO{(1\times {}^tf)}^{!}(\TO{p}^{!}\scrM\otimes \calE^{-\varphi/\lam}),\notag
\end{align}
where we used Lemma~\ref{koukyou} for the second isomorphism.
Combining it with (\ref{odayakani}),
we have
\begin{align*}
     \whmu{H^{j}\TO{f}_{*}\scrM}\simeq& 
     H^j\TO{q''}_{*}
    \TO{(1\times {}^tf)}^{!}(\TO{p}^{!}\scrM\otimes \calE^{-\varphi/\lam})(-n_{F})[-n_{F}]\\
    \simeq &
    H^j\TO{({}^tf)}^!\TO{q}_{*}(\TO{p}^{!}\scrM\otimes \calE^{-\varphi/\lam})(-n_{F})[-n_{F}]\\
    \simeq &
    H^j\TO{({}^tf)}^!\TO{q}_{*}({p}^{*}\scrM\otimes \calE^{-\varphi/\lam})(n_{f})[n_{f}]\\
    \simeq &
    H^{j+n_{f}}\TO{({}^tf)}^!\whmunk{\scrM}(n_{f}),
\end{align*}
where we used the base change formula for the second isomorphism.
This completes the proof.
\end{proof}

For a vector bundle $E$ over $X$, we use the morphisms $\omega\colon E\times_{X}E\to \CC\times E$, $p\colon E\times_{X}E^{\vee}\to E$ and $\iota\colon E^{\vee}\hookrightarrow \CC^{\vee}\times E^{\vee}$ defined in Section~\ref{maboo} just before Lemma~\ref{scaret}.
We define the terminology: ``monodromic $R$-modules'' in a similar way to Definition~\ref{mondef}.

\begin{lem}\label{moucyotto}
For the underlying $R^{\integ}$-module $\scrM$ of an integrable mixed twistor $D$-module on $E$,
assume that $\scrM$ is monodromic.
Then, we have an isomorphism:
\begin{align}\label{omensoflove}
\whmunk{\scrM}\simeq H^{1}\TO{\iota}^!\whmu{H^{0}\TO{\omega}_{*}p^*\scrM}(1).
\end{align}
\end{lem}

\begin{proof}
Since $p$ is a projection,
we have
\[\TO{p^{!}}\scrM\simeq p^*\scrM(\dim{X})[\dim{X}].\]
Moreover, since $\whmu{H^{j}{\omega}_{\dag}p^*M}$ is monodromic by Lemma~\ref{borero},
$\iota$ is non-characteristic with respect to $\whmu{H^{j}\TO{\omega}_{*}p^*\scrM}$.
Therefore, we have
\[\TO{\iota}^!\whmu{H^{0}\TO{\omega}_{*}p^*\scrM}\simeq \iota^{*}\whmu{H^{0}\TO{\omega}_{*}p^*\scrM}(-1)[-1].\]
Therefore, we have
\begin{align*}
 H^{1}\TO{\iota}^!\whmu{H^{0}\TO{\omega}_{*}p^*\scrM}
 \simeq &
 H^{1}\TO{\iota}^!(H^{n_{E}-1}\TO{({}^t\omega)}^{!}\whmu{p^*\scrM})(n_{E}-1) \quad (\mbox{by Lemma~\ref{abare}})\\
 \simeq &
 H^{1}\TO{\iota}^!(H^{n_{E}-1}\TO{({}^t\omega)}^{!}\whmu{H^{-n_{E}}\TO{p}^!\scrM})(-n_{E})(n_{E}-1)\\
 \simeq &
H^{1}\TO{\iota}^!(H^{n_{E}-1}\TO{({}^t\omega)}^{!}(H^{-n_{E}}\TO{p^{\vee}}^!\whmunk{\scrM}))(-1)\quad (\mbox{by Lemma~\ref{kaitaku}})\\
\simeq & 
H^{0}\TO(p^{\vee}\circ {}^t\omega\circ \iota)^{!}\whmunk{\scrM}(-1)\\
\simeq &
\whmunk{\scrM}(-1).
\end{align*}

\end{proof}

Let us consider the irregular Hodge filtration of the right hand side of (\ref{omensoflove}).
Remark that if an $R$-module $\scrM_{1}$ on $\CC^{\vee}\times E^{\vee}$ is monodromic with respect to $\CC^{\vee}$-direction, $\iota$ is non-characteristic with respect to $\scrM_{1}$.
In particular, we have $H^{1}\TO{\iota}^{!}\scrM_{1}\simeq \iota^*\scrM_{1}(-1)$.

\begin{lem}\label{tanktop}
Let $\scrM_{1}$ be the underlying $R^{\integ}$-module of an irregular Hodge module on $\CC^{\vee}\times E^{\vee}$
and its underlying $D$-module is denoted by $M_{1}$.
Assume that 
$\scrM_{1}$ (resp. $\RES{\scrM_{1}}$) is monodromic with respect to the $\CC^{\vee}$-direction (resp. $\CC_{\tau}$-direction).
Moreover, assume that $H^{1}\TO{\iota}^{!}\scrM_{1}(=\iota^{*}\scrM_{1}(-1))$ is the underlying $R^{\integ}$-module of an irregular Hodge module on $E^{\vee}$.
Then, for $\alpha\in [0,1)$, we have
\begin{align}
    \label{swing}
\irrF_{\alpha+\bullet}\iota^*M_{1}=\iota^*\irrF_{\alpha+\bullet+1}M_{1},
\end{align}
where we regard $\iota^*M_{1}$ as the underlying $D$-module of $H^{1}\TO{\iota}^{!}\scrM_{1}$.
\end{lem}

\begin{proof}
Consider the rescaling $\RES{(\iota^{*}\scrM_{1})}$ (resp. $\RES{\scrM_{1}}$), which is an object on $\CC_{\tau}\times E^{\vee}$ (resp. $\CC_{\tau}\times \CC^{\vee}\times E^{\vee}$).
We denote by the same symbol $\iota$ the inclusion $\CC_{\tau}\times E^{\vee}\hookrightarrow \CC_{\tau}\times \CC^{\vee}\times E^{\vee}$.
Then, by the definition, it is obvious that we have
\[\RES{(\iota^{*}\scrM_{1})}\simeq \iota^{*}\RES{\scrM_{1}}.\]
Note that for $\gamma\in \RR$ we have
\begin{align}\label{subetega}
    V^{\gamma}_{\tau}(\RES{\scrM_{1}})=\bigoplus_{\beta\geq \gamma}(\RES{\scrM_{1}})^{\gamma},
\end{align}
where $(\RES{\scrM_{1}})^{\gamma}=\bigoplus_{k\geq 0}\Ker(\tau\eth_{\tau}-\gamma)^{k}\subset \RES{\scrM_{1}}$.
Moreover,
it is clear that for $\beta\in \RR$ we have 
\[(\iota^*\RES{M_{1}})^{\beta}=\iota^*(\RES{M_{1}})^{\beta}.\]
Therefore, we have
\begin{align*}
V^{\gamma}_{\tau}\iota^*\RES{\scrM_{1}}=&\bigoplus_{\beta\geq \gamma}\iota^*(\RES{\scrM_{1}})^{\beta}\\
=&\iota^*V^{\gamma}_{\tau}(\RES{\scrM_{1}})\qquad \mbox{(by (\ref{subetega}))}.
\end{align*}
Hence, we obtain
\begin{align*}
    i_{\tau=\lam}^*V^{\beta}_{\tau}(\iota^*\RES{\scrM_{1}}(-1))
    =\lam^{-1}\iota^*(i_{\tau=\lam}^*V^{\beta}_{\tau}(\RES{\scrM_{1}}))),
\end{align*}
in $\iota^*M_{1}[\lam^{\pm 1}]$.
This equality means the equality (\ref{swing}).
\end{proof}

\begin{rem}
In \cite{MochiResc} (see Theorem~1.5 in loc. cit.), Lemma~\ref{tanktop} and some stronger results are shown in a more general situation.
For example, $H^{1}\TO\iota^{!}\scrN$ is always an irregular Hodge module. 
But, we do not need it here.
\end{rem}

\begin{proof}[Proof of Theorem~\ref{takarajima}]
By Lemma~\ref{moucyotto},
we have
\begin{align*}
    \irrF_{p}\whmunk{M}\simeq \irrF_{p-1}H^{1}\iota^\dag\whmu{H^{0}\omega_{\dag}H^{-n}p^{\dag}M}.
\end{align*}
By Lemma~\ref{tanktop},
the RHS is equal to
\begin{align}
\label{otukare}    
\iota^*\irrF_{p}\whmu{H^{0}\omega_{\dag}H^{-n}p^{\dag}M}.
\end{align}
By (\ref{dance}) and Theorem~\ref{prosyuto},
the Hodge filtration (defined by Lemma~\ref{soutou}) of the Fourier-Laplace transform of a monodromic mixed Hodge module on a line bundle coincides with the irregular Hodge filtration (for $\alpha=0$).
Therefore, for $p\in \ZZ$ we have
\begin{align}\label{senriyukumono}
    \irrF_{p}\whmu{H^{0}\omega_{\dag}H^{-n}p^{\dag}M}=F_{p}\whmu{H^{0}\omega_{\dag}H^{-n}p^{\dag}M},
    \end{align}
where the RHS is the Hodge filtration defined by Lemma~\ref{soutou}.
Hence, (\ref{otukare}) is equal to
\[\iota^*F_{p}\whmu{H^{0}\omega_{\dag}H^{-n}p^{\dag}M}.\]
On the other hand, by definition~\ref{ganbaruzo},
we have
\[F_{p}\whmunk{M}=F_{p-1}H^{1}\iota^\dag\whmu{H^{0}\omega_{\dag}H^{-n}p^{\dag}M}.\]
By the definition of the pullback functor $H^{1}\iota^{\dag}$ (between the category of mixed Hodge modules),
we have
\begin{align*}
    F_{p-1}H^{1}\iota^{\dag}\whmu{H^{0}\omega_{\dag}H^{-n}p^{\dag}M}
    \simeq \iota^*F_{p}\whmu{H^{0}\omega_{\dag}H^{-n}p^{\dag}M}.
\end{align*}
Combining these equality together,
we obtain
\[F_{p}\whmunk{M}=\irrF_{p}\whmunk{M}.\]
\end{proof}

Finally, we discuss the relationship between the irregular Hodge filtration and the Hodge filtration of $\whmunk{M}$ ``at infinity''.
Let $\whmunk{\calM}$ be the mixed Hodge module defined in Definition~\ref{ganbaruzo}.
Moreover, let
$\wt{\whmunk{\calM}}$ be the mixed Hodge module which is the unique extension of $\whmunk{\calM}$ to $\PP^n_{\xi}$ such that $\wt{\whmunk{\calM}}=\wt{\whmunk{\calM}}[*D^{\vee}_{\infty}]$.
We denote by $\wt{\whmunk{M}}$ the underlying $D$-module.
Then, we have (by Lemma~\ref{train})
\[\wt{\whmunk{M}}=\whmunk{N},\]
where $N$ is the one defined in the first part of Subsection~\ref{otukaresamadesu}.
By Theorem~\ref{takarajima} and Corollary~\ref{nonnonbaa}, we have the following.

\begin{cor}\label{rattlechain}
We have
\[\irrF_{p}\whmunk{N}=F_{p}\wt{\whmunk{M}},\]
for any $p\in \ZZ$.
In particular, the irregular Hodge filtration $\{\irrF_{p}\whmunk{N}\}_{p\in \ZZ}$ (for $\alpha=0$) is the Hodge filtration of a mixed Hodge module.
\end{cor}

\footnotesize
\bibliographystyle{alpha}
\bibliography{reference}
\Addresses

\end{document}